\documentclass[letterpaper, 10pt,reqno]{amsart}

\usepackage[usenames,dvipsnames]{xcolor}

\usepackage[portrait,margin=3cm]{geometry}
\usepackage{mathrsfs}

\usepackage{amssymb,amsthm,amsfonts,amsbsy,amstext,latexsym}
\usepackage[reqno]{amsmath}

\usepackage{color}
\usepackage{graphicx}
\usepackage{bbm}
\usepackage{comment}
\usepackage{a4wide,graphicx, listings,lscape, epstopdf, units, textcomp, fancyhdr, url, soul, color}
\usepackage{tikz}
\usetikzlibrary{shapes.geometric}
\usepackage[textsize=small]{todonotes}
\usepackage{enumitem}

\usepackage{mathtools}
\mathtoolsset{showonlyrefs}

\definecolor{MyDarkblue}{rgb}{0,0.08,0.50}
\definecolor{Brickred}{rgb}{0.65,0.08,0}

\usepackage{hyperref,upref}
\hypersetup{
	colorlinks=true,       % false: boxed links; true: colored links
	linkcolor=MyDarkblue,          % color of internal links
	citecolor=Brickred,        % color of links to bibliography
	filecolor=red,      % color of file links
	urlcolor=cyan           % color of external links
}

\newtheorem*{theorem*}{Theorem}
\newtheorem{theorem}{Theorem}[section]
\newtheorem{lemma}[theorem]{Lemma}
\newtheorem{proposition}[theorem]{Proposition}
\newtheorem{corollary}[theorem]{Corollary}

\theoremstyle{definition}
\newtheorem{definition}[theorem]{Definition}

\newtheorem{remark}[theorem]{Remark}

\renewcommand{\P}{\mathbb{P}}
\newcommand{\Pv}{\mathbb{P}}

\newcommand{\CC}{\mathcal{C}}

\newcommand{\eps}{\varepsilon}

\newcommand{\Var}{{\rm Var}}

\newcommand{\CE}{{\mathcal{E}}}

\newcommand{\e}{{\mathrm e}}

\newcommand{\R}{\mathbb{R}}
\newcommand{\N}{\mathbb{N}}
\newcommand{\Z}{\mathbb{Z}}

\renewcommand{\emptyset}{\varnothing}

\newcommand{\CA}{\mathcal {A}}
\newcommand{\CB}{\mathcal {B}}
\newcommand{\CD}{\mathcal {D}}
\newcommand{\CF}{\mathcal {F}}

\newcommand{\CI}{\mathcal {I}}

\newcommand{\CL}{\mathcal {L}}
\newcommand{\CM}{\mathcal {M}}
\newcommand{\CN}{\mathcal {N}}
\newcommand{\CO}{\mathcal {O}}
\newcommand{\CP}{\mathcal {P}}

\newcommand{\CR}{\mathcal {R}}
\newcommand{\CS}{\mathcal {S}}

\newcommand{\CX}{\mathcal {X}}

\newcommand{\CH}{\mathcal {H}}

\newcommand{\dist}{{\rm dist}}

\newcommand*{\wt}{\widetilde}

\newcommand*{\be}{\begin{equation}}
	\newcommand*{\ee}{\end{equation}}
\newcommand*{\ba}{\begin{aligned}}
	\newcommand*{\ea}{\end{aligned}}
\newcommand*{\barr}{\begin{array}{c}}
	\newcommand*{\earr}{\end{array}}
\def \toinp    {\buildrel {\Pv}\over{\longrightarrow}}
\def \toindis  {\buildrel {d}\over{\longrightarrow}}
\def \toas     {\buildrel {a.s.}\over{\longrightarrow}}

\newcommand*{\ind}{\mathbbm{1}}
\makeatletter
\def\namedlabel#1#2{\begingroup
	#2%
	\def\@currentlabel{#2}%
	\phantomsection\label{#1}\endgroup
}

\makeatother
\newcommand{\bes}{\begin{equation*}}
	\newcommand{\ees}{\end{equation*}}

\renewcommand{\P}[1]{\mathbb{P}\!\left(#1\right)}
\newcommand{\E}[1]{\mathbb{E}\left[#1\right]}

\renewcommand{\N}{\mathbb{N}}

\renewcommand{\d}{\mathrm{d}}

\renewcommand{\S}{\mathcal S}

\newcommand{\zni}{d_{T_n}(i)}

\newcommand{\inn}{i\in[n]}

\numberwithin{equation}{section}

\renewcommand{\e}{\mathrm{e}}

\renewcommand{\deg}{d}
\newcommand{\A}{\mathcal A_{\bar d}}
\newcommand{\B}{\mathcal B_{n,\delta}}
\newcommand{\wtA}{\wt{\mathcal A}_{\overline \ell}}
\newcommand{\wtB}{\wt{\mathcal B}_n}

\makeatletter
\newcommand{\leqnomode}{\tagsleft@true\let\veqno\@@leqno}
\newcommand{\reqnomode}{\tagsleft@false\let\veqno\@@eqno}
\makeatother
\newlength{\tagmarginsep} % Distance required
\tikzstyle{vertex}=[circle,fill=orange!60,minimum size=10pt,inner sep=0pt]
\tikzstyle{tedge} = [draw,ultra thick,->,>=stealth, orange]
\tikzstyle{esq}=[circle,fill=white,minimum size=10pt,inner sep=0pt]
\tikzstyle{up}=[<-,>=stealth]

\begin{document}
	
	\title[Joint properties of vertices with a given degree or label in the RRT]{On joint properties of vertices with a given degree or label in the random recursive tree}

	\date{\today}
	
	\author[Lodewijks]{Bas Lodewijks}
	\address{Institut Camille Jordan, Univ.\ Lyon 1, Lyon, France, Univ.\ Jean Monnet, Saint-\'Etienne, France}
	\email{bas.lodewijks@univ-st-etienne.fr}

	\begin{abstract}
		In this paper, we study the joint behaviour of the degree, depth and label of and graph distance between high-degree vertices in the random recursive tree. We generalise the results obtained by Eslava~\cite{Esl16} and extend these to include the labels of and graph distance between high-degree vertices. The analysis of both these two properties of high-degree vertices is novel, in particular in relation to the behaviour of the depth of such vertices.\\
		In passing, we also obtain results for the joint behaviour of the degree and depth of and graph distance between any fixed number of vertices with a prescribed label. This combines several isolated results on the degree~\cite{KubPan07}, depth~\cite{Dev98,Mah91} and graph distance~\cite{Dob96,FengSuLiu06} of vertices with a prescribed label already present in the literature. Furthermore, we extend these results to hold jointly for any number of fixed vertices and improve these results by providing more detailed descriptions of the distributional limits.\\
		Our analysis is based on a correspondence between the random recursive tree and a representation of the Kingman $n$-coalescent.
	\end{abstract}
	
	\maketitle 
	
	\section{Introduction}\label{sec:intr}
	
	The random recursive tree model has, since its introduction by Na and Rapoport~\cite{NaRap70}, received a wealth of interest and many properties have been studied. This wide range of topics includes, among others, the degree distribution~\cite{Jan05,MahSmy92,MeirMoon88}, the degree of vertices with a prescribed label~\cite{Dev98,KubPan07}, the maximum degree~\cite{AddEsl18,BanBha20,DevLu95,GohSch02,Szy90}, the height of the tree~\cite{Pit94}, the insertion depth of the tree~\cite{Dev98,Mah91}, and the graph distance between vertices~\cite{Dob96,FengSuLiu06}. Beyond these statistics, real-world applications of random recursive trees have been considered as well~\cite{GasWir84,Moon74,NaHey82}. See also~\cite{Drm09,MahLue92} for two surveys on random trees that include a more extensive overview of the research literature on random recursive trees.
	
	Different approaches for studying the random recursive tree model have been considered throughout the literature. Using the recursive definition of the model and the fact that the random recursive tree with $n$ vertices is defined to be a uniform tree among all increasing trees with $n$ vertices (labelled trees where the vertices on a path from the root to any vertex have increasing labels) are among the most prevalent. Other methods include using continuous-time embedding in Crump-Mode-Jagers branching processes, first introduced by Athreya and Karlin for P\'olya urns in~\cite{AthrKar68} and later used for a wide range of recursive tree models such as the random recursive tree (see e.g.~\cite{Bham07,Iyer20,Jan04,Pit94}), P\'olya urns~\cite{Jan05} and a representation of Kingman's coalescent~\cite{AddEsl18,Esl16,Pit94}.
	
	In most studies found in the literature regarding the random recursive tree model, statistics like those mentioned above are considered in \emph{isolation}, rather then studying their \emph{joint behaviour}. As far as the author is aware, only a handful of papers consider the joint behaviour of different statistics for the random recursive tree. In~\cite{Esl16}, Eslava studies the depth of high-degree vertices, Banerjee and Bhamidi study the label size of the vertex attaining the maximum degree in~\cite{BanBha20}, and the author studies the labels of high-degree vertices in the more general weighted recursive tree model~\cite{Lod21}, of which the random recursive tree model is a particular example.
	
	The aim of this paper is to extend what is known about the joint behaviour of several statistics of the random recursive tree. We consider, in particular, two settings. First, we study the joint behaviour of the depth and label of and graph distance between any fixed number of vertices selected uniformly at random, conditionally on having a degree that exceeds a certain quantity. We combine, extend, improve and recover the results of the author~\cite{Lod21} (in the particular case of the random recursive tree) and Eslava~\cite{Esl16}. We also recover the results of Addario-Berry and Eslava~\cite{AddEsl18} and Eslava, the author, and Ortgiese~\cite{EslLodOrt21} (again, in the particular case of the random recursive tree). 
	
	Let $T_n$ denote the random recursive tree with $n$ vertices. Eslava considers in~\cite{Esl16} the vector $
	(d^i_n-\lfloor \log_2n\rfloor,(h^i_n-\mu \log n)/\sqrt{\sigma^2\log n})_{\inn}$,
	where $d^i_n$ and $h^i_n$ denote the degree and depth of the vertex with the $i^{\text{th}}$ largest degree (ties broken uniformly at random), respectively, and sets $\mu:=1-1/(2\log 2), \sigma^2:=1-1/(4\log 2)$. Eslava shows this vector converges in distribution along suitable subsequences $(n_t)_{t\in\N}$ to a marked point process on $(\Z\cup\{\infty\}) \times \R$, where the marks are independent standard normal random variables. The author proves a similar result for the vector 
	\be 
	(d^i_n-\lfloor \log_2n\rfloor,(\ell^i_n-\mu \log n)/\sqrt{(1-\sigma^2)\log n})_{\inn}
	\ee 
	in~\cite{Lod21}, where $\ell^i_n$ denotes the label of the vertex with degree $d^i_n$ (ties broken uniformly at random). Again, along suitable subsequences, this vector converges in distribution to a marked point process on $(\Z\cup \{\infty\})\times \R$, where the marks are independent standard normal random variables. Our results here combine these results to show that the vector 
	\be 
	(d^i_n-\lfloor \log_2n\rfloor,(h^i_n-\mu \log n)/\sqrt{\sigma^2\log n},(\ell^i_n-\mu \log n)/\sqrt{(1-\sigma^2)\log n})_{\inn}
	\ee 
	converges along suitable subsequences to a marked point process on $(\Z\cup\{\infty\})\times \R^2$, where the marks are i.i.d.\ copies of $( M\sqrt{1-\mu/\sigma^2}+N\sqrt{\mu/\sigma^2},M)$, with $M,N,$ two i.i.d.\ standard normal random variables. This recovers both results and, additionally, provides a novel and interesting dependence between the scaling limit of the depth and label of high-degree vertices. It describes exactly \emph{how large} the largest degrees in the tree are, as well as \emph{where} and \emph{when} they appear in the tree. This natural extension of the current knowledge provides a rather complete picture of the behaviour of high-degree vertices.
	
	Moreover, we also obtain the distributional convergence of the (properly rescaled) depth and label of and graph distance between any finite number of vertices selected uniformly at random, conditionally on their degrees growing infinitely large as $n\to\infty$. The graph distance between such high-degree vertices has not been studied previously, and we are, in particular, able to characterise the limiting law of the graph distance in terms of the limiting law of the depth of these vertices. 
	
	Second, we study the joint behaviour of the degree and depth of and graph distance between any fixed number of vertices with a prescribed label. This combines, extends, improves and recovers a range of results on the degree~\cite{Dev98,KubPan07} and depth~\cite{Dev98,Mah91} of and graph distance~\cite{Dob96,FengSuLiu06} between vertices with a prescribed label. Given any fixed $k\geq 2$ vertices with labels $(v_{i,n})_{i\in[k]}$ such that $v_{i,n}$ diverges with $n$, we obtain the joint distributional convergence of the degree and depth of and graph distance between vertices $v_{1,n}, \ldots, v_{k,n}$. Again, we characterise the limiting law of the graph distances in terms of those of the depths of vertices $v_{1,n},\ldots, v_{k,n}$, which is novel. 
	
	Our extensions of the aforementioned results arise mainly due to two contributions. First, we are able to analyse the joint behaviour of multiple statistics beyond what was known already in the literature. Second, we obtain these results for any finite number of vertices, whereas only a single vertex or single pair of vertices is considered in most results available to date. It is exactly the correlations that arise due to considering several statistics and many vertices at once that prove to be the most challenging aspects of the analysis. The improvement of the existing results is mostly due to the fact that considering the joint behaviour of several statistics allows us, in certain cases, to obtain more detailed descriptions of their limiting laws beyond what was known previously. 
	
	The analysis in this paper is based on the Kingman $n$-coalescent construction of the random recursive tree. This construction was first observed by Pittel in~\cite{Pit94} and later recovered and used by Addario-Berry and Eslava~\cite{AddEsl18}, and Eslava~\cite{Esl16,Esl21}. This construction provides several advantages compared to the more common recursive construction of the random recursive tree. First, rather than in the recursive construction in which distinct vertices have different arrival times (which influence their degree, depth, label, and graph distance), the coalescent construction allows for a perspective in which all vertices are exchangeable. Second, the coalescent construction enables a more natural decoupling of the statistics of distinct vertices, which provides us with tools to tackle the correlations between these statistics in a more refined manner. Finally, in particular the degree, label and depth of a vertex can be expressed in terms of random numbers of coin flips, simplifying the analysis of these statistics. The degree of a vertex equals the length of the first streak of heads, the label equals the step at which the first tails occurs and the depth equals the total number of tails thrown.  \\
		
	\noindent \textbf{Notation.}	Throughout the paper we use the following notation: we let $\N:=\{1,2,\ldots\}$ denote the natural numbers, set $\N_0:=\{0,1,\ldots\}$ and let $[t]:=\{i\in\N: i\leq t\}$ for any $t\geq 1$. For $x\in\R$, we let $\lceil x\rceil:=\inf\{n\in\Z: n\geq x\}$ and $\lfloor x\rfloor:=\sup\{n\in\Z: n\leq x\}$. For $x\in \R, k\in\N$, we let $(x)_k:=x(x-1)\cdots (x-(k-1))$ and $(x)_0:=1$ and use the notation $\bar d$ to denote a $k$-tuple $d=(d_1,\ldots, d_k)$ (the size of the tuple will be clear from the context), where the $d_1,\ldots, d_k$ are either numbers or sets. For sequences $(a_n,b_n)_{n\in\N}$ such that $b_n$ is positive for all $n$ we say that $a_n=o(b_n), a_n=\omega(b_n), a_n\sim b_n, a_n=\mathcal{O}(b_n)$ if $\lim_{n\to\infty} a_n/b_n=0,\lim_{n\to\infty} |a_n|/b_n=\infty, \lim_{n\to\infty} a_n/b_n=1$ and if there exists a constant $C>0$ such that $|a_n|\leq Cb_n$ for all $n\in\N$, respectively. For random variables $X,(X_n)_{n\in\N}$ we let $X_n\toindis X, X_n\toinp X$ and $X_n\toas X$ denote convergence in distribution, probability and almost sure convergence of $X_n$ to $X$, respectively. Also, let $\Phi:\R\to(0,1)$ denote the cumulative density function of a standard normal random variable.
	
	We also provide a table with the most important symbols used throughout the paper and their definitions, in order of appearance.\\
	
	\begin{table}[h]
		\begin{tabular}{ll}
			\hline
			\textbf{Symbol} & \textbf{Definition}\\
			\hline
			$T_n$ & Random recursive tree on $n$ vertices\\
			$\d_{T_n}(u)$ & In-degree of vertex $u$ in $T_n$\\
			$\text{dist}_{T_n}(u,v)$ & Graph distance between vertices $u,v$ in $T_n$\\
			$h_{T_n}(u)$ & Depth of vertex $u$ in $T_n$ (graph distance to the root,  $\text{dist}_{T_n}(u,1)$)\\
			$v^j$ & $j^{\text{th}}$ vertex in $T_n$, in decreasing order of in-degree\\
			$d_n^j$ & In-degree of $v^j$, $\d_{T_n}(v^j)$\\
			$h^j_n$ & Depth of $v^j$, $h_{T_n}(v^j)$\\
			$\mu$ & $1-\frac{1}{2\log 2}$\\
			$\sigma^2$ & $1-\frac{1}{4\log 2}$\\
			$(v_i)_{i\in[k]}$ & $k$ distinct vertices in $T_n$ selected uniformly at random\\
			$T^{(n)}$ & Kingman $n$-coalescent tree\\
			$\d_{T^{(n)}}(i)/\d_n(i)$ & In-degree of vertex $i$ in $T^{(n)}$ \\
			$h_{T^{(n)}}(i)/h_n(i)$ & Depth of vertex $i$ in $T^{(n)}$\\
			$\ell_{T^{(n)}}(i)/\ell_n(i)$ & Label of vertex $i$ in $T^{(n)}$ after relabelling (as in~\eqref{eq:sigmac}) \\
			$\CS_n(i)$ & Selection set of vertex $i$ in $T^{(n)}$\\
			$S_n(i)$ & $|\CS_n(i)|$\\
			$\overline{\CS_n}$ & $(\CS_n(i))_{i\in[k]}$\\
			$\tau_k$ & $\max\cup_{1\leq i<j\leq k}\big(\CS_n(i)\cap \CS_n(j)\big)$, the first coalescence of vertices $1,\ldots,k$\\
			$\CS_{n,1}(i)$ & Truncated selection set of vertex $i$ in $T^{(n)}$\\
			$\overline \CS_{n,1}$ & $(\CS_{n,1}(i))_{i\in[k]}$\\
			$\overline \CR_{n,1}$ & $(\CR_{n,1}(i))_{i\in[k]}$, where each element is an independent copy of $\CS_{n,1}(1)$\\
			$h_{n,1}(i)$ & Truncated depth of vertex $i$ in $T^{(n)}$\\
			$h_{n,2}(i)$ & $h_n(i)-h_{n,1}(i)$, the remaining depth\\
			\hline
		\end{tabular}
	\end{table} 
	
	\clearpage 
	
	\section{Definitions and main results}\label{sec:def}
	
	The random recursive tree model is defined as follows:
	
	\begin{definition}[Random recursive tree model]\label{def:rrt}
		Let $(T_n)_{n\in\N}$ be a sequence of trees. Initialise $T_1$ by a root with label $1$. For every $n\in\N$, construct $T_{n+1}$ from $T_n$ by adding a vertex with label $n+1$ to $T_n$ and connecting it by a directed edge to a vertex $v\in[n]$ which is selected uniformly at random.
	\end{definition}
	
	Due to the temporal nature of the random recursive tree model, it is natural to think of the edges as directed towards the root. Throughout, for any $n\in\N$ and $u,v\in[n]$, we write 
	\be \ba
	\d_{T_n}(u)&:=\text{in-degree of vertex $u$ in }T_n, \\
	\dist_{T_n}(u,v)&:=\text{graph distance between vertices }u,v\text{ in }T_n,\\
	h_{T_n}(u)&:=\text{depth of vertex $u$ in }T_n=\dist_{T_n}(u,1).
	\ea \ee 
	The graph distance between vertices $u$ and $v$ denotes the number of edge on the unique path between vertices. Here we do not take the direction of the edges into account. This only matters for the in-degree.
	
	Addario-Berry and Eslava study behaviour of high-degree vertices in the RRT  in~\cite{AddEsl18} and Eslava extends this to the joint convergence of the degree and depth of such high-degree vertices in~\cite{Esl16}. We further extend this joint convergence by including the rescaled label of the vertices as well in the following result.
	
	\begin{theorem}[Degree, depth and label of high-degree vertices in the RRT]\label{thrm:degdepthlabelrrt}
		Consider the random recursive tree (RRT) model as in Definition~\ref{def:rrt}. Let $v^1,v^2,\ldots, v^n$ be the vertices in the RRT in decreasing order of their in-degree $($\!where ties are split uniformly at random$)$ and let $(d_n^s,h_n^s,\ell_n^s )_{s\in[n]}$ denote their in-degree, depth, and label, respectively. Fix $\eps\in[0,1]$, define $\eps_n:=\log_2n-\lfloor \log_2n\rfloor$, and let $(n_t)_{t\in\N}$ be a positive, diverging, integer-valued sequence such that $\eps_{n_t}\to\eps$ as $t\to\infty$. Finally, let $(P_s)_{s\in\N}$ be the points of the Poisson point process $\CP$ on $\R$ with intensity measure $\lambda(\d x)=2^{-x}\log 2\,\d x$, ordered in decreasing order, let $(M_s,N_s)_{s\in\N}$ be two sequences of i.i.d.\ standard normal random variables and define $\mu:=1-1/(2\log 2)$ and $\sigma^2:=1-1/(4\log 2)$. Then, as $t\to\infty$, 
		\be \ba
		\Big({}&d_{n_t}^s-\lfloor \log_2n_t\rfloor, \frac{h_{n_t}^s-\mu\log n_t}{\sqrt{\sigma^2\log n_t}}, \frac{\log(\ell_{n_t}^s)-\mu\log n_t}{\sqrt{(1-\sigma^2)\log n_t}},s\in[n_t]\Big)\\
		&\toindis \Big(\lfloor P_s+\eps\rfloor , M_s\sqrt{1-\frac{\mu}{\sigma^2}}+N_s\sqrt{\frac{\mu}{\sigma^2}}, M_s, s\in\N\Big).
		\ea \ee 
	\end{theorem}
	
	\begin{remark}\label{rem:rrt}
		Theorem~\ref{thrm:degdepthlabelrrt} extends both~\cite[Theorem $2.6$]{Lod21} in the case of the random recursive tree, as well as~\cite[Theorem $1.2$]{Esl16} (since, for each $s\in\N$, $M_s\sqrt{1-\mu/\sigma^2}+N_s\sqrt{\mu/\sigma^2}\sim \CN(0,1)$). Moreover, it provides the relation and dependence between the depth of a high-degree vertex and its label, which only becomes apparent in the second-order scaling and the limit. 		
	\end{remark}
	
	Beyond studying the behaviour of vertices with `near-maximum' degree, we are also interested in a more general setting. Here, we select $k\in\N$ many vertices uniformly at random from $T_n$ and condition on their degree. We can then provide the following detailed results on the joint behaviour of their depths, labels and the graph distances between them. The following result  is instrumental in proving Theorem~\ref{thrm:degdepthlabelrrt} as well.
	
	\begin{theorem}\label{thrm:condconvrrt}
		Consider the random recursive tree model as in Definition~\ref{def:rrt}. Fix $k\in\N$, $(a_i)_{i\in[k]}\in[0,2)^k$ and let $(v_i)_{i\in[k]}$ be $k$ distinct vertices chosen uniformly at random from $[n]$. Let $(d_i)_{i\in[k]}$ be $k$ integer-valued sequences such that
		\be\label{eq:ai}
		\lim_{n\to\infty}\frac{d_i}{\log n}=a_i,
		\ee 
		for each $i\in[k]$. The tuple 
		\be \label{eq:depthdist}
		\bigg(\!\Big(\frac{h_{T_n}(v_i)-(\log n-d_i/2)}{\sqrt{\log n-d_i/4}}\Big)_{i\in[k]}\!,
		\Big(\frac{\mathrm{dist}_{T_n}(v_i,v_j)-(2\log n-(d_i+d_j)/2)}{\sqrt{2\log n-(d_i+d_j)/4}}\Big)_{1\leq i<j\leq k}\!\bigg),
		\ee 
		conditionally on the event $\d_{T_n}(v_i)\geq d_i$ for all $i\in[k]$, converges in distribution to 
		\be 
		\bigg((H_i)_{i\in[k]},\bigg(\frac{\sqrt{4-a_i}H_i+\sqrt{4-a_j}H_j}{\sqrt{8-(a_i+a_j)}}\bigg)_{1\leq i<j\leq k}\bigg),
		\ee 
		where the $(H_i)_{i\in[k]}$ are independent standard normal random variables. Additionally assume that for all $i\in[k]$, $d_i$ diverges as $n\to\infty$. Then, the tuple
		\be\ba\label{eq:depthlabeldist}  
		\bigg({}&\Big(\frac{h_{T_n}(v_i)-(\log n-d_i/2)}{\sqrt{\log n-d_i/4}},\frac{\log v_i-(\log n-d_i/2)}{\sqrt{d_i/4}}\Big)_{i\in[k]},\\
		&\Big(\frac{\mathrm{dist}_{T_n}(v_i,v_j)-(2\log n-(d_i+d_j)/2)}{\sqrt{2\log n-(d_i+d_j)/4}}\Big)_{1\leq i<j\leq k}\bigg),
		\ea\ee 
		conditionally on the event $\d_{T_n}(v_i)\geq d_i$ for all $i\in[k]$, converges in distribution to 
		\be\ba   
		\bigg({}&\Big(M_i\sqrt{\frac{a_i}{4-a_i}}+N_i\sqrt{1-\frac{a_i}{4-a_i}},M_i\Big)_{i\in[k]},\\
		&\Big(\frac{M_i\sqrt{a_i}+N_i\sqrt{4-2a_i}+M_j\sqrt{a_j}+N_j\sqrt{4-2a_j}}{\sqrt{8-(a_i+a_j)}}\Big)_{1\leq i<j\leq k}\bigg),
		\ea\ee 
		where the $(M_i,N_i)_{i\in[k]}$ are independent standard normal random variables.
	\end{theorem} 
	
	\begin{remark}\label{rem:degthrm}
		$(i)$ With an almost identical proof, the same results can be obtained when using the conditional event $\{\d_{T_n}(v_i)= d_i, i\in[k]\}$ rather than $\{\d_{T_n}(v_i)\geq d_i, i\in[k]\}$.
		
		$(ii)$ When $a_i=0$ for all $i\in[k]$, we obtain the behaviour of the \emph{insertion depth} of $k$ uniform vertices, as well as the graph distance between them.
		
		$(iii)$ The conditional convergence of the tuple in~\eqref{eq:depthdist} recovers, improves, and extends the result of Eslava in~\cite[Theorem $1.1$]{Esl16}. When we omit the distance between the vertices $v_i,v_j$ and set $d_i:=\lfloor a_i\log n\rfloor +b_i$ for some $a_i\in[0,2),b_i\in\Z$ for all $i\in[k]$, we obtain~\cite[Theorem $1.1$]{Esl16}. Our result allows for a greater freedom in the choice of the degrees $d_i$ rather than the parametrised setting used by Eslava. We extend Eslava's result even further by including the graph distance between any pair of vertices and, in~\eqref{eq:depthlabeldist}, by also including the label of the vertices $v_1,\ldots, v_k$. The latter also allows for a more precise description of the limiting distribution of the depth compared to~\cite[Theorem $1.1$]{Esl16}. We observe that the scaling of the graph distance suggests that the graph distance between vertices $v_i$ and $v_j$, for any distinct $i,j\in[k]$, is the sum of their depths. Though this sum is a trivial upper bound, we show that it is of the correct order by using the fact that the largest common ancestor of $v_i$ and $v_j$, $\mathrm{LCA}_{i,j}$, forms a tight sequence of random variables (in $n\in\N$). 
	\end{remark}
	
	Next to conditioning on the degree of vertices selected uniformly at random, we also have the following result on the degree and depth of and graph distance between vertices with a fixed label. Though the marginal convergence of the degree and depth of a vertices and graph distance of a pair of vertices with a fixed label has been studied previously (see~\cite{KubPan07,Dev98,Mah91,Dob96,FengSuLiu06}), we combine, extend, and improve these results by considering the joint convergence and by allowing for any number of (pairs of) vertices.
	
	\begin{theorem}\label{thrm:fixedlabel}
		Consider the random recursive tree model as in Definition~\ref{def:rrt}. Fix $k\in\N$ and let $(v_{i,n})_{i\in[k]}\in[n]^k$ be $k$ distinct integer-valued sequences such that $v_{i,n}$ increases with $n$, diverges as $n\to\infty$ and such that 
		\be 
		c_{i,j}:=\lim_{n\to\infty}\sqrt{\frac{\log v_{i,n}}{\log v_{i,n}+\log v_{j,n}}}
		\ee 
		exists for all $1\leq i<j\leq k$. Let $(N_i)_{i\in[k]}$ be $k$ independent standard normal random variables. We also define for each $i\in[k]$,
		\be 
		\d_{T_n}^*(v_{i,n}):=\begin{cases}
			\frac{\d_{T_n}(v_{i,n})-\log(n/v_{i,n})}{\sqrt{\log(n/v_{i,n})}} &\mbox{if } v_{i,n}=o(n),\\
			\d_{T_n}(v_{i,n}), &\mbox{otherwise,}
		\end{cases},
		\ee 
		and let $(Z_i)_{i\in[k]}$ be $k$ independent random variables $($\!also independent of $(N_i)_{i\in[k]})$ such that, for $(\rho_i)_{i\in[k]}\in(0,1)^k$,
		\be
		Z_i\sim \begin{cases}
			\mathcal N(0,1) &\mbox{if } v_{i,n}=o(n),\\
			\mathrm{Poi}(\log(1/\rho_i)) & \mbox{if }v_{i,n}=(1+o(1))\rho_i n,\\
			0 & \mbox{if } v_{i,n}=n-o(n).
		\end{cases}
		\ee 
		Then,
		\be \ba\label{eq:degdepthdist}
		\bigg({}&\Big(\d_{T_n}^*(v_{i,n}),\frac{h_{T_n}(v_{i,n})-\log v_{i,n}}{\sqrt{\log v_{i,n}}}\Big)_{i\in[k]}, \Big(\frac{\mathrm{dist}_{T_n}(i,j)-(\log v_{i,n}+\log v_{j,n})}{\sqrt{\log v_{i,n}+\log v_{j,n}}}\Big)_{1\leq i<j\leq k}\bigg)\\
		&\toindis \big((Z_i, N_i)_{i\in[k]},(c_{i,j} N_i+c_{j,i} N_j)_{1\leq i<j\leq k}\big).
		\ea \ee 	
	\end{theorem}
	
	\begin{remark}
		$(i)$ The theorem partially recovers a result from Feng, Lui, and Su~\cite[Theorem $1$]{FengSuLiu06}, where the distance between vertices $i_n$ and $n$ for any integer sequence $(i_n)_{n\in\N}$ such that $i_n\in[n-1]$ holds is covered. In our setting, we require the labels $v_{i,n}$ to be increasing in $n$ and to diverge with $n$, as we are unable to characterise the limiting distributions of the depth and degree otherwise. We also recover the less general results (compared to Feng \emph{et al.}) of Dobrow~\cite[Theorems $3$ and $4$]{Dob96} on the graph distance between vertices $i_n$ and $n$ with $i_n=n-1$ or $i_n=\lfloor \lambda n\rfloor, \lambda\in(0,1)$. Moreover, we are able to provide a more detailed description of the scaling limit of the distance between the vertices $v_{1,n},\ldots, v_{k,n}$ in relation to their depth, which is not present in~\cite{FengSuLiu06} or~\cite{Dob96}. 
		
		$(ii)$  The theorem recovers the results of Devroye~\cite{Dev98} and Mahmoud~\cite{Mah91} on the insertion depth.
		
		$(iii)$  The theorem recovers a result of Kuba and Panholzer~\cite[Theorem $2$]{KubPan07} regarding the degree of a vertex with a prescribed label. 
		
		$(iv)$ In all cases described in points $(i)$, $(ii)$ and $(iii)$, we extend the results of Feng et al., Devroye, Mahmoud, and Kuba and Panholzer to $k$ vertices and $\binom{k}{2}$ pairs of vertices for any $k\geq 2$.	
		
		$(v)$ The constraint that all $v_{i,n}$ are increasing in $n$ arises due a technicality, which we illustrate with the following example. Suppose $k=2$ and 
		\be 
		v_{1,n}=\lfloor n/2\rfloor\ind_{\{n\text{ is even}\}}+\lfloor n/3\rfloor \ind_{\{n\text{ is odd}\}},\qquad v_{2,n}=\lfloor n/3\rfloor\ind_{\{n\text{ is even}\}}+\lfloor n/2\rfloor \ind_{\{n\text{ is odd}\}}.
		\ee 
		In this case, $c_{1,2}=c_{2,1}=1/\sqrt 2$ both exist, so that the limiting law of the graph distance can be obtained, but the limiting laws of $d^*_{T_n}(v_{1,n})$ and $d^*_{T_n}(v_{2,n})$ do not exist. Indeed, it holds that $d^*_{T_{2n}}(v_{1,2n})\toindis \text{Poi}(\log 2)$ and $d^*_{T_{2n+1}}(v_{1,2n+1})\toindis \text{Poi}(\log 3)$. Such cases are circumvented when the $v_{i,n}$ are increasing with $n$. When omitting the degree, \emph{any} diverging sequences $(v_{i,n})_{i\in[k]}$ such that the $(c_{i,j})_{1\leq i<j\leq k} $ exist can be considered. 
	\end{remark}
	
	The main approach to proving Theorems~\ref{thrm:degdepthlabelrrt}, \ref{thrm:condconvrrt}, and~\ref{thrm:fixedlabel} is to use a `reversed-time' construction or coalescent construction of the random recursive tree, known as the Kingman $n$-coalescent construction (see Section~\ref{sec:king}). This construction has several advantages compared to the construction in Definition~\ref{def:rrt}. First, the depth, degree, and label of vertices in the Kingman $n$-coalescent are exchangeable, which simplifies the analysis of their joint behaviour. Second, the coalescent construction simplifies dealing with correlations that appear when considering the depth, degree, and label of multiple vertices at once. In particular, it provides an elegant way to decouple the degree, depth, and label of distinct vertices. Finally, the size of the depth, degree, and label of a vertex can be understood in terms of sums of independent indicator random variables and independent fair coin flips. As a result, standard central limit theorem results can be applied to obtain the desired results.\\
	
	\noindent \textbf{Outline of the paper}
	
	The paper is organised as follows: We first provide some theoretical preparations, necessary to prove the Theorems stated in Section~\ref{sec:def}. We provide a perspective for Theorem~\ref{thrm:degdepthlabelrrt} in terms of marked point processes, and provide a construction of the random recursive tree, called the Kingman $n$-coalescent construction, that aids in the analysis of the properties of interest here. In particular, we rephrase Theorems~\ref{thrm:condconvrrt} and~\ref{thrm:fixedlabel} in terms of the Kingman $n$-coalescent in Theorems~\ref{thrm:condconvking} and~\ref{thrm:kingfixlabel}, respectively. Section~\ref{sec:pre} is then dedicated to developing some preliminary results based on the Kingman $n$-coalescent construction. These preliminary results are used in Sections~\ref{sec:deg} and~\ref{sec:lab} to obtain intermediate results on the behaviour of high-degree vertices and vertices with a given label, respectively. Finally, these intermediate results are used in Section~\ref{sec:proofppp} to prove Theorem~\ref{thrm:degdepthlabelrrt} and in Section~\ref{sec:kingproof} to prove Theorems~\ref{thrm:condconvrrt} and~\ref{thrm:fixedlabel}.	
	
	\section{The degree, depth, and label of high-degree vertices in the random recursive tree: theoretical preparations}\label{sec:king}
	
	In this section we provide a new perspective of Theorem~\ref{thrm:degdepthlabelrrt}, alongside a different construction of the random recursive tree compared to Definition~\ref{def:rrt}. The latter will be of aid in proving all results presented in Section~\ref{sec:def}.
	
	To prove Theorem~\ref{thrm:degdepthlabelrrt}, we use the convergence of marked point processes. Recall that $d_n^s,h_n^s$ and $\ell_n^s$ denote the degree, depth, and label of the vertex with the $s^{\text{th}}$ largest degree in the random recursive tree, respectively, with $s\in[n]$, where ties are split uniformly at random. Let $\mu:=1-1/(2\log 2)$ and $\sigma^2:=1-1/(4\log 2)$. We view the tuples 
	\be 
	\Big(d_n^s-\lfloor \log_2 n\rfloor, \frac{h_n^s-\mu\log n}{\sqrt{(1-\sigma^2)\log n}}, \frac{\log \ell_n^s -\mu\log n}{\sqrt{\sigma^2\log n}}, s\in[n]\Big),
	\ee 
	as a marked point process, where the rescaled degrees form the points and the rescaled depth and label form the marks of the points. Let $\Z^*:=\Z\cup \{\infty\}$ and endow $\Z^*$ with the metric $d(s,t):=|2^{-s}-2^{-t}|, d(s,\infty)=2^{-s}, s,t\in \Z$. We work with $\Z^*$ rather than $\Z$, as sets $[s,\infty]$ for $s\in\Z$ are now compact. 
	Let $\CP$ be a Poisson point process on $\R$ with intensity $\lambda( \d x):=2^{-x}\log 2\,\d x$ and let $(\xi^{(1)}_x,\xi^{(2)}_x)_{x\in\CP}$ be independent standard normal random variables. For $\eps\in[0,1]$, we define the ground process $\CP^\eps$ on $\Z^*$ and the marked process $\CM\CP^\eps$ on $\Z^*\times \R^2$ by 
	\be \label{eq:limppp2}
	\CP^\eps:=\sum_{x\in\CP}\delta_{\lfloor x+\eps\rfloor}, \qquad \CM\CP^\eps:=\sum_{x\in \CP}\delta_{(\lfloor x+\eps\rfloor, \sqrt{\mu/\sigma^2}\xi^{(1)}_x+\sqrt{1-\mu/\sigma^2}\xi^{(2)}_x, \xi^{(2)}_x)}, 
	\ee  
	where $\delta$ is a Dirac measure. Similarly, we define
	\be \ba\label{eq:pppn}
	\CP^{(n)}&:=\sum_{v=1}^n \delta_{\d_{T_n}(v)-\lfloor\log_2 n\rfloor}, \\ \CM\CP^{(n)}&:=\sum_{v=1}^n \delta_{(\d_{T_n}(v)-\lfloor\log_2 n\rfloor,(h_{T_n}(v)-\mu\log n)/\sqrt{\sigma^2\log n}, (\log v-\mu\log n)/\sqrt{(1-\sigma^2)\log n})}.
	\ea \ee 
	We then let $\CM_{\Z^*}^\#$ and $\CM_{Z^*\times \R^2}^\#$ be the spaces of boundedly finite measures on $\Z^*$ and $\Z^*\times \R^2$, respectively, and observe that $\CP^{(n)},\CP^\eps$ and $\CM\CP^{(n)},\CM\CP^\eps$ are elements of $\CM_{\Z^*}^\#$ and $\CM_{\Z^*\times \R^2}^\#$, respectively. Theorem~\ref{thrm:degdepthlabelrrt} is then equivalent to the weak convergence of $\CM\CP^{(n_t)}$ to $\CM\CP^\eps$ in $\CM_{\Z^*\times \R^2}^\#$ along suitable subsequences $(n_t)_{t\in\N}$, as we can order the points in the definition of $\CM\CP^{(n)}$ (resp.\ $\CM\CP^\eps$) in decreasing order of their degrees (resp.\ of the points $x\in \CP$). We remark that the weak convergence of $\CP^{(n_t)}$ to $\CP^{\eps}$ in $\CM_{\Z^*}^\#$ along subsequences has been established by Addario-Berry and Eslava in~\cite{AddEsl18} (later generalised to weighted recursive trees by Eslava, the author, and Ortgiese in~\cite{EslLodOrt21} and extended to marked point processes by the author in~\cite{Lod21}) and that Eslava established the weak convergence of $\wt{\CM\CP}^{(n_t)}$ along subsequences, which is $\CM\CP^{(n_t)}$ with each mark restricted to the first element (i.e.\ not considering the label), in~\cite{Esl16}. We extend these results here to the tuple of degree, depth, and label, which also shows an interesting dependence in the limit of the rescaled depth and rescaled labels.
	
	Recall the Poisson point process $\CP$ used in the definition of $\CP^\eps$ in~\eqref{eq:limppp2} and enumerate its points in decreasing order. That is, $P_v$ denotes the $v^{\text{th}}$ largest point of $\CP$ (ties broken uniformly at random). We observe that this is well-defined, since $\CP([x,\infty))<\infty$ almost surely for any $x\in \R$. Also, let $(M_v,N_v)_{v\in\N}$ be two sequences of i.i.d.\ standard normal random variables. To prove the weak convergence of the marked point process $\CM\CP^{(n)}$, we define, for $s\in\Z, B\in \CB(\R^2)$, the counting measures
	\be \ba \label{eq:xn}
	X^{(n)}_s(B)&:=\Big|\Big\{v\in[n]: \d_{T_n}(v)=\lfloor \log_2n\rfloor +s, \Big(\frac{h_{T_n}(v)-(\log n-(\lfloor \log_2n\rfloor +s)/2)}{\sqrt{\log n-(\lfloor \log_2n\rfloor +s)/4}},\\
	&\hphantom{:=\Big|\Big\{\inn: \d_{T_n}(v)=\lfloor \log_2n\rfloor +s, \Big(,} \frac{\log v-(\log n-(\lfloor \log_2n\rfloor +s)/2)}{\sqrt{(\lfloor \log_2n\rfloor+s)/4}}\Big)\in B\Big\}\Big|,\\
	X^{(n)}_{\geq s}(B)&:=\Big|\Big\{v\in[n]: \d_{T_n}(v)\geq\lfloor \log_2n\rfloor +s, \Big(\frac{h_{T_n}(v)-(\log n-(\lfloor \log_2n\rfloor +s)/2)}{\sqrt{\log n-(\lfloor \log_2n\rfloor +s)/4}},\\
	&\hphantom{:=\Big|\Big\{\inn: \d_{T_n}(v)=\lfloor \log_2n\rfloor +s, \Big(,} \frac{\log v-(\log n-(\lfloor \log_2n\rfloor +s)/2)}{\sqrt{(\lfloor \log_2n\rfloor+s)/4}}\Big)\in B\Big\}\Big|,\\
	\wt X^{(n)}_s(B)&:=\Big|\Big\{v\in[n]: \d_{T_n}(v)=\lfloor \log_2n\rfloor +s, \Big(\frac{h_{T_n}(v)-\mu\log n}{\sqrt{\sigma^2\log n}}, \frac{\log v-\mu\log n}{\sqrt{(1-\sigma^2)\log n}}\Big)\in B\Big\}\Big|,\\
	\wt X^{(n)}_{\geq s}(B)&:=\Big|\Big\{v\in[n]: \d_{T_n}(v)\geq\lfloor \log_2n\rfloor +s, \Big(\frac{h_{T_n}(v)-\mu\log n}{\sqrt{\sigma^2\log n}}, \frac{\log v-\mu\log n}{\sqrt{(1-\sigma^2)\log n}}\Big)\in B\Big\}\Big|,\\
	X_s(B)&:=\Big|\Big\{v\in\N: \lfloor P_v+\eps\rfloor = s, \Big(M_v\sqrt{1-\frac{\mu}{\sigma^2}}+N_v\sqrt{\frac{\mu}{\sigma^2}}, M_v\Big)\in B\Big\}\Big|,\\
	X_{\geq s}(B)&:=\Big|\Big\{v\in\N: \lfloor P_v+\eps\rfloor \geq s, \Big(M_v\sqrt{1-\frac{\mu}{\sigma^2}}+N_v\sqrt{\frac{\mu}{\sigma^2}}, M_v\Big)\in B\Big\}\Big|.
	\ea \ee 
	We note that, when $s=o(\sqrt{\log n})$, $X_s^{(n)}(B)\approx \wt X^{(n)}_s(B)$ and $X_{\geq s}^{(n)}(B)\approx \wt X^{(n)}_{\geq s}(B)$ for any fixed $B\subseteq \R$. For the result in Theorem~\ref{thrm:degdepthlabelrrt} we are interested in the distributional convergence of $\wt X_s^{(n)}(B),\wt X_{\geq s}^{(n)}(B)$ to $X_s(B),X_{\geq s}(B)$, which we obtain in a more general setting for the random variables $ X_s^{(n)}(B), X_{\geq s}^{(n)}(B)$. The following intermediate result related to these counting measures aids us in obtaining this distributional convergence. 
	
	\begin{proposition}[Factorial moments of counting measures]\label{prop:momentconv}
		Fix $K\in\N$ and $(a_m)_{m\in[K]}\in[0,2)^K$. Let $(s_m)_{m\in[K]}$ be a non-decreasing integer-valued sequence with $0\leq K':=\min\{m: s_{m+1}=s_K\}$ such that $s_1+\log_2 n=\omega(1)$ and 
		\be 
		\lim_{n\to\infty} \frac{s_m+\log_2 n}{\log n}=a_m,
		\ee 
		for all $m\in[K]$. Let $(B_m)_{m\in[K]}$ be a sequence of sets $B_m\subset \CB(\R^2)$ such that $B_m\cap B_\ell=\emptyset$ when $s_m=s_\ell$ and $m\neq \ell$, let $(c_m)_{m\in[K]}\in \N_0^K$ and let $M$ and $N$ be two independent standard normal random variables. Recall the random variables $X_s^{(n)}(B),X_{\geq s}^{(n)}(B)$ and $\wt X_s^{(n)}(B),\wt X_{\geq s}^{(n)}(B)$ from~\eqref{eq:xn}, and define $\eps_n:=\log_2n-\lfloor \log_2n\rfloor$. Then, 
		\be \ba 
		\mathbb E\bigg[\prod_{m=1}^{K'}{}&\Big(X_{s_m}^{(n)}(B_m)\Big)_{c_m}\prod_{m=K'+1}^K\Big(X_{\geq s_m}^{(n)}(B_m)\Big)_{c_m} \bigg]\\
		={}&(1+o(1))\prod_{m=1}^{K'}\bigg(2^{-(s_m+1)+\eps_n}\P{\Big(M\sqrt{\frac{a_m}{4-a_m}}+N\sqrt{1-\frac{a_m}{4-a_m}}, M\Big)\in B_m}\bigg)^{c_m}\\
		&\times\prod_{m=K'+1}^K\bigg(2^{-s_K+\eps_n}\P{\Big(M\sqrt{\frac{a_m}{4-a_m}}+N\sqrt{1-\frac{a_m}{4-a_m}}, M\Big)\in B_m}\bigg)^{c_m} .
		\ea \ee 
		Moreover, when $s_1, \ldots, s_K=o(\sqrt{\log n})$ and $a_m=1/\log 2$ for all $m\in[K]$,
		\be \ba 
		\mathbb E\bigg[\prod_{m=1}^{K'}{}&\Big(\wt X_{s_m}^{(n)}(B_m)\Big)_{c_m}\prod_{m=K'+1}^K\Big(\wt X_{\geq s_m}^{(n)}(B_m)\Big)_{c_m} \bigg]\\
		={}&(1+o(1))\prod_{m=1}^{K'}\bigg(2^{-(s_m+1)+\eps_n}\P{\Big(M\sqrt{1-\frac{\mu}{\sigma^2}}+N\sqrt{\frac{\mu}{\sigma^2}}, M\Big)\in B_m}\bigg)^{c_m}\\
		&\times\prod_{m=K'+1}^K\bigg(2^{-s_K+\eps_n}\P{\Big(M\sqrt{1-\frac{\mu}{\sigma^2}}+N\sqrt{\frac{\mu}{\sigma^2}}, M\Big)\in B_m}\bigg)^{c_m} .
		\ea \ee 
	\end{proposition}
	
	As the counting measures defined in~\eqref{eq:xn} are sums of indicator random variables, their factorial moments can be expressed in terms of probabilities 
	\be \ba 
	\mathbb P({}&\d_{T_n}(v_i) \geq d_i, (h_{T_n}(v_i), \log v_i)\in B_i, i\in[k])\\
	&=\P{\d_{T_n}(v_i)\geq d_i, i\in[k]}\P{(h_{T_n}(v_i), \log v_i)\in B_i, i\in[k]\,|\, \zni\geq d_i, i\in[k]}.
	\ea \ee 
	Here, we let $(d_i)_{i\in[k]}\in \N_0^k$ such that $d_i<2\log n$, $(B_i)_{i\in[k]}\in \CB(\R^2)^k$, $(v_i)_{i\in[k]}$ distinct vertices selected uniformly at random, and $k\in\N$. The first probability on the right-hand side is studied by Addario-Berry and Eslava in~\cite{AddEsl18}, and the latter is the subject of Theorem~\ref{thrm:condconvrrt}.	This can in turn be used to prove Proposition~\ref{prop:momentconv}, which finally leads to Theorem~\ref{thrm:degdepthlabelrrt}. We provide more details alongside the proof of Proposition~\ref{prop:momentconv} and Theorem~\ref{thrm:degdepthlabelrrt} in Section~\ref{sec:proofppp}.
	
	\subsection{\texorpdfstring{The Kingman $n$-coalescent}{}}
	
	We now provide an alternative construction of the random recursive tree (RRT), which we use to prove Theorems~\ref{thrm:degdepthlabelrrt}, ~\ref{thrm:condconvrrt} and~\ref{thrm:fixedlabel}. 
	
	This alternative construction of the RRT, (a variant of) the Kingman $n$-coalescent construction, was first discussed by Pittel in~\cite{Pit94} and recovered and used by Addario-Berry and Eslava to study high degrees in RRTs~\cite{AddEsl18}. Later, Eslava extended this to the joint convergence of the depth and degree of vertices with large degree~\cite{Esl16} and also provides a more general coupled recursive construction of a tree $T$ and a permutation $\sigma$ on the labels of the vertices of $T$, coined Robin-Hood pruning~\cite{Esl21}. Here, we further extend Eslava's results from~\cite{Esl16} on the depth and degree of high-degree vertices to also include the label of and graph distance between such high-degree vertices. We also obtain results on the joint behaviour of the degree and depth of and graph distance between vertices with a given label, which combine, extend and improve several known results from the literature on the degree~\cite{KubPan07} and depth~\cite{Dev98} of a vertex with a given label and the graph distance between vertices $n$ and $i_n$, for any sequence $i_n$~\cite{FengSuLiu06}.
	
	The variant of the Kingman $n$-coalescent we use here is a process which starts with $n$ trees, each consisting of only a single root. At every step $n$ through $2$ (counting backwards), a pair of roots is selected uniformly at random and independently of this selection a directed edge is formed between the two roots, each direction being equiprobable. This reduces the number of trees by one and, after completing step $2$, yields a directed tree. It turns out that a particular relabelling of this directed tree yields a tree equal in law to the random recursive tree. Moreover, using the Kingman $n$-coalescent construction simplifies the analysis of degrees, depths, and labels in the RRT model, among other reasons because the degree, depth, and label of the vertices are exchangeable random variables in the Kingman $n$-coalescent.
	
	We now formally introduce the Kingman $n$-coalescent construction of the random recursive tree. Let $\CC\mathcal F_n:=\{f: V(f)=[n]\}$ denote the set of all forests with exactly $n$ vertices. An $n$-chain is a sequence $(f_n,\ldots, f_1)$ of elements of $\CC\mathcal F_n$, where for each integer $1<j\leq n$, $f_{j-1}$ is obtained from $f_j$ by adding a directed edge between the roots of two trees in $f_j$. We write $f_j=\{t_1^{(j)},\ldots, t^{(j)}_j\}$, ordering the trees in increasing order of their smallest-labelled vertex. In particular, $f_n$ consists of $n$ trees, each of which is a root with no edges, and $f_1$ consists of exactly one tree. Also, we let $r(T)$ denote the root of the tree $T$ and write $F_j=\{T^{(j)}_1,\ldots, T^{(j)}_j\}$ for a random element in $\CC\CF_n$ for any $j\in[n]$.
	
	\begin{definition}[Kingman $n$-coalescent] \label{def:king}
		For each $1<j\leq n$, choose a pair \\$\{a_j,b_j\}\subseteq\{\{a,b\}: 1\leq a<b\leq j\}$ independently and uniformly at random; also let $(\xi_j)_{1<j\leq n}$ be a sequence of independent $\mathrm{Bernoulli}(1/2)$ random variables. Initialise the coalescent by $F_n$: a forest of $n$ trees, each consisting of a root and no edges. For $1<j\leq n$, $F_{j-1}$ is obtained from $F_j$ as follows: Add an edge $e_{j-1}$ between the roots $r(T_{a_j}^{(j)})$ and $r(T_{b_j}^{(j)})$; direct $e_{j-1}$ towards $r(T_{a_j}^{(j)})$ if $\xi_j=1$ and towards $r(T_{b_{j}}^{(j)})$ if $\xi_j=0$. Then, $F_{j-1}$ consists of the new tree and the remaining $j-1$ unaltered trees from $F_{j}$.
		
		Finally, let $T^{(n)}:=T^{(1)}_1= F_1$ denote the final tree in the coalescent $\mathbf{C}=(F_n,\ldots, F_1)$.
	\end{definition}
	
	See Figure~\ref{fig:king} for an example of the process. When at step $j$ the edge $e_j=v_ju_j$ is directed towards $u_j$, we say that the associated random variable $\xi_j$ (which we can interpret as flipping a fair coin) favours the root $u_i$. Similarly, we might also say that $\xi_j$ favours $w$ or that the associated coin flip at step $j$ favours $w$, where $w$ is any vertex in the tree that contains $u_j$. 
	
	\begin{figure}
		\centering
		\includegraphics[width=0.54\textwidth]{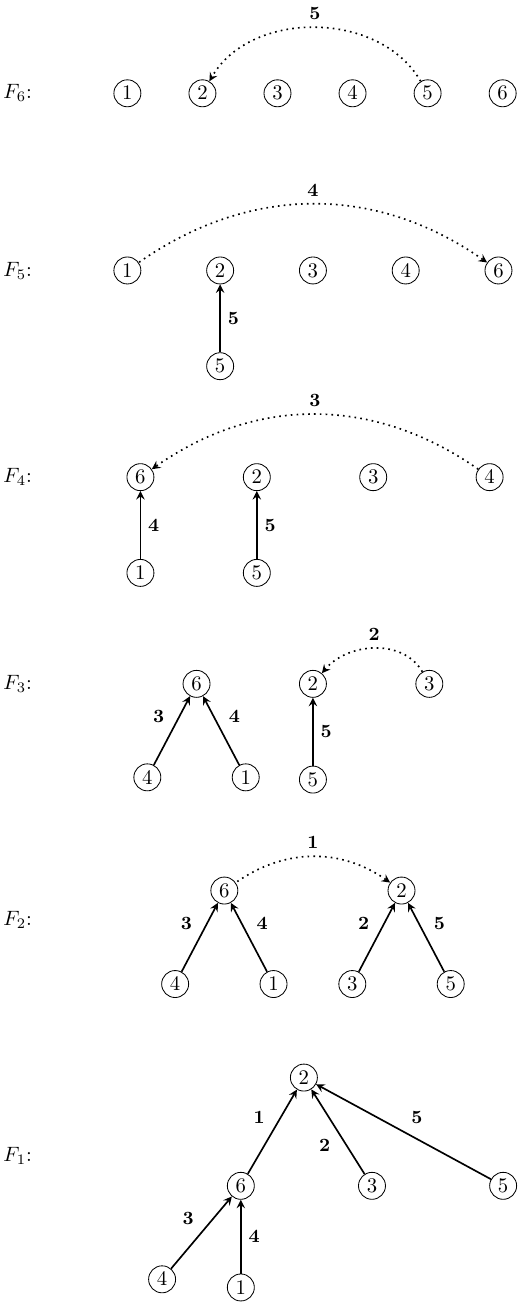}
		\caption{An example of the Kingman $n$-coalescent $\mathbf{C}=(F_n,\ldots, F_1)$ for $n=6$. For $2\leq j\leq 6$, we represent the edge in $E(F_{j-1})\backslash E(F_j)$ with a dotted line in $F_j$. In this case, $\xi_6=\xi_4=\xi_3=1, \xi_5=\xi_2=0$ and $\{a_6,b_6\}=\{2,5\}$, $\{a_5,b_5\}=\{1,5\}$, $ \{a_4,b_4\}=\{1,4\}$, $\{a_3,b_3\}=\{2,3\}$, $\{a_2,b_2\}=\{1,2\}$. From~\cite{Esl16}.}\label{fig:king}
	\end{figure}
	
	The link between the final tree in the coalescent and the RRT is as follows. Let us define the mapping $\sigma_C:V(T^{(n)})\to [n]$ by $\sigma_C(r(T^{(n)})):=1$ and for each edge $e_j=v_ju_j\in E(T^{(n)})$,  $j\in[n-1]$, 
	\be \label{eq:sigmac}
	\sigma_C(v_j):=j+1.
	\ee 
	As all edges are directed towards the root, $v_j\neq v_{j'}$ for all $j\neq j'\in[n-1]$, so that $\sigma_C$ is well-defined. $\sigma_C$ is the relabelling of $T^{(n)}$ into an increasing tree. If we let $\CI_n$ denote the set of all increasing trees on $n$ vertices, then it is clear that the RRT is a uniform element in $\CI_n$. The most important attribute of the $n$-chain in the Kingman $n$-coalescent is that it has a uniform distribution over all possible $n$-chains and that the relabelling of $T^{(n)}$ by $\sigma_C$ yields a uniform element of $\CI_n$, as outlined in the following proposition.
	
	\begin{proposition}[Lemma $7.1$ and Proposition $7.2$ in~\cite{Esl16}]\label{prop:lawking}
		The Kingman $n$-coalescent $\mathbf{C}$ is uniformly random in $\mathcal C\mathcal F_n$, the set of $n$-chains. Moreover, for each $C=(f_n,\ldots, f_1)\in \mathcal C\mathcal F_n$, relabel the vertices in $f_1$ with $\sigma_C$ to obtain a tree $\phi(C)\in \mathcal I_n$. Then the law of $\phi(\mathbf{C})$ is that of a random recursive tree of size $n$.
	\end{proposition}
	
	Recall that $\d_{T_n}(u)$, $h_{T_n}(u)$ and $\dist_{T_n}(u,v)$ denote the in-degree and depth of vertex $u\in[n]$ and the graph distance between vertices $u,v\in[n]$ in the random recursive tree $T_n$, respectively. Similarly, for a realisation of the final tree $T^{(n)}$ in the coalescent $\mathbf{C}$, let $\d_{T^{(n)}}(i), h_{T^{(n)}}(i)$ and $\dist_{T^{(n)}}(i,j)$ denote the in-degree and depth of vertex $i$ and the graph distance between $i$ and $j$, respectively, and let $\ell_{T^{(n)}}(i):=\sigma_C(i)$ denote the relabelling of vertex $i$, $\inn$. That is, $\ell_{T^{(n)}}(i)$ denotes the label that vertex $i$ in $\mathbf{C}$ obtains in the random recursive tree $\phi(\mathbf{C})$. We can then formulate the following corollary.
	
	\begin{corollary}\label{cor:unifnodes}
		Let $T_n$ be a random recursive tree and let $T^{(n)}$ be the resulting tree in the Kingman $n$-coalescent. Let $\sigma:[n]\to[n]$ be a uniform random permutation on $[n]$. Then, 
		\be \ba 
		\big({}&(\d_{T^{(n)}}(i), h_{T^{(n)}}(i),\ell_{T^{(n)}}(i))_{i\in[n]},(\dist_{T^{(n)}}(i,j))_{1\leq i<j\leq n}\big)\\
		&\overset d= \big((\d_{T_n}(\sigma(i)),h_{T_n}(\sigma(i)),\sigma(i))_{i\in[n]},(\dist_{T_n}(\sigma(i),\sigma(j)))_{1\leq i<j\leq n}\big).
		\ea \ee 
		Moreover, jointly for all $i,j\in\N$ and all sets $B\subseteq [n]$, we have
		\be 
		|\{v\in B: \d_{T^{(n)}}(v)=i, h_{T^{(n)}}(v)=j\}|\overset d=|\{v\in[n]: \sigma(v)\in B,\d_{T_n}(\sigma(v))=i,h_{T_n}(\sigma(v))=j\}|.
		\ee 
	\end{corollary}
	
	In what follows, we replace the subscript $T^{(n)}$ with $n$ for ease of writing, since we work with the coalescent from now on instead of the RRT. As a direct result from Corollary~\ref{cor:unifnodes}, Theorem~\ref{thrm:condconvrrt} follows from the following result (which is a reformulation of Theorem~\ref{thrm:condconvrrt} in terms of the Kingman $n$-coalescent).
	
	\begin{theorem}\label{thrm:condconvking}
		Consider the Kingman $n$-coalescent as in Definition~\ref{def:king}. Fix $k\in\N$ and $(a_i)_{i\in[k]}\in[0,2)^k$ and let $(d_i)_{i\in[k]}$ be $k$ integer-valued sequences such that 
		\be 
		\lim_{n\to\infty}\frac{d_i}{\log n}=a_i,
		\ee 
		for each $i\in[k]$. The tuple 
		\be 
		\bigg(\Big(\frac{h_{n}(i)-(\log n-d_i/2)}{\sqrt{\log n-d_i/4}}\Big)_{i\in[k]}\Big(\frac{\dist_n(i,j)-(2\log n-(d_i+d_j)/2)}{\sqrt{2\log n-(d_i+d_j)/4}}\Big)_{1\leq i<j\leq k}\bigg),
		\ee 
		conditionally on the event $d_{n}(i)\geq d_i$ for all $i\in[k]$, converges in distribution to 
		\be \label{eq:withoutlab}
		\bigg((H_i)_{i\in[k]},\bigg(\frac{\sqrt{4-a_i}H_i+\sqrt{4-a_j}H_j}{\sqrt{8-(a_i+a_j)}}\bigg)_{1\leq i<j\leq k}\bigg),
		\ee 
		where the $(H_i)_{i\in[k]}$ are independent standard normal random variables. Additionally assume that for all $i\in[k]$, $d_i$ diverges as $n\to\infty$. Then, the tuple
		\be \ba 
		\bigg({}&\Big(\frac{h_{n}(i)-(\log n-d_i/2)}{\sqrt{\log n-d_i/4}}, \frac{\log(\ell_{n}(i))-\log n-d_i/4}{\sqrt{d_i/4}}\Big)_{i\in[k]},\\
		&\Big(\frac{\dist_n(i,j)-(2\log n-(d_i+d_j)/2)}{\sqrt{2\log n-(d_i+d_j)/4}}\Big)_{1\leq i<j\leq k}\bigg),
		\ea \ee 
		conditionally on the event $d_{n}(i)\geq d_i$ for all $i\in[k]$, converges in distribution to 
		\be\ba   \label{eq:withlab}
		\bigg(\Big({}&M_i\sqrt{\frac{a_i}{4-a_i}}+N_i\sqrt{1-\frac{a_i}{4-a_i}},M_i \Big)_{ i\in[k]},\\
		&\Big(\frac{M_i\sqrt{a_i}+N_i\sqrt{4-2a_i}+M_j\sqrt{a_j}+N_j\sqrt{4-2a_j}}{\sqrt{8-(a_i+a_j)}}\Big)_{1\leq i<j\leq k}\bigg),
		\ea\ee
		where the $(M_i,N_i)_{i\in[k]}$ are independent standard normal random variables. 
	\end{theorem} 
	
	\begin{remark}
		As is the case in Remark~\ref{rem:degthrm}, the same results in Theorem~\ref{thrm:condconvking} can be obtained when working with the conditional event $\{\d_n(v_i)= d_i, i\in[k]\}$ rather than $\{\d_n(v_i)\geq d_i, i\in[k]\}$, with an almost identical proof.
	\end{remark}
	
	Moreover, Theorem~\ref{thrm:condconvking} can be used to prove Proposition~\ref{prop:momentconv}. By Corollary~\ref{cor:unifnodes}, we can redefine the random variables $X_s^{(n)}(B),X_{\geq s}^{(n)}(B)$ and $\wt X_s^{(n)}(B),\wt X_{\geq s}^{(n)}(B)$, as defined in~\eqref{eq:xn}, in terms of the Kingman $n$-coalescent, by writing, for $s\in \Z, B\in \CB(\R^2)$, 
	\be \ba\label{eq:xnking}
	X^{(n)}_s(B)&:=\Big|\Big\{\inn: \d_n(i)=\lfloor \log_2n\rfloor +s, \Big(\frac{h_n(i)-(\log n-(\lfloor \log_2n\rfloor +s)/2)}{\sqrt{\log n-(\lfloor \log_2n\rfloor +s)/4}},\\
	&\hphantom{:=\Big|\Big\{\inn: \d_n(i)=\lfloor \log_2n\rfloor +s, \Big(,} \frac{\log \ell_n(i)-(\log n-(\lfloor \log_2n\rfloor +s)/2)}{\sqrt{(\lfloor \log_2n\rfloor+s)/4}}\Big)\in B\Big\}\Big|,\\
	X^{(n)}_{\geq s}(B)&:=\Big|\Big\{\inn: \d_n(i)\geq\lfloor \log_2n\rfloor +s, \Big(\frac{h_n(i)-(\log n-(\lfloor \log_2n\rfloor +s)/2)}{\sqrt{\log n-(\lfloor \log_2n\rfloor +s)/4}},\\
	&\hphantom{:=\Big|\Big\{\inn: \d_n(i)=\lfloor \log_2n\rfloor +s, \Big(,} \frac{\log \ell_n(i)-(\log n-(\lfloor \log_2n\rfloor +s)/2)}{\sqrt{(\lfloor \log_2n\rfloor+s)/4}}\Big)\in B\Big\}\Big|,\\
	\wt X^{(n)}_s(B)&:=\Big|\Big\{\inn: \d_n(i)=\lfloor \log_2n\rfloor +s, \Big(\frac{h_n(i)-\mu\log n}{\sqrt{\sigma^2\log n}}, \frac{\log \ell_n(i)-\mu\log n}{\sqrt{(1-\sigma^2)\log n}}\Big)\in B\Big\}\Big|,\\
	\wt X^{(n)}_{\geq s}(B)&:=\Big|\Big\{\inn: \d_n(i)\geq\lfloor \log_2n\rfloor +s, \Big(\frac{h_n(i)-\mu\log n}{\sqrt{\sigma^2\log n}}, \frac{\log \ell_n(i)-\mu\log n}{\sqrt{(1-\sigma^2)\log n}}\Big)\in B\Big\}\Big|.
	\ea \ee 
	We can also reformulate Theorem~\ref{thrm:fixedlabel} in terms of the Kingman $n$-coalescent. As is the case with Theorem~\ref{thrm:condconvrrt}, combining Corollary~\ref{cor:unifnodes} with the following theorem immediately implies Theorem~\ref{thrm:fixedlabel}.
	
	\begin{theorem}\label{thrm:kingfixlabel}
		Consider the Kingman $n$-coalescent as in Definition~\ref{def:king}. Fix $k\in\N$ and let $(\ell_i)_{i\in[k]}\in[n]^k$ be $k$ distinct integer-valued sequences such that $\ell_i$ increases with $n$, diverges as $n\to\infty$ and such that 
		\be \label{eq:cij}
		c_{i,j}:=\lim_{n\to\infty}\sqrt{\frac{\log \ell_i}{\log \ell_i+\log \ell_j}}
		\ee 
		exists for all $1\leq i<j\leq k$. Let $(N_i)_{i\in[k]}$ be $k$ independent standard normal random variables. We also define for $(\rho_i)_{i\in[k]}\in(0,1)^k$ and each $i\in[k]$, 
		\be\ba \label{eq:dstar}
		d^*_n(i)&:=\begin{cases} \frac{\d_n(i)-\log(n/\ell_i)}{\sqrt{\log(n/\ell_i)}}, & \mbox{if }\ell_i=o(n),\\
			\d_n(i), &\mbox{otherwise},
		\end{cases}\\
		Z_i&\sim \begin{cases}
			\mathcal N(0,1) &\mbox{if } \ell_i=o(n),\\
			\mathrm{Poi}(\log(1/\rho_i)) & \mbox{if }\ell_i=(1+o(1))\rho_i n,\\
			0 & \mbox{if } \ell_i=n-o(n).
		\end{cases}
		\ea\ee 
		where the $Z_i$ are independent and also independent of the $(N_i)_{i\in[k]}$. The tuple
		\be
		\bigg(\Big(d^*_n(i),\frac{h_n(i)-\log \ell_i}{\sqrt{\ell_i}}\Big)_{i\in[k]}, \Big(\frac{\dist_n(i,j)-(\log \ell_i+\log \ell_j)}{\sqrt{\log \ell_i+\log \ell_j}}\Big)_{1\leq i<j\leq k}\bigg),
		\ee 
		conditionally on the event $\ell_n(i)=\ell_i$ for all $i\in[k]$, converges in distribution to 
		\be 
		\big((Z_i,N_i)_{i\in[k]}, (c_{i,j}N_i+c_{j,i}N_j)_{1\leq i<j\leq k}\big).
		\ee 
	\end{theorem}
	
	\begin{remark}
		It is necessary to work on the conditional event $\{\ell_n(i)=\ell_i, i\in[k]\}$ in Theorem~\ref{thrm:kingfixlabel}, despite this not being the case in Theorem~\ref{thrm:fixedlabel}. Since vertices $1,\ldots, k$ in the Kingman $n$-coalescent obtain a \emph{random} label in the relabelled tree $\phi(\mathbf{C})$ (which is equal in law to the random recursive tree by Proposition~\ref{prop:lawking}), the need to condition on their relabelling $\ell_n(i)=\ell_i,i\in[k],$ arises.
	\end{remark}
	
	In the next sections we analyse the Kingman $n$-coalescent construction to prove Theorems~\ref{thrm:condconvking} and~\ref{thrm:kingfixlabel} and Proposition~\ref{prop:momentconv}. 
	
	\section{Preliminary results} \label{sec:pre}

	In this section we provide some important intermediate results related to the Kingman $n$-coalescent construction, provided in Section~\ref{sec:king}. We focus on two things in this section. First, we study the evolution of the degree, depth, and label of vertices $1,\ldots, k$ in the Kingman $n$-coalescent, which is an important first step in proving the theorems in Section~\ref{sec:king}. Second, we investigate the correlations between the steps $j\in[2,n]$ at which vertices $1,\ldots, k$ are selected in the coalescent.
	
	Though the theorems presented in Section~\ref{sec:king} are concerned with the graph distance between vertices $1,\ldots, k$ as well as their degree, depth, and label, we do not include this in our analysis yet. While the latter quantities are easier to explicitly understand in terms of the Kingman $n$-coalescent, the graph distance does not lend itself to an equally elegant analysis. As it turns out, though, there is a close relation between the depth of and graph distance between the vertices $1, \ldots, k$ which allows us to infer the scaling limit of the graph distances from the results on the depth. We make use of this relation in Section~\ref{sec:kingproof} when proving Theorems~\ref{thrm:condconvking} and~\ref{thrm:kingfixlabel}.
	
	\subsection{Analysis of the Kingman $n$-coalescent}
	
	We start by introducing some notation related to the Kingman $n$-coalescent. For an $n$-chain $C=(f_n,\ldots, f_1)$ and some $i,j\in[n]$, let $T^{(j)}(i)$ denote the tree in $f_j$ that contains vertex $i$. For $\inn$, let $s_{i,j}$ be the indicator that $T^{(j)}(i)\in\{T_{a_j}^{(j)},T_{b_j}^{(j)}\}$ and let $h_{i,j}$ be the indicator that the edge $e_j$ is directed outwards from $r(T^{(j)}(i))$, $2\leq j\leq n$. That is, $s_{i,j}$ equals one if $i$ is part of one of the two trees selected to merge at step $j$, and $h_{i,j}$ is one if $s_{i,j}$ is one and if the new edge $e_j$ causes vertex $i$ to increase its depth by one, see Figure~\ref{fig:edgedirect}.
	
	\begin{figure}[t]
		\centering
		\includegraphics[width=0.7\textwidth]{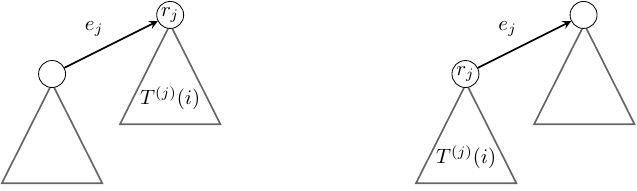}
		\caption{For $\inn$ and $2\leq j\leq n$, let $r_j:=r(T^{(j)}(i))$ denote the root of the tree in $f_j$ that contains vertex $i$, and suppose that $T^{(j)}(i)\in \{T^{(j)}_{a_j},T^{(j)}_{b_j}\}$. If $e_j$ is directed towards $r_j$, then the degree of $r_j$ increases by one in $F_{j-1}$. If $e_i$ is directed outwards from $r_j$, then the depth of each $v\in T^{(j)}(i)$ increases by one in $F_{j-1}$. From~\cite{Esl16}.}\label{fig:edgedirect}
	\end{figure}
	
	Since the trees selected to be merged at every step are independent and uniformly distributed, the variables $(s_{i,j})_{2\leq j\leq n}$ are independent Bernoulli random variables for any fixed $\inn$, with $\E{s_{i,j}}=2/j$. Similarly, since the direction of the edge $e_j$ depends only on $\xi_j$, the variables $(h_{i,j})_{2\leq j\leq n}$ are also independent Bernoulli random variables for any fixed $\inn$, with $\E{h_{i,j}}=1/j$. 
	
	Let us define
	\be 
	\mathcal S_n(i):= \{2\leq j\leq n: s_{i,j}=1\},\qquad \inn,
	\ee 
	and set $S_n(i):=|\mathcal S_n(i)|$. We refer to $\S_n(i)$ as the \emph{selection set} of vertex $i$. We can express the quantities $\d_n(i),h_n(i)$ and $\ell_n(i)$ in terms of $\S_n(i)$ and the indicator variables $(h_{i,j})_{j\in\S_n(i)}$. Namely, if we write $\S_n(i)=\{j_{i,1}, \ldots, j_{i,S_n(i)}\}$ with $j_{i,1}>j_{i,2}>\ldots>j_{i,S_n(i)}$, then
	\be \ba\label{eq:quantdescr}
	\d_n(i)&=\max\{0\leq d \leq S_n(i): h_{i,j_{i,1}}=\ldots =h_{i,j_{i,d}}=0\}, \\
	h_n(i)&=\sum_{j\in\S_n(i)}h_{i,j}=\sum_{j=2}^n h_{i,j}, \\ \ell_n(i)&=\max\{j\in \S_n(i): h_{i,j}=1\}=\max\{\inn: h_{i,j}=1\},
	\ea\ee 
	where we set $h_{i,1}=1$ for all $\inn$, so that $\max\{j\in[n]: h_{i,j}=1\}=1$ if there is no $2\leq j\leq n$ such that $h_{i,j}=1$ (which corresponds to vertex $i$ being the root of $T^{(n)}$, so that its relabelling with $\sigma_C$ as in~\eqref{eq:sigmac} yields $\ell_n(i)=1$). Note that there is always a unique vertex $i$ for which $h_{i,j}=0$ for all $2\leq j\leq n$, so that $\ell_n(i)\neq \ell_n(i')$ whenever $i\neq i'$. Explaining~\eqref{eq:quantdescr} in words, the degree of a vertex $i$ is equal to the length of the first streak of zeros of the indicators $(h_{i,j_\ell})_{\ell\in [S_n(i)]}$, the relabelling of vertex $i$ in the RRT is equal to the first step directly after this streak when $h_{i,j}=1$, and the depth equals the total number of steps $j$ for which $h_{i,j}=1$.
	
	We are interested in the behaviour of the degree, depth, and label of the vertices $1,\ldots, k$ for \emph{any} fixed $k\in\N$. While these quantities are easily expressed in terms of the selection sets $(\S_n(i))_{i\in[k]}$ and the associated coin flips, as in~\eqref{eq:quantdescr}, considering $k$ vertices provides some additional difficulties in terms of correlations between the selection sets of these $k$ vertices. The main issue is the following: whenever two distinct vertices $i,i'\in[k]$ are both selected \emph{at the same step}, say step $\lambda_{i,i'}$, there is a dependence between the outcome of the associated coin flip of vertices $i$ and $i'$. Namely, $h_{i,\lambda_{i,i'}}=1-h_{j,\lambda_{i,i'}}$. Furthermore, for any step $2\leq j<\lambda_{i,j}$, we know that $h_{i,j}=h_{i',j}$. As these correlations between the vertices $1,\ldots, k$ are difficult to handle, we define  
	\be\label{eq:tauk}
	\tau_k:=\max\{2\leq j\leq n: s_{i,j}=s_{i',j}=1\text{ for distinct }i,i'\in[k]\}.
	\ee 
	Since the trees in the Kingman $n$-coalescent are ordered based on their smallest-labelled vertex, $\tau_k$ is the first step at which two vertices $i,i'\in[k]$ are both selected (in the sense that the root of the tree they belong to is selected), and thus up to step $\tau_k$ the vertices $1,\ldots, k$ are contained in disjoint trees. As a result, this implies that the sets $[\tau_k+1,n]\cap \S_n(1), \ldots, [\tau_k+1,n]\cap \S_n(k)$ are disjoint, and since the associated coin flips of these disjoint sets are independent, the evolutions of the degree, depth, and label of vertices $1,\ldots,k$, up to step $\tau_k$ are independent. This helps to avoid correlations and simplifies the analysis. Eslava (implicitly) shows in the proof of~\cite[Lemma $3.2$]{Esl16} that the sequence $(\tau_k)_{n\in\N}$ is a tight sequence of random variables. As a result, for any integer-valued sequence $(t_n)_{n\in\N}$ which diverges to infinity as $n\to\infty$, we know that $\P{\tau_k<t_n}=1-o(1)$. This justifies, for $t_n\leq n$, the definition of the sets, for each $\inn$,
	\be \ba \label{eq:omega}
	\Omega_1:=\{t_n,\ldots, n\},\quad \S_{n,1}(i):=\{j\in\Omega_1: s_{i,j}=1\},\quad \CH_{n,1}(i):=\{j\in\Omega_1: h_{i,j}=1\},
	\ea \ee 
	and we let $S_{n,1}(i):=|\S_{n,1}(i)|$ and $h_{n,1}(i):=|\CH_{n,1}(i)|, h_{n,2}(i):=h_n(i)-h_{n,1}(i)$. We refer to the sets $(\S_{n,1}(i))_{\inn}$ as the \emph{truncated selection sets}, to $h_{n,1}(i)$ as the \emph{truncated depth} of vertex $i$, and to $(t_n)_{n\in\N}$ as the \emph{truncation sequence}. Though $\S_{n,1}(i)$, $h_{n,1}(i),h_{n,2}(i)$ depend on $t_n$, we omit this in their notation for ease of writing. The truncated depth $h_{n,1}(i)$ and $h_{n,2}(i)$ can be described similar to $h_n(i)$ in~\eqref{eq:quantdescr}, as 
	\be 
	h_{n,1}(i)=\sum_{j\in \S_{n,1}(i)}\!\! h_{i,j}=\sum_{j=t_n}^n h_{i,j}, \qquad h_{n,2}(i)=\sum_{j\in \S_n(i)\backslash \S_{n,1}}\!\! h_{i,j}=\sum_{j=2}^{t_n-1}h_{i,j}=h_n(i)-h_{n,1}(i).
	\ee 
	
	The following lemma uses~\eqref{eq:quantdescr} to provide a description of the relation between the joint distribution of $\d_n(1), h_{n,1}(1)$ and $\ell_n(1)$ and the truncated selection set $\S_{n,1}(1)$. Since the vertices are exchangeable, as follows from Corollary~\ref{cor:unifnodes}, the lemma also holds for any vertex $\inn$.
	
	\begin{lemma}\label{lemma:probonevert}
		Let $G\sim \mathrm{Geo}(1/2)$ be independent from $\S_n(1)$. Then $\d_n(1)\overset d= \min\{G,S_n(1)\}$. Moreover, fix $h,d\in\N_0$ and consider a truncation sequence $(t_n)_{n,\in\N}$ such that $t_n\leq n$ for all $n\in\N$. Let $\ell\in\Omega_1$, $J\subseteq\Omega_1$, and let $X_{n,\ell,1}\sim\mathrm{Bin}(|[\ell,n]\cap J|-d,1/2)$ and\\ $X_{n,\ell,2}\sim\mathrm{Bin}(|[t_n,\ell-1]\cap J|,1/2)$ be two independent binomial random variables $($\!where we set $X_{n,\ell,1}=0, X_{n,\ell,2}=0$ when $|[\ell,n]\cap J|-d\leq 0,|[t_n,\ell-1]\cap J|=0$, respectively$)$. Then, 
		\be \ba\label{eq:geql}
		\mathbb P{}&(h_{n,1}(1)\leq h, \ell_n(1)\geq \ell, \d_n(1)\geq d\,|\, \S_{n,1}(1)=J)\\
		&=2^{-d}\ind_{\{|[\ell,n]\cap J|\geq d+1\}}\P{X_{n,\ell,1}+X_{n,\ell,2}\leq h,X_{n,\ell,1}\geq 1}.
		\ea\ee 
		Furthermore, 
		\be \ba \label{eq:eql}
		\mathbb P{}&(h_{n,1}(1)\leq h, \ell_n(1)= \ell, \d_n(1)\leq d\,|\, \S_{n,1}(1)=J)\\
		&=\ind_{\{|[\ell+1,n]\cap J|\leq d\}}\ind_{\{\ell\in J\}}2^{-(|[\ell+1,n]\cap J|+1)}\P{X_{n,\ell,2}\leq h-1}.
		\ea \ee 
	\end{lemma}
	
	\begin{remark}\label{rem:nolab}
		In the case $\ell=1$, the result in~\eqref{eq:geql} simplifies to 
		\be 
		\P{h_n(1)\leq h, \d_n(1)\geq d\,|\, \S_{n,1}(1)=J}=2^{-d}\ind_{\{|J|\geq d\}}\P{X_n\leq h},
		\ee 
		where $X_n\sim \text{Bin}(|J|-d,1/2)$ and we set $X_n=0$ when $|J|-d\leq 0$. The proof follows the same approach as the proof of~\eqref{eq:geql} and is hence omitted.
	\end{remark} 
	\begin{remark}
		The constraint $\ell\geq t_n$ ensures that the events $\ell_n(1)\geq \ell$ and $\ell_n(1)=\ell$, as in~\eqref{eq:geql} and~\eqref{eq:eql}, respectively, can be determined by step $t_n$ of the Kingman $n$-coalescent. In what follows, we let $t_n$ grow sufficiently slow so that this constraint is satisfied for any choice of $\ell$ that is of interest.
	\end{remark} 
	
	\begin{proof}
		Let us start by proving~\eqref{eq:geql}. We define $\mathcal E_n:=\{h_{n,1}(1)\leq h, \ell_n(1)\geq \ell, \d_n(1)\geq d\}$. If we condition on the event $\{\S_{n,1}(1)=J\}$ for some set $J\subseteq \Omega_1$, then we can express the occurrence and probability of the event $\mathcal E_n$ in terms of $J$: 
		\begin{enumerate}
			\item[$(i)$] Conditionally on $\{\S_{n,1}(1)=J\}$, $\mathcal E_n$ can only occur if $|[\ell,n]\cap J|\geq d+1$ by the first and last line of~\eqref{eq:quantdescr}: 
			\begin{enumerate}
				\item[$(a)$] By the first line of~\eqref{eq:quantdescr}, the degree of vertex $i$ is at least $d$ when a streak $h_{1,j_{1,1}}=\ldots=h_{1,j_{1,d}}=0$ occurs, where we recall that $\S_n(1)=\{j_{1,1},\ldots, j_{1,S_n(1)}\}$ (and, similarly, $\S_{n,1}(1)=\{j_{1,1},\ldots, j_{1,S_{n,1}(1)}\}$). This can only happen when vertex $1$ is selected at at least $d$ steps, so $S_n(1)\geq d$, and the coin flips associated with the first $d$ of these steps need to be heads.
				\item[$(b)$] After this streak, vertex $1$ needs to be selected at least once more, but not later than step $\ell$. Moreover, the associated coin flip at this step has to be tails to ensure that the label of vertex $1$ in the random recursive tree is at least $\ell$, by the last line of~\eqref{eq:quantdescr}. So, combined with $(a)$, $J$ needs to contain at least $d+1$ elements that are at least $\ell$, i.e.\ $|[\ell,n]\cap J|\geq d+1$. Given this, we then require the first $d$ associated coin flips to favour vertex $1$ and the remaining $|[\ell,n]\cap J|-d$ coin flips to not favour vertex $1$ at least once, i.e.\  $X_{n,\ell,1}\geq 1$, to obtain a degree at least $d$ and a label at least $\ell$.
			\end{enumerate}
			\item[$(ii)$] The required streak of $d$ coin flips favouring vertex $1$ occurs with probability $2^{-d}$, and is independent from everything else which occurs afterwards (in particular, what occurs in steps $(i)_{(b)}$ and $(iii)$). Moreover, as the coin flips are independent of the selection set, the degree of $1$ is determined by the length of the first streak of coin flips that favour $1$. So, $\d_n(1)\overset d=\min\{G,S_n(1)\}$. 
			\item[$(iii)$]  After the first streak of $d$ coin flips that favour vertex $1$, the number of remaining coin flips which do not favour vertex $1$, associated to the selection set $J$, should be at most $h$. That is, $X_{n,\ell,1}+X_{n,\ell,2}\leq h$. 
		\end{enumerate}
		Combining all of the above, we can then write,
		\be \ba
		\mathbb P({}&\mathcal E_n\,|\, \S_{n,1}(1)=J)\\
		&=\ind_{\{|[\ell,n]\cap J|\geq d+1\}}\P{\mathcal E_n\,|\, \S_{n,1}(1)=J}\qquad  &(i)&\\
		&=2^{-d}\ind_{\{|[\ell,n]\cap J|\geq d+1\}}\P{h_n(1)\leq h,\ell_n(1)\geq \ell\,|\, \S_{n,1}(1)=J,\d_n(1)\geq d} &(ii)&\\
		&=2^{-d}\ind_{\{|[\ell,n]\cap J|\geq d+1\}}\P{X_{n,\ell,1}+X_{n,\ell,2}\leq h, X_{n,\ell,1}\geq 1}, &(i)_{(b)}+(iii)&
		\ea \ee 
		where we remark that we can omit the conditioning due to the fact that the coin flips are independent of everything else.
		
		We now prove~\eqref{eq:eql}. Let us set $ \wt \CE_n:=\{h_n(1)\leq h, \ell_n(1)=\ell, \d_n(1)\leq d\}$. Again, we express the occurrence and the probability of the event $\wt \CE_n$ in terms of $J$:
		\begin{enumerate}
			\item[$(i)$] $\ell_n(1)=\ell$ and $\d_n(1)\leq d$ can only occur together if the following two things occur:
			\begin{enumerate}
				\item[$(a)$] Vertex $1$ is selected at most $d$ times in steps $n$ through $\ell+1$, and all associated coin flips favour vertex $1$. The latter occurs with probability $2^{-|J\cap [\ell+1,n]|}$.
				\item[$(b)$] Vertex $1$ is selected at step $\ell$ and is not favoured by the associated coin flip. The latter occurs with probability $1/2$.
			\end{enumerate}
			Indeed, if $(a)$ does not occur then either the degree or the label of vertex $1$ (in the random recursive tree) is too large. If $(b)$ does not occur, then the label of vertex $1$ (in the random recursive tree) is not equal to $\ell$.
			\item[$(ii)$]  In steps $\ell-1$ through $2$, the number of coin flips which do not favour vertex $1$, associated to the selection set $J$, is at most $h-1$ (since the height of $1$ equals one after step $\ell$). That is, $X_{n,\ell,2}\leq h-1$.
		\end{enumerate}
		Combining this, we can write
		\be \ba 
		\mathbb P(\wt \CE_n\,|\, \S_{n,1}(1)=J)=&\ind_{\{| [\ell+1,n]\cap J|\leq d\}}\ind_{\{\ell\in J\}}\mathbb P(\wt \CE_n\,|\, \S_{n,1}(1)=J) &&\hspace{-6pt}(i)_{(a)}+(i)_{(b)}\\
		=&\ind_{\{|[\ell+1,n]\cap J|\leq d\}}\ind_{\{\ell\in J\}}2^{-(|[\ell+1,n]\cap J|+1)}   &&\hspace{-6pt}(i)_{(a)}+(i)_{(b)}\\
		&\times \P{h_n(1)\leq h-1\,|\,\ell_n(1)=\ell, \S_{n,1}(1)=J}\\
		=&\ind_{\{|[\ell+1,n]\cap J|\leq d\}}\ind_{\{\ell\in J\}}2^{-(|[\ell+1,n]\cap J|+1)}\P{X_{n,\ell,2}\leq h-1}. &&\hspace{-6pt}(ii)
		\ea\ee 
		We remark that in the last step, as in the proof of~\eqref{eq:geql}, we can omit the conditional event $\{\ell_n(1)=\ell,\S_n(1)=J\}$, as the coin flips are independent of everything else. Moreover, in the second step we can omit the event $\{\d_n(1)\leq d\}$, as the occurrence of $\{h_n(1)\leq h-1\}$, conditionally on $\{\ell_n(1)=\ell\}$ is independent of $\{\d_n(1)\leq d\}$. This concludes the proof.
	\end{proof}
	
	We now extend this result to multiple vertices, which we can do with relative ease as long as the truncated selection sets of the vertices $1,\ldots, k$ are disjoint. For ease of writing, we define $\overline S_{n,1}:=(\S_{n,1}(i))_{i\in[k]}$ and $\bar J:=(J_i)_{i\in[k]}$ (where $J_i\subseteq \Omega_1$ for each $i\in[k]$).
	
	\begin{lemma}\label{lemma:probmultvert}
		Fix $k\in\N$ and consider a truncation sequence $(t_n)_{n\in\N}$ such that $t_n\leq n$ for all $n\in\N$. Let $h_i,d_i\in\N_0, i\in[k], (J_i)_{i\in[k]}\in \Omega_1^k$ such that the $(J_i)_{i\in[k]}$ are pairwise disjoint. Then, 
		\be\ba
		\mathbb P{}&(h_{n,1}(i)\leq h_i, \d_n(i)\geq d_i,i\in[k]\,|\,\overline \S_{n,1}=\bar J)=\prod_{i=1}^k \mathbb P(h_{n,1}(i)\leq h_i, \d_n(i)\geq d_i\,|\,\S_{n,1}(i)=J_i).
		\ea \ee 
		If, additionally, we let $\ell_i\in\Omega_1$ for all $i\in[k]$,
		\be \ba\label{eq:multgeql}
		\mathbb P{}&(h_{n,1}(i)\leq h_i, \ell_n(i)\geq \ell_i, \d_n(i)\geq d_i,i\in[k]\,|\,\overline \S_{n,1}=\bar J)\\
		&=\prod_{i=1}^k \mathbb P(h_{n,1}(i)\leq h_i, \ell_n(i)\geq \ell_i, \d_n(i)\geq d_i\,|\,\S_{n,1}(i)=J_i),
		\ea \ee
		and 
		\be \ba
		\mathbb P{}&(h_{n,1}(i)\leq h_i, \ell_n(i)= \ell_i, \d_n(i)\geq d_i,i\in[k]\,|\,\overline \S_{n,1}=\bar J)\\
		&=\prod_{i=1}^k \mathbb P((h_{n,1}(i)\leq h_i, \ell_n(i)= \ell_i, \d_n(i)\geq d_i,i\in[k]\,|\,\S_{n,1}(i)=J_i).
		\ea \ee 
	\end{lemma} 
	
	\begin{proof}
		The first result follows from~\cite[Lemma $3.1$]{Esl16}. We prove~\eqref{eq:multgeql}, the proof of the last result follows an analogous approach.
		
		The proof is similar to that of~\cite[Lemma $3.1$]{Esl16}. Let us rewrite $J_i=\{j_{i,1},\ldots, j_{i,|J_i|}\}$, where $j_{i,1}>\ldots>j_{i,J_i}$ for each $i\in[k]$. Conditionally on $\{\overline \S_{n,1}=\bar J\}$, we have that for each $i\in[k]$, 
		\be 
		h_{n,1}(i)=\sum_{m=1}^{|J_i|}h_{i,j_{i,m}}. 
		\ee 
		Also, the event $\{\d_n(i)\geq d_i\}$ holds if and only if $|J_i|\geq d_i$ and $h_{i,j_{i,m}}=0$ for all $m\in[d_i]$, and the event $\{\ell_n(i)\geq \ell_i\}$ holds if and only if $\max\{m\in J_i: h_{i,m}=0\}\geq \ell_i$. As $\ell_i\geq t_n$, it follows that the event $\{h_{n,1}(i)\leq h_i, \ell_n(i)\geq \ell_i, \d_n(i)\geq d_i, i\in[k]\}$, conditionally on $\overline \S_{n,1}=\bar J$, depends solely on $(h_{i,m})_{m\in J_i, i\in[k]}$ and $J_i$. Since the sets $(J_i)_{i\in[k]}$ are pairwise disjoint, the occurrence of the events $\{h_{n,1}(i)\leq h_i, \ell_n(i)\geq \ell_i, \d_n(i)\geq d_i\}$, for each $i\in[k]$, depend on disjoint sets of random variables. Moreover, since the random variables $h_{i,m}$ for different values of $m$ are determined by independent coin flips, we have that the events $\{h_{n,1}(i)\leq h_i, \ell_n(i)\geq \ell_i, \d_n(i)\geq d_i\}$, for each $i\in[k]$, depend on disjoint sets of independent random variables, from which~\eqref{eq:multgeql} follows. A similar reasoning proves the final result. 
	\end{proof}
	
	To end the first part of this section, we recall a result from Addario-Berry and Eslava on the degree of vertices $1,\ldots, k$ in the Kingman $n$-coalescent.
	
	\begin{proposition}[Proposition $4.2$, \cite{AddEsl18}]\label{prop:deg}
		Fix $k\in\N, c\in(0,2)$. There exists $\beta=\beta(c,k)>0$ such that uniformly over integers $(d_i)_{i\in[k]}\in[0,c\log n]^k$,
		\be 
		\P{\d_n(i)\geq d_i, i\in[k]}=2^{-\sum_{i=1}^k d_i}(1+o(n^{-\beta})).
		\ee  
	\end{proposition}
	
	\subsection{Truncated selection sets} 
	
	As we have seen in the first part of this section, we can obtain explicit formulations for the probability of events related to the size of the degree, depth, and label of vertices $1,\ldots, k$ in the Kingman $n$-coalescent, under certain conditions on the truncated selection sets $\overline \S_{n,1}$. In this part of the section, we formalize these conditions and show that they are met with high probability. We also introduce some other properties of the truncated selection sets that are useful in the analysis that follows in Sections~\ref{sec:deg} through~\ref{sec:kingproof}.
	
	Recall $\Omega_1$ from~\eqref{eq:omega} and recall that we write $\overline\S_{n,1}=(\S_{n,1}(i))_{i\in[k]},\bar J=(J_i)_{i\in[k]}$. For $\delta\in(0,2)$ and $\bar d=(d_i)_{i\in[k]}\in \Z^k$, define
	\be \ba\label{eq:AB}
	\A&:=\{\bar J\in \Omega_1^k: \P{\overline \S_{n,1}=\bar J, \d_n(i)\geq d_i,i\in[k]}>0\},\\
	\B&:=\{\bar J\in \Omega_1^k: (J_1,\ldots, J_k) \text{ are pairwise disjoint and } |\,|J_i|-2\log n|\leq \delta \log n, i\in [k]\}.
	\ea\ee 
	$\A$ consists of all possible outcomes of the truncated selection sets that enable the event $\{\d_n(i)\geq d_i, i\in [k]\}$, and $\B$ consists of all truncated selection sets which enable the decoupling of the depth, label and degree of the vertices $1,\ldots, k$, as follows from Lemma~\ref{lemma:probmultvert}. 
	
	We now present some results related to the sets $\A$ and $\B$, which are based on several results from~\cite{Esl16}. Though we defined the truncated selection sets and the truncated depth in terms of a general truncation sequence $t_n$, it suffices to consider only the case $t_n=\lceil (\log n)^2\rceil$ in the following lemmas (as we will mostly use this choice for $t_n$ in what follows).
	
	\begin{lemma}[Lemma $3.1$, \cite{Esl16}]\label{lemma:BinA}
		Let $\delta\in(0,2)$ and let $t_n=\lceil (\log n)^2\rceil$. If $\bar d=(d_i)_{i\in[k]}\in\N_0^k$ satisfies $d_i<(2-\delta)\log n$ for all $i\in[k]$, then $\B\subseteq \A$.
	\end{lemma}   
	
	We have already discussed that $\tau_k<t_n$ with high probability when the truncation sequence $t_n$ diverges as $n\to\infty$ infinity. The concentration of the size of $\S_{n,1}(i)$ around $2\log n$ for any $i\in[k]$ when $t_n=\lceil (\log n)^2\rceil$ (or, more generally, when $t_n=o(n)$, which follows from a direct application of Bernstein's inequality, see also~\cite[$(32)$]{Esl16} for a more formal statement) yields the following result:
	
	\begin{lemma}[Lemma $3.2$, \cite{Esl16}]\label{lemma:B}
		Fix an integer $k\in \N$ and $\delta\in(0,2)$ and let $t_n=\lceil (\log n)^2\rceil$. Then, 
		\be 
		\P{\overline \S_{n,1}\in \B}=1-o(1).
		\ee 
	\end{lemma} 
	
	We also know that the elements of $\overline \S_{n,1}$ are asymptotically independent for any $k\in\N$, uniformly over the set $\B$. Let $\overline \CR_{n,1}:=(\CR_{n,1}(1),\ldots, \CR_{n,1}(k))$ be $k$ independent copies of $\S_{n,1}(1)$. Then, we have the following result:
	
	\begin{lemma}[Lemma $3.2$, \cite{Esl16}]\label{lemma:StoR}
		Fix an integer $k\in\N$ and $\delta\in(0,2)$ and let $t_n=\lceil (\log n)^2\rceil$. Uniformly over $\bar J\in \B$, 
		\be
		\P{\overline \S_{n,1}=\bar J}=(1+o(1))\P{\overline \CR_{n,1}=\bar J}. 
		\ee
	\end{lemma}  
	
	The following lemma provides bounds for the decay of the tail distribution of $\tau_k$, conditionally on certain events.
	
	\begin{lemma}\label{lemma:tauk}
		Fix $k\in\N$ and recall $\tau_k$ from~\eqref{eq:tauk}. We have that $(\tau_k)_{n\in\N}$ is a tight sequence of random variables. Furthermore, fix $c\in(0,2)$ and let $(d_i)_{i\in[k]}\in \N_0^k$ such that $d_i\leq c\log n$ for all $i\in[k]$. Then, 
		\be \label{eq:taudeg}
		\P{\tau_k<\lceil (\log n)^2\rceil \,\Big|\, \d_n(i)\geq d_i, i\in[k]}\geq 1-\CO\Big(\frac{1}{\log n }\Big).
		\ee 
		Furthermore, let $(\ell_i)_{i\in[k]}\in [n]^k$ be distinct such that $\ell_i$ diverges as $n\to\infty$ for all $i\in[k]$. Then, 
		\be \label{eq:taulabel}
		\P{\tau_k<\min_{i\in[k]}\log \ell_i\, \Big|\, \ell_n(i)=\ell_i, i\in[k]}\geq 1-\CO\Big(\frac{1}{\min_{i\in[k]}\log \ell_i}\Big).
		\ee 
	\end{lemma}
	
	\begin{proof}
		We first prove the tightness of $(\tau_k)_{n\in\N}$. Fix $\eps>0$ and set $K_\eps:=\lceil2+ k^2/\eps\rceil $. We recall that in Definition~\ref{def:king}, $\{a_j,b_j\}$ denotes the two trees selected at step $j$ in the Kingman $n$-coalescent, for $2\leq j\leq n$. Also, the trees are ordered by their smallest-labelled vertex, so that $\tau_k<K_\eps$ is implied by $\{a_j,b_j\}\not\subseteq [k]$ for all $K_\eps\leq j\leq n$. Since the selection of these roots is independent at each step, we obtain
		\be 
		\P{\tau_k<K_\eps} \geq \mathbb P\bigg(\bigcap_{j=K_\eps}^n \big\{\{a_j,b_j\}\not\subseteq [k]\big\}\bigg)
		=\prod_{j=K_\eps}^n \P{\{a_j,b_j\}\not\subseteq [k]}=\prod_{j=K_\eps}^n \Big(1-\frac{k(k-1)}{j(j-1)}\Big). 
		\ee   
		We then bound the product from below to obtain the lower bound 
		\be \label{eq:prodlb}
		\P{\tau_k<K_\eps} \geq 1-\sum_{j=K_\eps}^n \frac{k^2}{(j-1)^2}\geq 1-k^2\int_{K_\eps-2}^\infty x^{-2}\,\d x=1-\frac{k^2}{K_\eps-2}\geq 1-\eps.
		\ee 
		As a result, $\P{\tau_k\geq K_\eps}\leq \eps$ for all $n\in\N$, from which the tightness follows.
		
		We then prove~\eqref{eq:taudeg} and set $s_n:=\lceil (\log n)^2\rceil $ for ease of writing. Using Bayes' theorem, the bound in~\eqref{eq:prodlb} and that $\P{\d_n(i)\geq d_i, i\in[k]}=2^{-\sum_{i=1}^k d_i}(1+o(1))$ by Proposition~\ref{prop:deg}, we obtain 
		\be\ba
		\P{\tau_k<s_n\,|\, \d_n(i)\geq d_i, i\in[k]}&=\frac{\P{\tau_k<s_n}}{\P{\d_n(i)\geq d_i, i\in[k]}}\P{\d_n(i)\geq d_i, i\in[k]\,|\, \tau_k<s_n}\\
		&=\big(1-\CO\big(s_n^{-1}\big)\big)2^{\sum_{i=1}^k d_i}\P{\d_n(i)\geq d_i, i\in[k]\,|\, \tau_k<s_n}.
		\ea\ee 
		As in the proof of Lemma~\ref{lemma:probmultvert}, the event $\{\d_n(i)\geq d_i, i\in[k]\}$ occurs when both \\$|\S_n(i)\cap [s_n,n]|\geq d_i$ holds and when the first $d_i$ associated coin flips favour vertex $i$, for all $i\in[k]$. Conditionally on $\{\tau_k<s_n\}$, we know that all these coin flips occurs at different steps for \emph{all} vertices $1,\ldots, k$, so that they are independent. Moreover, they are independent of the selection sets, so that we obtain the lower bound
		\be\ba \label{eq:unionineq}
		\P{\tau_k<s_n\,|\, \d_n(i)\geq d_i, i\in[k]}&\geq \big(1-\CO\big(s_n^{-1}\big)\big)\P{|\S_n(i)\cap [s_n, n]|\geq d_i, i\in[k]\,|\, \tau_k<s_n}\\
		&=\big(1-\CO\big(s_n^{-1}\big)\big)\mathbb P\bigg(\Big|\cup_{i=1}^k \S_n(i)\cap [s_n,n]\Big|\geq \sum_{i=1}^k d_i\bigg| \tau_k<s_n\bigg).
		\ea\ee
		Again, the last step uses the conditional event, on which we have that all $\S_n(i)\cap[s_n,n]$ are disjoint, so that $|\S_n(i)\cap [s_n,n]|\geq d_i$ for all $i\in[k]$ is equivalent to the cardinality of the union of all these sets being greater than the sum of the $d_i$. We also know, conditionally on $\{\tau_k<s_n\}$, that for every $j\in[s_n,n]$, at most one $s_{i,j}$ can equal one among all $i\in[k]$. So, for every $j\in[s_n,n]$, 
		\be 
		\P{\cup_{i=1}^k \{s_{i,j}=1\}\,|\, \tau_k<s_n}=k\P{s_{1,j}=1\,|\, \tau_k<s_n}=2k\frac{j-k}{j(j-1)-k(k-1)}=\frac{2k}{j+k-1}.
		\ee 
		Hence, if we let $(\wt s_j)_{j=s_n}^n$ be independent indicator random variables such that \\$\P{\wt s_j=1}=2k/(j+k-1)$, we can write, conditionally on $\{\tau_k<s_n\}$. 
		\be 
		\Big|\big(\cup_{i=1}^k \S_n(i)\big)\cap [s_n,n]\Big|\overset d= \sum_{j=s_n}^n \wt s_j.
		\ee 
		Since $\log s_n=o(\log n)$, it is readily checked that 
		\be
		\E{\sum_{j=s_n}^n \wt s_j}=2k(1+o(1))\log n, \qquad \Var\Big(\sum_{j=s_n}^n \wt s_j\Big)=2k(1+o(1))\log n.
		\ee 
		Again using that $d_i\leq c\log n$ for each $i\in[k]$ and all $n$ sufficiently large, where $c<2$, we obtain for some $\wt c\in (0,2-c)$ by using Chebychev's inequality,
		\be \ba
		\mathbb P{}&\bigg(\Big|\big(\cup_{i=1}^k \S_n(i)\big)\cap [s_n,n]\Big|\geq \sum_{i=1}^k d_i\,\bigg|\, \tau_k<s_n\bigg)\\
		&\geq \mathbb P\Bigg(\sum_{j=s_n}^n \wt s_j\geq ck \log n\Bigg)\\
		&\geq 1- \mathbb P\Bigg(\bigg(\sum_{j=s_n}^n \wt s_j-\mathbb E\bigg[\sum_{j=s_n}^n \wt s_j\bigg]\bigg)^2\geq (\wt ck\log n)^2\Bigg)\\
		&\geq 1-\frac{1}{(\wt ck\log n)^2}\Var\Big(\sum_{j=s_n}^n \wt s_j\Big)=1-\mathcal O\Big(\frac{1}{\log n}\Big),
		\ea\ee 
		which, combined with~\eqref{eq:unionineq}, completes the proof of~\eqref{eq:taudeg}.
		
		We now prove~\eqref{eq:taulabel} and we set $t_n=\min_{i\in[k]}\log \ell_i$ and note that $t_n$ diverges with $n$. As in the proof of~\eqref{eq:taudeg}, 
		\be\ba 
		\mathbb P\Big({}&\tau_k<t_n\,\Big|\,\ell_n(i)=\ell_i, i\in[k]\Big)\\
		&\geq \mathbb P\bigg(\bigcap_{j=t_n}^n \big\{\{a_j,b_j\}\not\subseteq [k]\big\}\,\Big|\, \ell_n(i)=\ell_i, i\in[k]\bigg)\\
		&=\prod_{j=t_n}^n \P{\{a_j,b_j\}\not\subseteq [k]\,|\, \ell_n(i)=\ell_i, i\in[k]}\\
		&\geq\!\!\!\prod_{\substack{j=t_n\\ j\neq \ell_i, i\in[k]}}^n\!\!\!\! \P{\{a_j,b_j\}\not\subseteq [k]}\prod_{j=1}^k \P{\{a_{\ell_j},b_{\ell_j}\}\not\subseteq[k]\,|\, \ell_n(i)=\ell_i, i\in[k]}.
		\ea\ee 
		Here, omitting the conditional event for $j\neq \ell_i$ for any $i\in[k]$ yields a lower bound. Indeed, for any two distinct $i,i'\in[k]$, if $j>\max\{\ell_i,\ell_{i'}\}$ then $\{a_j,b_j\}=\{i,i'\}$ cannot occur conditionally on $\ell_n(i)=\ell_i$. Furthermore, we isolate the steps $\ell_1,\ldots, \ell_k$, since the conditional event prescribes that vertex $i$ is selected at step $\ell_i$. For any $j\in[k]$, 
		\be\ba  
		\P{ \{a_{\ell_j},b_{\ell_j}\}\not \subseteq [k]\,|\, \ell_n(i)=\ell_i, i\in[k]}&=1-\P{\{a_{\ell_j},b_{\ell_j}\}\subseteq[k]\,|\,\{ a_{\ell_j}=j\}\cup\{b_{\ell_j}=j\}}\\
		&=1-\frac{k-1}{\ell_j-1}. 
		\ea \ee 
		As a result, we obtain 
		\be\ba  
		\P{\tau_k<t_n\,\Big|\,\ell_n(i)=\ell_i, i\in[k]} &\geq \prod_{\substack{j=t_n\\ j \neq \ell_i, i\in[k]}}^m \Big(1-\frac{k(k-1)}{j(j-1)}\Big)\prod_{j=1}^k \Big(1-\frac{k-1}{\ell_j-1}\Big)\\
		&\geq \prod_{\substack{j=t_n\\ j \neq \ell_i, i\in[k]}}^m \Big(1-\frac{k(k-1)}{(j-1)^2}\Big)\Big(1-\frac{k-1}{t_n}\Big)^k\\
		&=\prod_{j=t_n}^n \Big(1-\frac{k(k-1)}{(j-1)^2}\Big)\prod_{j=1}^k \Big(1-\frac{k(k-1)}{(\ell_j-1)^2}\Big)^{-1}\Big(1-\frac{k-1}{t_n}\Big)^k\\
		&\geq \prod_{j=t_n}^n \Big(1-\frac{k(k-1)}{(j-1)^2}\Big)\Big(1-\frac{k(k-1)}{(n-1)^2}\Big)^{-k}\Big(1-\frac{k-1}{t_n}\Big)^k\\
		&\geq 1-\mathcal O\big(t_n^{-1}\big),
		\ea\ee 
		where the last step follows from~\eqref{eq:prodlb}, and which concludes the proof.
	\end{proof}
	
	Beyond the sets $\A$ and $\B$ and the random variable $\tau_k$, we also want to control of the probability of the events $\{\ell_n(i)=\ell_i, i\in[k]\}$ and $\{\d_n(i)\geq d_i, i\in[k]\}$ conditionally on the truncated selection sets $\overline S_{n,1}$. To this end, we define, for $\bar\ell:=(\ell_i)_{\in[k]}\in \N^k$,
	\be \ba 
	\wtA&:=\{\bar J\in \Omega_1^k: \P{\ell_n(i)=\ell_i, i\in[k], \overline \S_{n,1}=\bar J}>0\},\\
	\wtB&:=\{\bar J\in \Omega_1^k: (J_1, \ldots, J_k)\text{ are pairwise disjoint}\}.
	\ea \ee 
	We then have the following lemma, which is partially covered by~\cite[Lemma $3.1$]{Esl16}.
	
	\begin{lemma}\label{lemma:labeldegprob}
		Fix $k\in\N$ and let $(\ell_i)_{i\in[k]}\in [n]^k$ such that $\ell_i\neq \ell_j$ when $i\neq j$. Then, 
		\be \label{eq:labelprob}
		\P{\ell_n(i)=\ell_i, i\in[k]}=\frac{1}{(n)_k}.
		\ee 
		Also, when the truncation sequence $t_n$ diverges with $n$,
		\be\label{eq:wtbc} 
		\P{\ell_n(i)=\ell_i, i\in[k], \overline \S_{n,1}\in \wtB^c}=o(n^{-k}).
		\ee 	
		Finally, let $(d_i)_{i\in[k]},\N_0^k$ and let $t_n=\lceil (\log n)^2\rceil$. If $\bar J\in \A$, 
		\be \label{eq:degprob}
		\P{\d_n(i)\geq d_i, i\in[k]\,|\, \overline{\S}_{n,1}=\bar J}=2^{-\sum_{i=1}^k d_i}.
		\ee
	\end{lemma} 
	
	\begin{proof}
		The first result in~\eqref{eq:labelprob} follows from Corollary~\ref{cor:unifnodes}, as each vertex obtains a uniform label from $[n]$ after the relabelling of the final tree $F_1$ in the Kingman $n$-coalescent and all $\ell_i$ are distinct.
		
		To prove~\eqref{eq:wtbc}, we write 
		\be \ba  
		\P{\ell_n(i)=\ell_i, i\in[k], \overline \S_{n,1}\in \wt{ \mathcal B}_n^c}&=\P{\overline \S_{n,1}\in \wt{ \mathcal B}_n^c\,\Big|\,\ell_n(i)=\ell_i, i\in[k]}\P{\ell_n(i)=\ell_i, i\in[k]}\\
		&=\P{\overline \S_{n,1}\in \wt{ \mathcal B}_n^c\,\Big|\,\ell_n(i)=\ell_i, i\in[k]}\frac{1}{(n)_k},
		\ea \ee
		where the last step follows from~\eqref{eq:labelprob}. It thus remains to argue to that probability on the right-hand side is $o(1)$. For $\overline \S_{n,1}\in \wtB^c$ to hold, the truncated selection sets should overlap at some step $t_n\leq j\leq n$, i.e.\ $\tau_k\geq t_n$ should hold. Conditionally on the event $\{\ell_n(i)=\ell_i,i\in[k]\}$, however, the truncated selection sets in $\overline\S_{n,1}$ cannot overlap at certain steps. Namely, for $j>\max_{i\in[k]}\ell_i$, $j\in \S_{n,1}(i)$ can hold for at most one $i\in[k]$. Indeed, if the converse would be the case, i.e.\ $j\in \S_{n,1}(i)$ and $j\in \S_{n,1}(i')$ for some distinct $i,i'\in[k]$, then one of the vertices $i,i'$, let us assume this is vertex $i$, would lose the associated coin flip at step $j$ and hence its label in the random recursive tree would be $j>\ell_i$. This clearly contradicts the conditional event. As a result, on the conditional event $\{\ell_n(i)=\ell_i, i\in[k]\}$, the event $\{\overline \S_{n,1}\in\wtB^c\}$ implies that $\{t_n\leq \tau_k\leq \max_{i\in[k]}\ell_i\}$ holds. Hence, by Lemma~\ref{lemma:tauk} and since $t_n$ diverges with $n$,
		\be 
		\P{\overline \S_{n,1}\in \wt{ \mathcal B}_n^c\,\Big|\,\ell_n(i)=\ell_i, i\in[k]}\leq \mathbb P\Big(t_n\leq \tau_k\leq \max_{i\in[k]}\ell_i\Big)\leq \P{\tau_k\geq t_n}=o(1),
		\ee 
		as desired.
		
		The final result in~\eqref{eq:degprob} is proved in~\cite[Lemma $3.2$]{Esl16}.
	\end{proof}
	
	In Lemma~\ref{lemma:probmultvert} we saw that, as long as the truncated selection sets $(\S_{n,1}(i))_{i\in[k]}$ are pairwise disjoint, then the events $\{h_{n,1}(i)\leq h_i, \ell_n(i)\geq \ell_i, \d_n(i)\geq d_i\}$, $i\in[k]$, are independent, conditionally on $\overline \S_{n,1}$. Furthermore, when the truncation sequence $t_n$ diverges as $n\to\infty$, we already observed that the event $\{\tau_k<t_n\}$ holds with high probability by Lemma~\ref{lemma:tauk}, which implies that the $(\S_{n,1}(i))_{i\in[k]}$ are disjoint.
	
	On the other hand, we use the truncated depths $(h_{n,1}(i))_{i\in[k]}$ merely for technical reasons, and are really interested in the depths $(h_n(i))_{i\in[k]}$. As a result, choosing a truncation sequence $(t_n)_{n\in\N}$ that diverges with $n$ `too quickly', may lead to different behaviour of $h_{n,1}(1)$ compared to $h_n(1)$. In other words, if $t_n$ grows `too quickly', then $h_{n,2}(1)=h_n(1)-h_{n,1}(1)$ might become `too large'. In the following lemma we make this informal statement more precise and provide constraints on $t_n$ to avoid such discrepancies between $h_n(1)$ and $h_{n,1}(1)$.  
	
	\begin{lemma}[Partially from Lemma $2.7$, \cite{Esl16}]\label{lemma:h2n}
		Fix $k\in\N$ and $c\in(0,2)$. If $d_i\leq c\log n$ for all $i\in[k]$ and $t_n=\lceil (\log n)^2\rceil$, then for any $j\in[k]$ and any $\eps>0$, 
		\be \label{eq:h2ndeg}
		\lim_{n\to\infty}\P{h_{n,2}(j)\geq \eps\sqrt{\log n}\,\Big|\, \d_n(i)\geq d_i, i\in[k]}=0.
		\ee 
		Furthermore, let $(\ell_i)_{i\in[k]}\in [n]^k$ be $k$ distinct integers that diverge as $n\to\infty$. If $\log t_n=o(\min_{i\in[k]}\sqrt{\log \ell_i})$, then for any $j\in[k]$ and any $\eps>0$, 
		\be  \label{eq:h2nlab}
		\lim_{n\to\infty} \P{h_{n,2}(j)\geq \eps \sqrt{\log \ell_j}\,\Big|\, \ell_n(i)=\ell_i, i\in[k]}=0.
		\ee 
	\end{lemma} 
	
	\begin{remark} \label{rem:h2n}
		Assume the truncation sequence $(t_n)_{n\in\N}$ satisfies the assumptions of Lemma~\ref{lemma:h2n}. As a direct consequence of Lemma~\ref{lemma:h2n}, the limiting distributions of 
		\be 
		\frac{h_n(j)-(\log n-d_j/2)}{\sqrt{\log n-d_j/4}}\quad\text{and}\quad \frac{h_{n,1}(j)-(\log n-d_j/2)}{\sqrt{\log n-d_j/4}}
		\ee 
		conditionally on the event $\d_n(i)\geq d_i$ for al $i\in[k]$, with $d_i\leq c\log n$ for all $i\in[k]$, are identical (assuming they both exist), for any $j\in[k]$. This follows from Slutsky's theorem~\cite[Lemma $2.8$]{Vaart00} and since $\sqrt{\log n-d_j/4}=\Theta(\sqrt{\log n})$. Similarly, conditionally on $\ell_n(i)=\ell_i, i\in[k]$ (where the $\ell_i$ diverge as $n\to\infty$), the limiting distributions of 
		\be
		\frac{h_n(j)-\log \ell_j}{\sqrt{\log \ell_j}}\quad \text{and} \quad \frac{h_{n,1}(j)-\log \ell_j}{\sqrt{\log \ell_j}}
		\ee
		are identical (assuming they exist), for any $j\in[k]$. 
	\end{remark} 
	
	\begin{proof} 
		The result in~\eqref{eq:h2ndeg} follows from~\cite[Lemma $2.7$]{Esl16}. 
		
		To prove~\eqref{eq:h2nlab}, we consider $j=1$ only by the exchangeability of the vertices. We first note that $t_n\leq \min_{i\in[k]}\ell_i$ by the assumption on $t_n$ and since the $\ell_i$ diverge with $n$. As a result, the event $\{\ell_n(i)=\ell_i, i\in[k]\}$ is solely dependent on the truncated selection sets $\overline \S_{n,1}$ and the associated coin flips of the truncated selection sets, whereas $h_{n,2}(1)$ is determined by the set $\S_n(1)\cap [2,t_n-1]$ and its associated coin flips. It thus follows that $h_{n,2}(1)$ is independent of the event $\{\ell_n(i)=\ell_i, i\in[k]\}$. The result then follows from the Markov inequality and by the assumption on $t_n$, as 
		\be 
		\mathbb P(h_{n,2}(1)\geq \eps\sqrt{\log \ell_1})\leq \frac{\E{h_{n,2}(1)}}{\eps\sqrt{\log \ell_1}}=\frac{1}{\eps\sqrt{\log \ell_1}}\sum_{j=2}^{t_n-1}\frac 1j=\mathcal O \Big(\frac{\log t_n}{\sqrt{\log \ell_1}}\Big)=o(1),
		\ee 
		by the assumptions on $t_n$, which concludes the proof.
	\end{proof}	
	
	\section{Joint properties of high-degree vertices} \label{sec:deg}
	
	In this section we use the preliminary results proved in Section~\ref{sec:pre} to study the joint behaviour of the depth and label of high-degree vertices.
	
	We set 
	\be\label{eq:hld}
	h:=(\log n-d/2)+y\sqrt{\log n-d/4},\quad \ell:=n\exp(-d/2+x\sqrt{d/4}),\quad t_n:=\lceil (\log n)^2\rceil,
	\ee
	with $x,y\in \R$.  We then have the following result.
	
	\begin{proposition}\label{prop:limit}
		Fix $a\in [0,2)$, let $h,\ell$ and $t_n$ be as in~\eqref{eq:hld}, with $d\in \N_0$, and let $M$ and $N$ be two independent standard normal random variables. When $\limsup_{n\to\infty}d<\infty$,
		\be \label{eq:reswithoutlab}
		\lim_{n\to\infty}\P{h_{n,1}(1)\leq h\,|\, \d_n(1)\geq d}=\Phi(y).
		\ee 
		When, instead, $d$ diverges as $n\to\infty$ such that $\lim_{n\to\infty} d/\log n=a$, 
		\be \ba\label{eq:reswithlab}
		\lim_{n\to\infty}\P{h_{n,1}(1)\leq h, \ell_n(1)\geq \ell\,|\, \d_n(1)\geq d}=\P{M\sqrt{\frac{a}{4-a}}+N\sqrt{1-\frac{a}{4-a}}\leq y, M> x}.
		\ea\ee
	\end{proposition}
	
	\begin{remark}\label{rem:limit}
		$(i)$ As $M\sqrt{a/(4-a)}+N\sqrt{1-a/(4-a)}\sim \CN(0,1)$, the result in~\eqref{eq:reswithlab}, when omitting the event $\ell_n(1)\geq \ell$ (or, equivalently, letting $x\to-\infty$ and setting $a=0$), yields 
		\be 
		\lim_{n\to\infty} \P{h_{n,1}(1)\leq h\,|\, \d_n(1)\geq d}=\Phi(y),
		\ee 
		and hence complements the result in~\eqref{eq:reswithoutlab} in the case that $d$ diverges with $n$ such that $d\leq c\log n$ for some $c\in(0,2)$. Together, they are a generalisation of~\cite[Lemma $2.5$]{Esl16}, where only a parametrised version with $d=\lfloor a\log n\rfloor +b$ and $a\in[0,2),b\in \Z$ is considered.
		
		$(ii)$ Combined with Lemma~\ref{lemma:h2n} and Remark~\ref{rem:h2n}, we obtain that the results in Proposition~\ref{prop:limit} hold when we substitute $h_n(1)$ for $h_{n,1}(1)$ as well.
	\end{remark} 
	
	\begin{proof} 
		We first prove~\eqref{eq:reswithlab} and briefly discuss how to prove~\eqref{eq:reswithoutlab} using~\cite[Lemma $2.5$]{Esl16} at the end. In the setting of~\eqref{eq:reswithlab}, we recall that we assume that $d$ diverges as $n\to\infty$ and $h,\ell$ and $t_n$ are as in~\eqref{eq:hld}. 
		
		Take $c\in(a,2)$. Recall that, conditionally on $\S_{n,1}(1)$, $X_{n,\ell,1}\sim \text{Bin}(|[\ell,n]\cap \S_{n,1}(1)|-d,1/2)$ and $X_{n,\ell,2}\sim \text{Bin}(|[t_n,\ell-1]\cap \S_{n,1}(1)|,1/2)$ (where we set $X_{n,\ell,1}=0, X_{n,\ell,2}=0$ when \\ $|[\ell,n]\cap \S_{n,1}(1)|-d\leq 0$ and $|[t_n,\ell-1]\cap \S_{n,1}(1)|=0$, respectively). By Lemma~\ref{lemma:probonevert}, 
		\be \ba 
		\mathbb P{}&(h_{n,1}\leq h, \ell_n(1)\geq \ell\,|\, \d_n(1)\geq d)\\
		&=\frac{\P{h_{n,1}\leq h, \ell_n(1)\geq \ell,\d_n(1)\geq d}}{\P{\d_n(1)\geq d}}\\
		&=\frac{1}{\P{\d_n(1)\geq d}2^d}\E{\ind_{\{|[\ell,n]\cap \S_{n,1}(1)|\geq d+1\}}\P{X_{n,\ell,1}+X_{n,\ell,2}\leq h, X_{n,\ell,1}\geq 1\,|\, \S_{n,1}(1)}}
		\ea \ee 
		Since $d$ diverges with $n$ and $d\leq c\log n$ (with $c\in (a,2)$), 
		\be 
		\ell=n\exp(-d/2+x\sqrt{d/4})>\lceil (\log n)^2\rceil=t_n
		\ee 
		for all $n$ large. As a result, $X_{n,\ell,2}$ is non-zero with positive probability. 
		
		Since $\P{\d_n(1)\geq d}2^d=1+o(1)$ by Proposition~\ref{prop:deg}, we obtain 
		\be \ba
		\mathbb P({}&h_{n,1}\leq h, \ell_n(1)\geq \ell\,|\, \d_n(1)\geq d)\\
		&=(1+o(1))\E{\ind_{\{|[\ell,n]\cap \S_{n,1}(1)|\geq d+1\}}\P{X_{n,\ell,1}+X_{n,\ell,2}\leq h, X_{n,\ell,1}\geq 1\,|\, \S_{n,1}(1)}}.
		\ea\ee 
		To prove the expected value has the desired limit, we start by rewriting the binomial random variables $X_{n,\ell,1}$ and $X_{n,\ell,2}$. Let $(I_j^n)_{j\in[n],n\in\N}$, $(\wt I_j^n)_{j\in[n], n\in\N}$ be two i.i.d.\ sequences of independent Bernoulli$(1/2)$ random variables. Finally, let $Q_n:=|[\ell,n]\cap \S_{n,1}(1)|$, $\wt Q_n:=|[t_n, \ell-1]\cap \S_{n,1}(1)|$, independent of the $I_j^n, \wt I_j^n$. Then,  
		\be \label{eq:binomassum}
		X_{n,\ell,1}:=\sum_{j=1}^{Q_n-d}I_j^{Q_n-d},\qquad X_{n,\ell,2}:=\sum_{j=1}^{\wt Q_n}\wt I_j^{\wt Q_n}.
		\ee
		Here, we set $X_{n,\ell,1}=0, X_{n,\ell,2}=0$ if $Q_n-d\leq 0,\wt Q_n=0$, respectively. Notice that $Q_n$ and $\wt Q_n$ are independent, that they can be determined from $\S_{n,1}(1)$ and that the values of the $I_j^n,\wt I_j^n$ are independent of $\S_{n,1}(1)$, so that conditioning on $\S_{n,1}(1)$ is equivalent to conditioning on $Q_n,\wt Q_n$. We can then write the expected value in the statement of the proposition as
		\be\ba\label{eq:expprob}
		\mathbb E{}&\left[\ind_{\{Q_n\geq d+1\}}\P{\sum_{j=1}^{\wt Q_n}\wt I_j^{\wt Q_n}+\sum_{j=1}^{Q_n-d}I_j^{Q_n-d}\leq h, \sum_{j=1}^{Q_n-d}I_i^{Q_n-d}\geq 1\,\bigg|\, Q_n,\wt Q_n}\right]\\
		&= \P{\sum_{j=1}^{\wt Q_n}\wt I_j^{\wt Q_n}+\sum_{j=1}^{(Q_n-d)\ind_{\{Q_n-d\geq 1\}}}I_j^{Q_n-d}\leq h, \sum_{j=1}^{(Q_n-d)\ind_{\{Q_n-d\geq 1\}}}I_j^{Q_n-d}\geq 1}.
		\ea\ee 
		The second line follows from the fact that, by changing the upper limits of the second and third sum in the probability on the first line from $Q_n-d$ to $(Q_n-d)\ind_{\{Q_n-d\geq 1\}}$, we can remove the indicator in the expected value. Indeed, if $Q_n\leq d$, then $\ind_{\{Q_n-d\geq 1\}}=0$, and hence the second event in the probability cannot occur almost surely, so that the probability is zero. As a result, the indicator in the expected value is redundant. We thus obtain
		\be\ba \label{eq:expprob2}
		\mathbb P\Bigg({}&\sum_{j=1}^{\wt Q_n}\wt I_j^{\wt Q_n}+\sum_{j=1}^{(Q_n-d)\ind_{\{Q_n-d\geq 1\}}}I_j^{Q_n-d}\leq h\Bigg)-\mathbb P\Bigg( \sum_{j=1}^{\wt Q_n}\wt I_j^{\wt Q_n}\leq h, \sum_{j=1}^{(Q_n-d)\ind_{\{Q_n-d\geq 1\}}}I_j^{Q_n-d}=0\Bigg)\\
		={}&\mathbb P\Bigg(\sum_{j=1}^{\wt Q_n}\wt I_j^{\wt Q_n}+\sum_{j=1}^{(Q_n-d)\ind_{\{Q_n-d\geq 1\}}}I_j^{Q_n-d}\leq h\Bigg)\\
		&-\mathbb P\Bigg(\sum_{j=1}^{\wt Q_n}\wt I_j^{\wt Q_n}\leq h\Bigg)\mathbb P\Bigg(\sum_{j=1}^{(Q_n-d)\ind_{\{Q_n-d\geq 1\}}}I_j^{Q_n-d}=0\Bigg),
		\ea\ee
		where the second step follows from the independence of the two sums in the second probability on the first line. The event
		\be 
		\Bigg\{\sum_{j=1}^{(Q_n-d)\ind_{\{Q_n-d\geq 1\}}}I_j^{Q_n-d} =0\Bigg\}
		\ee 
		occurs either when $Q_n\leq d$ or when, given $Q_n\geq d+1$, $I_1^{Q_n-d}=\ldots= I_{Q_n-d}^{Q_n-d}=0$. Hence,
		\be 
		\mathbb P\Bigg(\sum_{j=1}^{(Q_n-d)\ind_{\{Q_n-d\geq 1\}}}I_j^{Q_n-d}=0\Bigg)=\P{Q_n\leq d}+\E{\ind_{\{Q_n\geq d+1\}}2^{-(Q_n-d)}}.
		\ee
		Combining this with~\eqref{eq:expprob2} yields 
		\be\ba\label{eq:finterm}
		\mathbb P\Bigg({}&\sum_{j=1}^{\wt Q_n}\wt I_j^{\wt Q_n}+\sum_{j=1}^{(Q_n-d)\ind_{\{Q_n-d\geq 1\}}}I_j^{Q_n-d}\leq h\Bigg)-\mathbb P\Bigg(\sum_{j=1}^{\wt Q_n}\wt I_j^{\wt Q_n}\leq h\Bigg)\P{Q_n\leq d}\\
		&+\mathcal O\Big(\E{\ind_{\{Q_n\geq d+1\}}2^{-(Q_n-d)}}\Big).
		\ea \ee
		What remains is to show that the first two terms yield the desired limit and that the last term is negligible compared to the first two. Let us start with the former and tackle the product of two probabilities on the first line. It follows from Lindeberg's conditions~\cite[Theorem $3.4.5$]{Dur19} that 
		\be\label{eq:normconv} 
		\frac{Q_n-\E{Q_n}}{\sqrt{\Var(Q_n)}}\toindis N, \qquad \frac{\wt Q_n-\mathbb E[\wt Q_n]}{\sqrt{\Var(\wt Q_n)}}\toindis \wt N,
		\ee 
		with $N,\wt N\sim \mathcal N(0,1)$ independent standard normal random variables, as we recall that $Q_n$ and $\wt Q_n$ are sums of independent Bernoulli random variables. It is readily checked that by the choice of $\ell$ in~\eqref{eq:hld} and since $d$ diverges with $n$,
		\be\ba  \label{eq:expvar}
		\E{Q_n}&=\sum_{j=\ell}^n \frac2j =2\log(n/\ell)+\mathcal O(1)=d-x\sqrt{d}(1+o(1)), \\
		\Var(Q_n)&=\sum_{j=\ell}^n \frac2j\Big(1-\frac2j\Big)=d-x\sqrt{d}(1+o(1)), \ea\ee 
		and, by the choice of $\ell, d$ and $t_n$,
		\be\ba\label{eq:expvartilde}
		\mathbb E[\wt Q_n]&=\sum_{j=t_n}^{\ell-1}\frac2j=2\log n-d+x\sqrt{d}-(1+o(1))\log\log n, \\
		\Var(\wt Q_n)&=\sum_{j=t_n}^{\ell-1}\frac2j\Big(1-\frac2j\Big)=2\log n-d+x\sqrt{d}-(1+o(1))\log\log n.
		\ea \ee  
		By~\eqref{eq:normconv} and~\eqref{eq:expvar} we thus obtain that
		\be \ba \label{eq:qnlim}
		\P{Q_n\leq d}=\P{\frac{Q_n-\E{Q_n}}{\sqrt{\Var(Q_n)}}\leq \frac{d-\E{Q_n}}{\sqrt{\Var(Q_n)}}}=\P{\frac{Q_n-\E{Q_n}}{\sqrt{\Var(Q_n)}}\leq \frac{x\sqrt{d }(1+o(1))}{\sqrt{d}(1+o(1))}}, 
		\ea\ee 
		which converges to $\Phi(x)$, where we recall that $\Phi:\R\to (0,1)$ denotes the cumulative density function of a standard normal distribution. By Skorokhod's representation theorem~\cite[Theorem $6.7$]{Bill99} there exists a probability space and a coupling of $(Q_n)_{n\in\N}, (\wt Q_n)_{n\in\N}$ and $(I_j^n)_{j\in[n], n\in\N}, (\wt I_j^n)_{j\in[n], n\in\N}$ such that the collections  $(I_j^n)_{j\in\N}, (\wt I_j^n)_{j\in\N}$ are independent of $Q_n$ and $\wt Q_n$ and the convergence in~\eqref{eq:normconv} is almost sure rather than in distribution. In particular, $Q_n/d\toas 1$, $\wt Q_n/(2\log n-d)\toas 1$ and $Q_n,\wt Q_n\toas \infty$. Moreover, it also follows from this representation that 
		\be 
		\frac{2\sum_{j=1}^n I_j^n-n}{\sqrt n}\toas N', \qquad \frac{2\sum_{j=1}^n\wt I_j^n-n}{\sqrt n}\toas N'',
		\ee
		as $n\to\infty$ as well, where $N',N''$ are independent standard normal random variables, also independent of $N,\wt N$ in~\eqref{eq:normconv}. We then rewrite  
		\be \ba \label{eq:qntilde}
		\frac{2\sum_{j=1}^{\wt Q_n}\wt I_j^{\wt Q_n}-(2\log n-d)}{\sqrt{4\log n-d}}={}&\frac{2\sum_{j=1}^{\wt Q_n}\wt I_j^{\wt Q_n}-\wt Q_n}{\sqrt{\wt Q_n}}\sqrt{\frac{\wt Q_n}{2\log n-d}}\sqrt{\frac{2\log n-d}{4\log n-d}}\\
		&+\frac{\wt Q_n-\mathbb E[\wt Q_n]}{\sqrt{\Var(\wt Q_n)}}\sqrt{\frac{\Var(\wt Q_n)}{2\log n-d}}\sqrt{\frac{2\log n-d}{4\log n -d}}\\
		&+\frac{\mathbb E[\wt Q_n]-(2\log n-d)}{\sqrt{d}}\sqrt{\frac{d}{4\log n-d}}.
		\ea \ee 
		Combining this with the Skorokhod representation, the fact that $d/\log n\to a$ and~\eqref{eq:expvartilde}, yields
		\be 
		\frac{2\sum_{j=1}^{\wt Q_n}\wt I_j^{\wt Q_n}-(2\log n-d)}{\sqrt{4\log n-d}}\toindis N_1\sqrt{\frac{2-a}{4-a}}+N_2\sqrt{\frac{2-a}{4-a}}+x\sqrt{\frac{a}{4-a}},
		\ee
		where $N_1,N_2$ are independent standard normal random variables. Combining this with~\eqref{eq:qnlim} and using that $h=\log n-d/2+y\sqrt{\log n-d/4}$, we obtain
		\be \ba \label{eq:indeplim}
		\lim_{n\to\infty}\P{Q_n\leq d}\P{\sum_{j=1}^{\wt Q_n}\wt I_j^{\wt Q_n}\leq h}&=\Phi(x)\P{N_1\sqrt{\frac{2-a}{4-a}}+N_2\sqrt{\frac{2-a}{4-a}}+x\sqrt{\frac{a}{4-a}}\leq y}\\
		&=\Phi(x)\P{N\sqrt{1-\frac{a}{4-a}}+x\sqrt{\frac{a}{4-a}}\leq y},
		\ea \ee 
		where $ N$ is again a standard normal random variable. This deals with the second term of~\eqref{eq:finterm}. For the first term, we observe that 
		\be 
		\P{\frac{(Q_n-d)\ind_{\{Q_n-d\geq 1\}}}{\sqrt{d}}=0}=\P{Q_n\leq d}\to\Phi(x),
		\ee 
		as $n\to\infty$ by~\eqref{eq:qnlim}, and similarly for $z\geq 0$, 
		\be 
		\P{\frac{(Q_n-d)\ind_{\{Q_n-d\geq 1\}}}{\sqrt{d}}> z}=\P{\frac{Q_n-\E{Q_n}}{\sqrt{\Var(Q_n)}}> \frac{d-\E{Q_n}+z\sqrt{d}}{\sqrt{\Var(Q_n)}}}\to 1-\Phi(x+z),
		\ee
		as $n\to\infty$. Hence, for $x\in\R$ fixed, let us define a random variable $M_x:=\ind_{\{M>x\}}(M-x)$, where $M$ is a standard normal random variable. It then follows that $\P{M_x=0}=\Phi(x)$ and $ \P{M_x>z}=\P{M>x+z}=1-\Phi(x+z),z>0$, so that
		\be \label{eq:mxlim}
		\frac{(Q_n-d)\ind_{\{Q_n-d\geq 1\}}}{\sqrt{d}}\toindis M_x.
		\ee 
		By the independence of the Bernoulli random variables $I_j^n,\wt I_j^n$, we can relabel them as a sequence of i.i.d.\ random variables. If we set $O_n:=\wt Q_n+(Q_n-d)\ind_{\{Q_n-d\geq 1\}}$, then we can write them as $( \hat I_j^{O_n})_{j\in[O_n]}$, with $\hat I^{O_n}_j:=\wt I_j^{\wt Q_n}$ if $1\leq j\leq \wt Q_n$ and $\hat I^{O_n}_j:=I_{j-\wt Q_n}^{Q_n-d}$ if $\wt Q_n+1\leq j\leq \wt Q_n+(Q_n-d)\ind_{\{Q_n-d\geq 1\}}$. Again following Lindeberg's conditions, we find that 
		\be 
		\frac{2\sum_{j=1}^n\hat I^n_j -n}{\sqrt n}\toindis N',
		\ee
		where $N'$ is a standard normal random variable. Moreover, $O_n/(2\log n-d)\toinp 1$ by combining~\eqref{eq:normconv}, \eqref{eq:expvartilde} and~\eqref{eq:mxlim}. We can then write
		\be\ba 
		{}&\hspace{-15pt}\frac{2\sum_{j=1}^{\wt Q_n}\wt I_j^{\wt Q_n}+2\sum_{j=1}^{(Q_n-d)\ind_{\{Q_n-d\geq 1\}}}I_j^{Q_n-d}-(2\log n-d)}{\sqrt{4\log n-d}}\\
		={}&\frac{2\sum_{j=1}^{O_n}\hat I^{O_n}_j-O_n}{\sqrt{O_n}}\sqrt{\frac{O_n}{2\log n-d}}\sqrt{\frac{2\log n-d}{4\log n-d}}+\frac{\wt Q_n-\mathbb E[\wt Q_n]}{\sqrt{\Var(\wt Q_n)}}\sqrt{\frac{\Var(\wt Q_n)}{2\log n-d}}\sqrt{\frac{2\log n-d}{4\log n-d}}\\
		&+\frac{(Q_n-d)\ind_{\{Q_n-d\geq 1\}}}{\sqrt{d}}\sqrt{\frac{d}{4\log n-d}}+\frac{\mathbb E[\wt Q_n]-(2\log n-d)}{\sqrt{d}}\sqrt{\frac{d}{4\log n-d}}.
		\ea\ee
		If we let $N,N', N''$ be i.i.d.\ standard normal random variables, independent of $M_x$, and use similar steps as in~\eqref{eq:mxlim} and~\eqref{eq:qntilde} (in particular using the Skorokhod representation for the random variables $(\hat I_j^n)_{j\in[n]}, O_n, (Q_n-d)\ind_{\{Q_n-d\geq 1\}}$ and that $d/\log n\to a$), this converges in distribution to
		\be\ba 
		N'{}&\sqrt{\frac{2-a}{4-a}}+N''\sqrt{\frac{2-a}{4-a}}+M_x\sqrt{\frac{a}{4-a}}+x\sqrt{\frac{a}{4-a}}\\
		&\overset d=  N\sqrt{1-\frac{a}{4-a}}+M_x\sqrt{\frac{a}{4-a}}+x\sqrt{\frac{a}{4-a}},
		\ea\ee 
		Combining this with~\eqref{eq:indeplim} in~\eqref{eq:finterm} yields
		\be \ba\label{eq:main}
		\lim_{n\to\infty}{}&\Bigg[\P{\sum_{j=1}^{\wt Q_n}\wt I_j^{\wt Q_n}+\sum_{j=1}^{(Q_n-d)\ind_{\{Q_n-d\geq 1\}}}I_j^{Q_n-d}\leq h}-\P{Q_n\leq d}\P{\sum_{j=1}^{\wt Q_n}\wt I_j^{\wt Q_n}\leq h}\Bigg]\\
		={}&\P{ M_x\sqrt{\frac{a}{4-a}}+x\sqrt{\frac{a}{4-a}}+N\sqrt{1-\frac{a}{4-a}}\leq y}\\
		&-\Phi(x)\P{N\sqrt{1-\frac{a}{4-a}}+x\sqrt{\frac{a}{4-a}}\leq y}.
		\ea\ee
		By intersecting the event in the first probability on the right-hand side with the complementary events $\{M_x=0\}, \{M_x>0\},$ and using that $M_x$ is independent of $N$, we arrive at 
		\be 
		\P{M_x\sqrt{\frac{a}{4-a}}+x\sqrt{\frac{a}{4-a}}+N\sqrt{1-\frac{a}{4-a}}\leq y, M_x>0}.
		\ee 
		By the definition of $M_x$, it follows that  the event $\{M_x>0\}$ is equivalent to $\{M>x\}$, where we recall that $M$ is a standard normal random variable. Moreover, on the event $\{M_x>0\}=\{M>x\}$, $M_x+x=\ind_{\{M>x\}}(M-x)+x=M$. Thus, we obtain
		\be \label{eq:limit}
		\P{M\sqrt{\frac{a}{4-a}}+N\sqrt{1-\frac{a}{4-a}}\leq y, M>x},
		\ee  
		as desired. Finally, we show that 
		\be \label{eq:neg}
		\lim_{n\to\infty}\E{\ind_{\{Q_n\geq d+1\}}2^{-(Q_n-d)}}=0.
		\ee 
		By splitting the expected value into the cases where $Q_n$ is at most $d+1+\lfloor d^{1/2-\eta}\rfloor$ and at least $d+1+\lceil d^{1/2-\eta}\rceil$, respectively, for some $\eta\in(0,1/2)$, we obtain
		\be \ba
		\mathbb E\Big[\ind_{\{Q_n\geq d+1\}}2^{-(Q_n-d)}\Big]
		={}&  \sum_{m=d+1}^{d+1+\lfloor d^{1/2-\eta}\rfloor}\!\!\!\!\!\!\!\!\!\P{Q_n=m}2^{-(m-d)} +\!\!\!\!\sum_{m\geq d+1+\lceil d^{1/2-\eta}\rceil}\!\!\!\!\!\!\!\!\!\P{Q_n=m}2^{-(m-d)}\\
		\leq{}& \sum_{m=d+1}^{d+1+\lfloor d^{1/2-\eta}\rfloor}\!\!\!\!\!\!\!\!\!\P{Q_n=m}+\!\!\!\!\sum_{m\geq d+1+\lceil d^{1/2-\eta}\rceil} \!\!\!\!\!\!\!\!\!\P{Q_n=m}2^{-d^{1/2-\eta}}\\
		\leq{}&\P{d+1\leq Q_n\leq d+1+\lfloor d^{1/2-\eta}\rfloor}+2^{-d^{1/2-\eta}}.
		\ea \ee 
		Since $d^{1/2-\eta}=o\big(\sqrt{\Var(Q_n)}\big)$ (see~\eqref{eq:expvar}), it follows from~\eqref{eq:normconv} that the probability in the last line converges to zero. This proves~\eqref{eq:neg}, and combining this with the limit~\eqref{eq:limit} of the left-hand side of~\eqref{eq:main} in~\eqref{eq:finterm} yields the desired result and concludes the proof of~\eqref{eq:reswithlab}.
		
		We now discuss the the proof of~\eqref{eq:reswithoutlab}. We recall that now $L:=\limsup_{n\to\infty}d<\infty$. Also, conditionally on $\S_{n,1}(1)$, let $X_n=X_n(d)\sim \text{Bin}(|\S_{n,1}(1)|-d,1/2)$ (where we set $X_n=0$ when $|\S_{n,1}(1)|-d\leq 0$) and let us define $h':=\log n+y\sqrt{\log n}$. Note that  $(h-h')/\sqrt{\log n}=o(1)$ since $L<\infty$, so that using $h'$ instead of $h$ yields the same result. Again using Lemma~\ref{lemma:probonevert} and Proposition~\ref{prop:deg}, we obtain 
		\be \ba 
		\P{h_{n,1}(1)> h'\,|\, \d_n(1)\geq d}&=\frac{\P{h_{n,1}(1)> h', \d_n(1)\geq d}}{\P{\d_n(1)\geq d}}\\
		&=(1+o(1))\mathbb E\big[\ind_{\{|\S_{n,1}(1)|\geq d\}}\P{X_n> h'\,|\, S_{n,1}(1)}\big].
		\ea \ee 
		Notice that, for any $\S_{n,1}(1)\subseteq\Omega_1$, both the indicator as well as the probability are decreasing functions of $d$. As a result, we can bound the expected value from above by setting $d=0$ in the indicator and using $X_n(0)$ in the probability. The upper bound has the desired limit by~\cite[Lemma $2.5$]{Esl16} with $a=b=0$. Similarly, we can bound the expected value from below by setting $d=L$ in the indicator and using $X_n(L)$ in the probability. The result then follows from~\cite[Lemma $2.5$]{Esl16} with $a=0, b=L$, which yields a matching lower bound.
	\end{proof}
	
	To finish this section, we use the results related to the truncated selection sets developed in Section~\ref{sec:pre} to extend Proposition~\ref{prop:limit} to the case of multiple vertices. 
	
	\begin{proposition}\label{prop:condprobmult}
		Fix $k\in\N,(a_i)_{i\in[k]}\in[0,2)^k$. Let $(d_i)_{i\in[k]}$ be $k$ integer-valued sequences such that, for all $i\in[k]$, $\lim_{n\to\infty}d_i/\log n=a_i$. Let $\ell_i:=n\exp(-d_i/2+x_i\sqrt{d_i/4})$ and	$h_i:=(\log n-d_i/2)+y_i\sqrt{\log n-d_i/4}$ with $( x_i)_{i\in[k]},( y_i)_{i\in[k]}\in \R^k$, and set $t_n=\lceil (\log n)^2\rceil$. Then, 
		\be \label{eq:multres}
		\lim_{n\to\infty}\mathbb P(h_{n,1}(i)\leq h_i,i\in[k]\,|\, \d_n(i)\geq d_i,i\in[k])=\prod_{i=1}^k \Phi(y_i).
		\ee 
		If, additionally, $d_i$ diverges as $n\to\infty$ for all $i\in[k]$, let $M$ and $N$ be independent standard normal random variables. Then,
		\be \ba\label{eq:multreswithlab}
		\lim_{n\to\infty}\mathbb P{}&(h_{n,1}(i)\leq h_i, \ell_n(i)\geq \ell_i,i\in[k]\,|\, \d_n(i)\geq d_i,i\in[k])\\
		&=\prod_{i=1}^k \P{M\sqrt{\frac{a_i}{4-a_i}}+N\sqrt{1-\frac{a_i}{4-a_i}}\leq y_i, M>x_i}.
		\ea \ee  
	\end{proposition}
	
	\begin{remark} \label{rem:condprobmult}
		As is the case in Remark~\ref{rem:limit}, it follows from Lemma~\ref{lemma:h2n} and Remark~\ref{rem:h2n} that the result in Proposition~\ref{prop:condprobmult} holds when substituting $h_n(i)$ for $h_{n,1}(i)$ as well.
	\end{remark}
	
	\begin{proof}
		We provide a proof for~\eqref{eq:multreswithlab}, the proof of~\eqref{eq:multres} uses the same steps.
		
		It suffices to prove that 
		\be \ba
		\mathbb P{}&(h_{n,1}(i)\leq h_i, \ell_n(i)\geq \ell_i, \d_n(i)\geq d_i,i\in[k])\\
		&=(1+o(1))2^{-\sum_{i=1}^k d_i}\prod_{i=1}^k \P{M\sqrt{\frac{a_i}{4-a_i}}+N\sqrt{1-\frac{a_i}{4-a_i}}\leq y_i, M>x_i},
		\ea\ee
		since then, by Proposition~\ref{prop:deg}, 
		\be \ba
		\lim_{n\to\infty}\mathbb P{}&(h_{n,1}(i)\leq h_i, \ell_n(i)\geq \ell_i,i\in[k]\,|\, \d_n(i)\geq d_i,i\in[k])\\
		&=\lim_{n\to\infty}\frac{(1+o(1))2^{-\sum_{i=1}^k d_i}\prod_{i=1}^k \P{M\sqrt{\frac{a_i}{4-a_i}}+N\sqrt{1-\frac{a_i}{4-a_i}}\leq y_i, M>x_i}}{\P{\d_n(i)\geq d_i, i\in[k]}}\\
		&=\prod_{i=1}^k \P{M\sqrt{\frac{a_i}{4-a_i}}+N\sqrt{1-\frac{a_i}{4-a_i}}\leq y_i, M>x_i}.
		\ea\ee
		Let us define 
		\be\ba 
		f_n(\bar J):={}&\P{h_{n,1}(i)\leq h_i,\ell_n(i)\geq \ell_i,\d_n(i)\geq d_i, i\in[k]\,|\, \overline{ \mathcal{ S}}_{n,1}=\bar J},\\
		g_n(\bar J):={}&\prod_{i=1}^k\P{h_{n,1}(i)\leq h_i, \ell_n(i)\geq \ell_i, \d_n(i)\geq d_i\,|\, \S_{n,1}(i)=J_i}.
		\ea\ee 
		Then, take $c\in(\max_{i\in[k]}a_i,2)$ and set $\delta:=2-c$ so that $\B\subseteq \A$ by Lemma~\ref{lemma:BinA}. We write,
		\be\ba \label{eq:fngn}
		\mathbb P{}&(h_{n,1}(i)\leq h_i, \ell_n(i)\geq \ell_i,\deg_n(i)\geq d_i , i\in[k])\\
		&=\E{f_n(\overline \S_{n,1})}\\
		&=\E{f_n(\overline \S_{n,1})\ind_{\{\overline \S_{n,1}\in \B\}}}+\E{f_n(\overline \S_{n,1})\ind_{\{\overline \S_{n,1} \in\A\backslash\B\}}}.
		\ea \ee 
		For the first term on the right-hand side, we use that the truncated selection sets are pairwise disjoint by the definition of $\B$ in~\eqref{eq:AB} and that by Lemma~\ref{lemma:probmultvert}, $f_n(\bar J)=g_n(\bar J)$ for all $\bar J\in \B$ and $n$ sufficiently large as a result. Together with Lemma~\ref{lemma:StoR}, recalling that $\overline \CR_{n,1}$ is a tuple of $k$ independent copies of $\S_{n,1}(1)$, this yields
		\be\ba \label{eq:fngn2}
		\E{f_n(\overline \S_{n,1})\ind_{\{\overline \S_{n,1}\in \B\}}}&=\sum_{\bar J\in \B}f_n(\bar J)\P{\overline{ \mathcal{ S}}_{n,1}=\bar J}\\
		&=\sum_{\bar J\in \B}g_n(\bar J)\P{\overline \CR_{n,1}=\bar J}(1+o(1))\\
		&=\E{g_n(\overline \CR_{n,1})\ind_{\{\overline \CR_{n,1}\in \B\}}}(1+o(1)).
		\ea\ee 
		Moreover, since $f_n(\bar J),g_n(\bar J)\leq 2^{-\sum_{i=1}^k d_i}$ when $\bar J\in \A$ by~\eqref{eq:degprob} in Lemma~\ref{lemma:labeldegprob}, and using Lemmas~\ref{lemma:B} and~\ref{lemma:StoR},
		\be\ba \label{eq:snrn}
		\Big|{}&\E{f_n(\overline \S_{n,1})\ind_{\{\overline \S_{n,1} \in\A\backslash\B\}}}-\E{g_n(\overline \CR_{n,1})\ind_{\{\overline \CR_{n,1} \in\A\backslash\B\}}}\Big|\\
		&\leq 2^{-\sum_{i=1}^k d_i}(\P{\overline \S_{n,1} \in\A\backslash\B}+\P{\overline \CR_{n,1} \in\A\backslash\B})\\
		&=o\Big(2^{-\sum_{i=1}^k d_i}\Big).
		\ea\ee 
		Thus, combining~\eqref{eq:fngn}, \eqref{eq:fngn2} and~\eqref{eq:snrn}, we arrive at
		\be\ba \label{eq:probgn}
		\mathbb P(h_{n,1}(i)\leq h_i, \ell_n(i)\geq \ell_i,\deg_n(i)\geq d_i , i\in[k])=\E{g_n(\overline \CR_{n,1})}(1+o(1))+o\Big(2^{-\sum_{i=1}^k d_i}\Big).
		\ea\ee 
		As the elements of $\overline \CR_{n,1}$ are i.i.d., we obtain 
		\be \ba
		\mathbb E[g_n(\overline \CR_{n,1})]&=\prod_{i=1}^k \P{h_{n,1}(i)\leq h_i, \ell_n(i)\geq \ell_i, \d_n(i)\geq d_i}\\
		&=\prod_{i=1}^k \P{h_{n,1}(i)\leq h_i, \ell_n(i)\geq \ell_i\,|\, \d_n(i)\geq d_i}\P{\d_n(i)\geq d_i}.
		\ea \ee 
		By combining Proposition~\ref{prop:deg}, Proposition~\ref{prop:limit} and~\eqref{eq:probgn}, we thus have 
		\be \ba
		\mathbb P{}&(h_{n,1}(i)\leq h_i, \ell_n(i)\geq \ell_i,\deg_n(i)\geq d_i , i\in[k])\\
		&=(1+o(1))2^{-\sum_{i=1}^k d_i}\prod_{i=1}^k \P{M\sqrt{\frac{a_i}{4-a_i}}+N\sqrt{1-\frac{a_i}{4-a_i}}\leq y_i, M>x_i}, 
		\ea\ee 
		as desired, which concludes the proof.
	\end{proof}
	
	\section{Joint properties of vertices with a given label}\label{sec:lab}
	
	This section is devoted to studying the joint behaviour of the degree and depth of vertices with a given label. We use the preliminary results proved in Section~\ref{sec:pre} to obtain the required results. The section is structured in the same way as Section~\ref{sec:deg}.

	We let $\ell\in[n]$ be increasing in $n$ such that $\ell$ diverges as $n\to\infty$, and set 
	\be \label{eq:hdt}
	h:=\log \ell+x\sqrt{\log \ell}, \quad \begin{cases} d:=\log(n/\ell)+y\sqrt{\log(n/\ell)} & \mbox{if } \ell=o(n),\\
		d\in\N_0\text{ fixed} & \mbox{otherwise,}
	\end{cases}\quad t_n:=\lceil \log \ell\rceil ,
	\ee 
	with $x,y\in\R$. Moreover, we define, for the same $y\in\R$ used in the definition of $d$, $\rho\in(0,1)$ and with $P(\rho)\sim \text{Poi}(\log(1/\rho))$, 
	\be \label{eq:pr}
	\text{Pr}=\text{Pr}(y,\rho,\ell):=\begin{cases}
		\Phi(y) & \mbox{if } \ell=o(n),\\
		\P{P(\rho)\leq d} & \mbox{if } \ell=(1+o(1))\rho n, \\
		1 & \mbox{if } \ell=n-o(n).
	\end{cases}
	\ee 
	We then have the following result.	
	
	\begin{proposition}\label{prop:limitfixell}
		Let $d,h,\ell$ and $t_n$ be as in~\eqref{eq:hdt} with $x,y\in \R$, and recall $\mathrm{Pr}$ from~\eqref{eq:pr}, with $y\in \R,\rho \in (0,1)$. Then,
		\be 
		\lim_{n\to\infty} \P{h_{n,1}(1)\leq h, \d_n(1)\leq d\,|\, \ell_n(1)=\ell}=\Phi(x)\mathrm{Pr}.
		\ee 
	\end{proposition} 
	
	\begin{remark} 
		As is the case in Remark~\ref{rem:limit}, it follows from Lemma~\ref{lemma:h2n} and Remark~\ref{rem:h2n} that the result in Proposition~\ref{prop:limitfixell} holds when substituting $h_n(i)$ for $h_{n,1}(i)$ as well.
	\end{remark}
	
	\begin{proof}
		We start by using Lemmas~\ref{lemma:probonevert} and~\ref{lemma:labeldegprob} to obtain 
		\be\ba 
		\mathbb P{}&(h_{n,1}(1)\leq h, \d_n(1)\leq d\,|\, \ell_n(1)=\ell)\\
		&=\frac{\P{h_{n,1}(1)\leq h, \d_n(1)\leq d, \ell_n(1)=\ell}}{\P{ \ell_n(1)=\ell}}\\
		&=n\mathbb E\Big[2^{-([\ell+1,n]\cap \S_{n,1}(1)|+1)}\ind_{\{|[\ell+1,n]\cap \S_{n,1}(1)|\leq d\}}\ind_{\{\ell\in \S_{n,1}\}} \mathbb P\big(X_{n,\ell,2}(1)\leq h-1\big| \S_{n,1}(1)\big)\Big].
		\ea \ee  
		We observe that we can divide the terms in the expected value into three parts which are pairwise independent. Namely, the exponent and the first indicator, the second indicator, and finally the conditional probability, respectively. Indeed, the exponent and first indicator only depend on $[\ell+1,n]\cap \S_{n,1}(1)$, the second indicator only on the event $\{\ell\in \S_{n,1}(1)\}$, and the conditional probability depends only on $[t_n,\ell-1]\cap \S_{n,1}(1)$. Since $\ell>\lceil\log \ell\rceil=t_n$ for all $n$ sufficiently large, these three parts depend on disjoint sets of independent random variables and are hence independent. As a result, we obtain 
		\be \ba \label{eq:firststep}
		\mathbb E\big[{}&2^{-([\ell+1,n]\cap \S_{n,1}(1)|+1)}\ind_{\{|[\ell+1,n]\cap \S_{n,1}(1)|\leq d\}}\ind_{\{\ell\in \S_{n,1}\}} \P{X_{n,\ell,2}(1)\leq h-1\,|\, \S_{n,1}(1)}\big]\\
		={}&\frac 12 \P{\ell\in \S_{n,1}(1)}\E{2^{-|[\ell+1,n]\cap \S_{n,1}(1)|}\ind_{\{|[\ell+1,n]\cap \S_{n,1}(1)|\leq d\}}}\P{X_{n,\ell,2}(1)\leq h-1}.
		\ea \ee 
		The first probability on the right-hand side equals $2/\ell$. The expected value on the right-hand side can be rewritten as follows. First, by summing over all possible truncated selection sets $\S_{n,1}(1)$, 
		\be \ba 
		\mathbb E\Big[{}&2^{-|[\ell+1,n]\cap \S_{n,1}(1)|}\ind_{\{|[\ell+1,n]\cap \S_{n,1}(1)|\leq d\}}\Big]\\
		&=\sum_{m=0}^d \sum_{\substack{S\subseteq [\ell+1,n]\\ |S|=m}}2^{-m}\prod_{j\in S}\frac 2j\prod_{j\not\in S}\Big(1-\frac 2j\Big)\\
		&=\sum_{m=0}^d \sum_{\substack{S\subseteq \{\ell+1,\ldots, n\}\\ |S|=m}} \prod_{j\in S}\frac 1j \prod_{j\not\in S}\Big(1-\frac 1j\Big) \prod_{j\not \in S}\frac{j-2}{j-1}\\
		&=\frac{\ell-1}{n-1}\sum_{m=0}^d \sum_{\substack{S\subseteq \{\ell+1,\ldots, n\}\\ |S|=m}} \prod_{j\in S}\frac{j-1}{j-2}\prod_{j\in S}\frac 1j\prod_{j\not\in S}\Big(1-\frac 1j\Big).
		\ea \ee 
		As $\ell$ diverges with $n$, $(\ell-1)/(n-1)=(1+o(1))\ell/n$. Defining $\wt \S_{n,1}(1)$ as a random subset of $\{\ell+1,\ldots ,n\}$ which includes each integer $j\in \{\ell+1,\ldots, n\}$ independently with probability $1/j$, the double sum and triple product can be interpreted as
		\be 
		\E{\ind_{\{|\wt \S_{n,1}(1)|\leq d\}}\prod_{j=\ell+1}^n \Big(1+\ind_{\{j\in \wt \S_{n,1}(1)\}}\frac{1}{j-2}\Big)}.
		\ee 
		Combining both observations we obtain 
		\be \ba \label{eq:tildeexp}
		\mathbb E\Big[{}&2^{-|[\ell+1,n]\cap \S_{n,1}(1)|}\ind_{\{|[\ell+1,n]\cap \S_{n,1}(1)|\leq d\}}\Big]\\
		&=(1+o(1))\frac \ell n \E{\ind_{\{|\wt \S_{n,1}(1)|\leq d\}}\prod_{j=\ell+1}^n \Big(1+\ind_{\{j\in \wt \S_{n,1}(1)\}}\frac{1}{j-2}\Big)}.
		\ea \ee 
		By bounding the product from below by one and using~\eqref{eq:firststep}, we obtain the lower bound
		\be \ba\label{eq:lower}
		n\mathbb E\big[{}&2^{-([\ell+1,n]\cap \S_{n,1}(1)|+1)}\ind_{\{|[\ell+1,n]\cap \S_{n,1}(1)|\leq d\}}\ind_{\{\ell\in \S_{n,1}\}} \P{X_{n,\ell,2}(1)\leq h-1\,|\, \S_{n,1}(1)}\big]\\
		&\geq (1+o(1))\P{|\wt \S_{n,1}(1)|\leq d}\P{X_{n,\ell,2}(1)\leq h-1}\\
		&=\Phi(x)\text{Pr}+o(1), 
		\ea\ee 
		where the last step follows if we assume that the two probabilities in the first step are asymptotically equal to $\text{Pr}$ and $\Phi(x)$, respectively. For an upper bound, we first expand the product in the expected value of~\eqref{eq:tildeexp} to obtain   
		\be\ba \label{eq:expand}
		\mathbb E\Big[{}&2^{-|[\ell+1,n]\cap \S_{n,1}(1)|}\ind_{\{|[\ell+1,n]\cap \S_{n,1}(1)|\leq d\}}\Big]\\
		={}&(1+o(1))\frac \ell n \Bigg( \P{|\wt \S_{n,1}(1)|\leq d}\\
		&+\sum_{m=1}^{n-\ell}\sum_{\ell+1\leq j_1<\ldots <j_m\leq n}\Big(\prod_{t=1}^m \frac {1}{j_t-2}\Big)\E{\ind_{\{|\wt \S_{n,1}(1)|\leq d\}}\prod_{t=1}^m \ind_{\{j_t\in \wt \S_{n,1}(1)\}}}\Bigg).
		\ea\ee 
		We then use the Cauchy-Schwarz inequality to bound 
		\be \ba 
		\sum_{m=1}^{n-\ell}{}&\sum_{\ell+1\leq j_1<\ldots <j_m\leq n}\Big(\prod_{t=1}^m \frac {1}{j_t-2}\Big)\E{\ind_{\{|\wt \S_{n,1}(1)|\leq d\}}\prod_{t=1}^m \ind_{\{j_t\in \wt \S_{n,1}(1)\}}}\\
		&\leq \P{|\wt \S_{n,1}(1)|\leq d}^{1/2}\sum_{m=1}^{n-\ell}\sum_{\ell+1\leq j_1<\ldots <j_m\leq n}\prod_{t=1}^m \bigg(\frac {1}{j_t-2} \P{j_t\in \wt \S_{n,1}(1)}^{1/2}\bigg).
		\ea \ee 
		As $\mathbb P(j_t\in \wt \S_{n,1}(1))=1/j_t\leq 1/(j_t-2)$, we arrive at the upper bound
		\be \ba 
		\P{|\wt \S_{n,1}(1)|\leq d}^{1/2}\sum_{m=1}^{n-\ell}{}&\sum_{\ell-1\leq j_1<\ldots <j_m\leq n-2}\prod_{t=1}^m \frac {1}{j_t^{3/2}}\\
		\leq \P{|\wt \S_{n,1}(1)|\leq d}^{1/2}{}&\sum_{m=1}^{n-\ell}\big(2(\ell-2)^{-1/2}\big)^m\\
		\leq  \P{|\wt \S_{n,1}(1)|\leq d}^{1/2}{}&\frac{2}{(\ell-2)^{1/2}-2}.
		\ea \ee 
		Combining this with~\eqref{eq:expand} and~\eqref{eq:firststep} and since $\ell$ diverges with $n$, we thus obtain the upper bound 
		\be \ba 
		n\mathbb E\big[{}&2^{-([\ell+1,n]\cap \S_{n,1}(1)|+1)}\ind_{\{|[\ell+1,n]\cap \S_{n,1}(1)|\leq d\}}\ind_{\{\ell\in \S_{n,1}\}} \P{X_{n,\ell,2}(1)\leq h-1\,|\, \S_{n,1}(1)}\big]\\
		\leq{}& (1+o(1))\Big( \P{|\wt \S_{n,1}(1)|\leq d}+\P{|\wt \S_{n,1}(1)|\leq d}^{1/2}\frac{2}{(\ell-2)^{1/2}+2}\Big)\P{X_{n,\ell,2}(1)\leq h-1}\\
		={}&\Phi(x)\text{Pr}+o(1), 
		\ea \ee 
		when we (again) assume that the first and last probability on the second line are asymptotically equal to $\text{Pr}$ and $\Phi(x)$, respectively. As this upper bound matches the lower bound in~\eqref{eq:lower}, we arrive at the desired result.
		
		It remains to prove that 
		\be \label{eq:finproof}
		\P{|\wt \S_{n,1}(1)|\leq d}=\text{Pr}+o(1), \qquad \P{X_{n,\ell,2}(1)\leq h-1}=\Phi(x)+o(1).
		\ee 
		For the first result, let us start by considering $\ell=o(n)$, so that $d:=\log(n/\ell)+y\sqrt{\log (n/\ell)}$ for $y\in\R$ fixed. It then follows from Lindeberg's conditions~\cite[Theorem $3.4.5$]{Dur19} that 
		\be \ba 
		\P{|\wt \S_{n,1}(1)|\leq d}&=\P{\frac{|\wt \S_{n,1}(1)|-\mathbb E[|\wt \S_{n,1}(1)|]}{\sqrt{\Var(|\wt \S_{n,1}(1)|)}}\leq\frac{\log(n/\ell)+y\sqrt{\log(n/\ell)}-\mathbb E[|\wt \S_{n,1}(1)|]}{\sqrt{\Var(|\wt \S_{n,1}(1)|)}}} \\
		&=\Phi(y)+o(1), 
		\ea \ee 
		since 
		\be \ba
		\mathbb E[|\wt \S_{n,1}(1)|]&=\sum_{j=\ell+1}^n \frac 1j=\log(n/\ell)+\CO(1),\qquad \\
		\Var( |\wt \S_{n,1}(1)|)&=\sum_{j=\ell+1}^n \frac 1j\Big(1-\frac 1j\Big)=\log(n/\ell)+\CO(1),
		\ea \ee 
		when $\ell=o(n)$. When $\ell=(1+o(1))\rho n$ for some $\rho\in(0,1)$, we recall that $d\in \N_0$ is fixed, and instead use that for any $t\in\R$,
		\be 
		\E{\e^{t|\wt \S_{n,1}(1)|}}=\prod_{j=\ell+1}^n \Big(1-\frac 1j+\e^t\frac 1j\Big)=\prod_{j=\ell+1}^n \Big(1+\big(\e^t-1\big)\frac 1j\Big).
		\ee 
		Using that $x-x^2\leq \log(1+x)\leq x$ for all $x>0$ and that 
		\be 
		\sum_{j=\ell+1}^n \frac 1j=(1+o(1))\int_\rho^1 x^{-1}\,\d x=(1+o(1))\log(1/\rho), 
		\ee 
		yields
		\be 
		\E{\e^{t|\wt \S_{n,1}(1)|}}=\e^{(\e^t-1)\log(1/\rho)}+o(1).
		\ee
		Since, for any $t\in\R$, the moment generating function (MGF) of $|\wt \S_{n,1}(1)|$ converges to the MGF of $P(\rho)$,  
		\be 
		\P{|\wt \S_{n,1}(1)|\leq d}=\P{P(\rho)\leq d}+o(1). 
		\ee 
		Finally, when $\ell=n-o(n)$, using Markov's inequality yields
		\be 
		\P{|\wt \S_{n,1}(1)|=0}\geq 1-\mathbb E[|\wt \S_{n,1}(1)|]=1-\sum_{j=\ell+1}^n \frac 1j=1-(1+o(1))\log(n/\ell)=1-o(1), 
		\ee 
		as desired.
		
		For the latter result in~\eqref{eq:finproof} we set $\wt Q_n:=|[t_n, \ell-1]\cap \S_{n,1}(1)|$, let $(I^n_j)_{j\in[n]}$ denote independent Bernoulli random variables with success probability $1/2$, also independent of $\wt Q_n$, and write 
		\be 
		\frac{2X_{n,\ell,2}(1)-2\log \ell}{2\sqrt{\log \ell}}=\frac{2\sum_{j=1}^{\wt Q_n}I^{\wt Q_n}_j-\wt Q_n}{\sqrt{\wt Q_n}}\sqrt{\frac{\wt Q_n}{4\log \ell}}+\frac{\wt Q_n-\mathbb E[\wt Q_n]}{\sqrt{\Var(\wt Q_n)}}\sqrt{\frac{\Var(\wt Q_n)}{4\log \ell}}+\frac{\mathbb E[\wt Q_n]-2\log \ell}{2\sqrt{\log \ell}}.
		\ee  
		We then use a similar approach as~\eqref{eq:qntilde}. In particular, we use the Skorokhod embedding which provides us with a coupling of the random variables $\wt Q_n$ and $(I_i^n)_{\inn}$ such that
		\be
		\frac{2\sum_{j=1}^n I^n_j-n}{\sqrt n}\toas N_1, \qquad \frac{\wt Q_n-\mathbb E[\wt Q_n]}{\sqrt{\Var(\wt Q_n)}}\toas N_2, 
		\ee 
		where $N_1,N_2$ are two independent standard normal random variables. Moreover, a straightforward computation of $\mathbb E[\wt Q_n]$ and $\Var(\wt Q_n)$ shows that $\wt Q_n/(2\log \ell) \toas 1$ (and hence $\wt Q_n\toas \infty$), that $\Var(\wt Q_n)/(2\log \ell)\to 1$ and that $\mathbb E[\wt Q_n]-2\log\ell=o(\sqrt{\log\ell})$ as $n\to\infty$. As a result, it follows that 
		\be 
		\frac{2X_{n,\ell,2}(1)-2\log \ell}{2\sqrt{\log \ell}}\toindis \frac{1}{\sqrt{2}} N_1 +\frac{1}{\sqrt{2}} N_2\overset d= N, 
		\ee 
		where $N$ is a standard normal random variable. As a result, recalling that $h:=\log \ell+x\sqrt{\log \ell}$,
		\be 
		\P{X_{n,\ell,2}(1)\leq h-1}=\P{\frac{2X_{n,\ell,2}(1)-2\log \ell}{2\sqrt{\log \ell}}\leq \frac{h-1-\log \ell}{\sqrt{\log \ell}}}=\Phi(x)+o(1), 
		\ee  
		as required, which concludes the proof.
	\end{proof}
	
	To finish this section, we use the results related to the truncated selection sets developed in Section~\ref{sec:pre} to extend Proposition~\ref{prop:limitfixell} to the case of multiple vertices. The choice of $t_n$ is imperative, and so we define, for some $(\ell_i)_{i\in[k]}\in [n]^k$,
	\be\label{eq:tnlab} 
	t_n:=\min_{i\in[k]}\lceil \log \ell_i\rceil. 
	\ee 
	We can then formulate the following result.
	
	\begin{proposition}\label{prop:condprobfixell}
		Fix $k\in\N$, let $(\ell_i)_{i\in[k]}$ be $k$ distinct integer-valued sequences such that $\ell_i$ increases with $n$ and $\ell_i$ diverges as $n\to\infty$ for all $i\in[k]$. Let, for $i\in[k]$, $h_i:=\log \ell_i+x_i\sqrt{\log \ell_i}$ and $d_i:=\log(n/\ell_i)+y_i\sqrt{\log (n/\ell_i)}$ if $\ell_i=o(n)$ and $d_i\in\N_0$ fixed otherwise, with $(x_i)_{i\in[k]},(y_i)_{i\in[k]}\in \R^k$ and let $t_n$ as in~\eqref{eq:tnlab}. Furthermore, recall the definition of $\mathrm{Pr}$ in~\eqref{eq:pr}. Then,
		\be\ba 
		\lim_{n\to\infty} \P{h_{n,1}\leq h_i, \d_n(i)\leq d_i, i\in[k]\,|\, \ell_n(i)=\ell_i, i\in[k]}=\prod_{i=1}^k \Phi(x_i)\mathrm{Pr}(y_i,\rho_i,\ell_i).
		\ea\ee 
	\end{proposition}
	
	\begin{remark} 
		As is the case in Remark~\ref{rem:limit}, it follows from Lemma~\ref{lemma:h2n} and Remark~\ref{rem:h2n} that the result in Proposition~\ref{prop:condprobfixell} holds when substituting $h_n(i)$ for $h_{n,1}(i)$ as well.
	\end{remark}
	
	\begin{proof}
		The proof follows a similar approach as the proof of Proposition~\ref{prop:condprobmult}. We first write  
		\be\ba\label{eq:fixlabcond}
		\mathbb P({}&h_{n,1}\leq h_i, \d_n(i)\leq d_i, i\in[k]\,|\, \ell_n(i)=\ell_i, i\in[k])\\
		&=\frac{ \P{h_{n,1}\leq h_i,  \ell_n(i)=\ell_i,\d_n(i)\leq d_i,  i\in[k]}}{ \P{ \ell_n(i)=\ell_i, i\in[k]}}\\
		&=(n)_k \P{h_{n,1}\leq h_i,  \ell_n(i)=\ell_i,\d_n(i)\leq d_i,  i\in[k]},
		\ea \ee 
		where the last step follows from Lemma~\ref{lemma:labeldegprob}. We then define
		\be \ba 
		f_n(\bar J)&:=\P{h_{n,1}(i)\leq h_i, \ell_n(i)=\ell_i, \d_n(i)\leq d_i, i\in[k]\, |\, \overline{\S}_{n,1}=\bar J},\\
		g_n(\bar J)&:=\prod_{i=1}^k\P{h_{n,1}(i)\leq h_i, \ell_n(i)=\ell_i, \d_n(i)\leq d_i,\, |\,\S_{n,1}(i)= J_i}.
		\ea \ee 
		With similar steps as in~\eqref{eq:fngn} through~\eqref{eq:probgn}, we then have
		\be \ba \label{eq:claim}
		\mathbb P({}&h_{n,1}(i)\leq h_i,\ell_n(i)=\ell_i, \d_n(i)\leq d_i, i\in[k])\\
		={}&\E{f_n(\overline{\S}_{n,1})\ind_{\{\overline{\S}_{n,1}\in \wtB\}}}+\E{f_n(\overline{\S}_{n,1})\ind_{\{\overline{\S}_{n,1}\in \wtB^c\}}}\\
		={}&\E{g_n(\overline{\CR}_{n,1})}(1+o(1))+\Big(\E{f_n(\overline{\S}_{n,1})\ind_{\{\overline{\S}_{n,1}\in \wtB^c\}}}-\E{g_n(\overline{\CR}_{n,1})\ind_{\{\overline{\CR}_{n,1}\in \wtB^c\}}}\Big).
		\ea\ee 
		It follows from~\eqref{eq:wtbc} in Lemma~\ref{lemma:labeldegprob} that 
		\be 
		\E{f_n(\overline{\S}_{n,1})\ind_{\{\overline{\S}_{n,1}\in \wtB^c\}}}=\P{\ell_n(i)=\ell_i, i\in[k], \overline \S_{n,1}\in \wtB^c}=o(n^{-k}).
		\ee 
		A similar argument as in the proof of~\eqref{eq:wtbc} can be used to show that 
		\be 
		\E{g_n(\overline \CR_{n,1})\ind_{\{\overline \CR_{n,1}\in \wt B^c\}}}=o(n^{-k}), 
		\ee 
		as well. As the elements of $\overline \CR_{n,1}$ are i.i.d., we have 
		\be \ba 
		\mathbb E[g_n(\overline{\CR}_{n,1})]&=\prod_{i=1}^k \P{h_{n,1}(i)\leq h_i, \ell_n(i)=\ell_i, \d_n(i)\leq d_i}\\
		&=\prod_{i=1}^k \P{h_{n,1}(i)\leq h_i, \d_n(i)\leq d_i\,|\, \ell_n(i)=\ell_i}\P{\ell_n(i)=\ell_i}.
		\ea\ee 
		The product on the right-hand side equals 
		\be 
		(1+o(1))n^{-k}\prod_{i=1}^k \Phi(y_i)\Pr(y_i,\rho_i,\ell_i)
		\ee 
		by Proposition~\ref{prop:limitfixell} and Lemma~\ref{lemma:labeldegprob}. Using this in~\eqref{eq:claim} yields
		\be \ba
		\mathbb P({}&h_{n,1}(i)\leq h_i,\ell_n(i)=\ell_i, \d_n(i)\leq d_i, i\in[k])
		=\frac{1+o(1)}{n^k}\prod_{i=1}^k \Phi(y_i)\Pr(y_i,\rho_i,\ell_i).
		\ea\ee 
		Combining this with~\eqref{eq:fixlabcond} then yields the desired result. 
	\end{proof}
	\newpage 
	
	\section{Proof of Theorem~\ref{thrm:degdepthlabelrrt}}\label{sec:proofppp}
	
	This section is devoted to proving Theorem~\ref{thrm:degdepthlabelrrt}. We first provide some additional theory on top of what is introduced in Section~\ref{sec:king} prior to stating the proof. 
	
	\subsection{Convergence of marked point processes via finite dimensional distributions}	
	
	We recall that, as discussed in Section~\ref{sec:king}, Theorem~\ref{thrm:degdepthlabelrrt} can be understood as the weak convergence of the marked point process $\CM\CP^{(n_t)}$ to $\CM\CP^\eps$, as defined in~\eqref{eq:pppn} and~\eqref{eq:limppp2}, respectively. The approach to prove this is via the convergence of its finite dimensional distributions (FDDs) along suitable subsequences. The FDDs of a random measure $\CP$ on $\CX$ are defined as the joint distributions, for all finite families of bounded Borel sets $(B_1, \ldots, B_k)\in \CB(\CX)^k$, of the random variables $(\CP(B_1), \ldots, \CP(B_k))$, see~\cite[Definition $9.2.$II]{DalVer08}. Moreover, by~\cite[Proposition $9.2.$III]{DalVer08}, the distribution of a random measure $\CP$ on $\CX$ is completely determined by the FDDs for all finite families $(B_1, \ldots, B_k)$ of \emph{disjoint} sets from a semiring $\CA$ that generates $\CB(\CX)$. In our case, we consider the marked point process $\CM\CP^{(n)}$ on $\CX:=\Z^*\times \R^2$, see~\eqref{eq:limppp2}. Here, we let
	\be \label{eq:A}
	\CA:=\{\{s\}\times (a,b]\times (c,d]: s\in \Z, a,b,c,d\in \R\}\cup \{[s,\infty]\times (a,b]\times (c,d]: s\in \Z, a,b,c,d\in \R\}
	\ee
	be the semiring that generates $\CB(\Z^*\times \R^2)$. 
	
	Recall the counting measures $X_s^{(n)}(B), X_{\geq s}^{(n)}(B), \wt X_s^{(n)}(B),\wt X_{\geq s}^{(n)}(B)$ defined in~\eqref{eq:xnking} (in terms of the Kingman $n$-coalescent) and $X_s(B), X_{\geq s}(B)$ defined in~\eqref{eq:xn}. We observe that \\ $\wt X_s^{(n)}(B)=\CM\CP^{(n)}(\{s\}\times B), \wt X_{\geq s}^{(n)}(B)=\CM\CP^{(n)}([s,\infty]\times B), X_s(B)=\CM\CP^\eps(\{s\}\times B)$ and $X_{\geq s}(B)=\CM\CP^\eps([s,\infty]\times B)$. As a result, the convergence of the FDDs of $\CM\CP^{(n_t)}$ to the FDDs of $\CM\CP^\eps$ can be obtained via the convergence of any finite collection of these counting measures.
	
	For any $K\in\N$, take any (fixed) increasing integer sequence $(s_m)_{m\in[K]}$ and let $0\leq   K':=\min\{m: s_{m+1}=s_K\}$. Also, fix any sequence $(B_m)_{m\in[K]}$ with $B_m\in\CB(\R^2)$ such that $B_m\cap B_\ell=\emptyset $ when $s_m=s_\ell$ and $m\neq \ell$. The conditions on the sets $B_m$ ensure that the elements $\{s_1\}\times B_1, \ldots, \{s_K'\}\times B_{K'}, \{s_{K'+1},\ldots\}\times B_{K'+1}, \ldots, \{s_K,\ldots\}\times B_K$ of $\CA$ are disjoint. We are thus required to prove the joint distributional convergence of the random variables
	\be\label{eq:xnseq}
	(\wt X_{s_1}^{(n)}(B_1),\ldots,\wt  X^{(n)}_{s_{K'}}(B_{K'}),\wt X^{(n)}_{\geq s_{K'+1}}(B_{K'+1}),\ldots, \wt X_{\geq s_K}^{(n)}(B_K)),
	\ee
	to prove Theorem~\ref{thrm:degdepthlabelrrt}. We use the method of moments combined with Proposition~\ref{prop:momentconv} to achieve this:
	
	\begin{proof}[Proof of Theorem~\ref{thrm:degdepthlabelrrt} subject to Proposition~\ref{prop:momentconv}]
		As discussed, it suffices to prove the weak convergence of $\CM\CP^{(n_t)}$ to $\CM\CP^\eps$ along subsequences $(n_t)_{t\in\N}$ such that $\eps_{n_t}\to \eps$ (where $\eps\in[0,1]$) as $t\to\infty$. In turn, this is implied by the convergence of the FDDs, i.e., by the joint convergence of the counting measures $\wt X_s^{(n)}(B),\wt X_{\geq s}^{(n)}(B)$ of finite collections of disjoint subsets of $\mathcal A$ (see~\eqref{eq:A}). 
		
		We recall that the points $P_i$ in the definition of the variables $X_s(B),X_{\geq s}(B)$ in~\eqref{eq:xn} are the points of the Poisson point process $\CP$ with intensity measure $\lambda(\d x):= 2^{-x}\log 2\,\d x$ in decreasing order. As a result, as the random variables $(M_i,N_i)_{i\in\N}$ are i.i.d.\ and also independent of $\CP$, $X_s(B)\sim \text{Poi}(\lambda_s(B)), X_{\geq s}(B)\sim \text{Poi}(2\lambda_s(B))$, where 
		\be 
		\lambda_s(B)=2^{-(s+1)+\eps}\P{M_1\sqrt{1-\frac{\mu}{\sigma^2}}+N_1\sqrt{\frac{\mu}{\sigma^2}}\in B}.
		\ee 
		We also recall that $(n_\ell)_{\ell\in\N}$ is a subsequence such that $\eps_{n_\ell}\to \eps$ as $\ell\to\infty$. We now take $c\in(1/\log 2,2)$ and for any $K\in\N$ consider any fixed non-decreasing integer sequence $(s_m)_{m\in[K]}$. By the choice of $c$ and the fact that the $s_m$ are fixed with respect to $n$, $s_1+\log_2 n=\omega(1)$ and $s_K+\log_2 n <c\log n$ for all $n\geq 2$ follow. Moreover, let $K':=\min\{m: s_{m+1}=s_K\}$ and let $(B_m)_{m\in[K]}$ be a sequence of sets in $\CB(\R^2)$ such that $B_m\cap B_\ell=\emptyset$ when $s_m=s_\ell$ and $m\neq \ell$.  We can then, for any $(c_m)_{m\in[K]}\in \N_0^K$, obtain from Proposition~\ref{prop:momentconv} and since $s_1, \ldots, s_K=o(\sqrt{\log n})$,
		\be  \ba
		\lim_{n\to\infty}{}&\mathbb E\bigg[\prod_{m=1}^{K'}\Big(\wt X_{s_m}^{(n_\ell)}(B_m)\Big)_{c_m}\prod_{m=K'+1}^K\!\!\Big(\wt X_{\geq s_m}^{(n_\ell)}(B_m)\Big)_{c_m} \bigg]
		\\
		&=\prod_{m=1}^{K'}\lambda_{s_m}(B_m)^{c_m}\prod_{m=K'+1}^K(2\lambda_{s_m}(B_m))^{c_m}\\
		&=\mathbb E\bigg[\prod_{m=1}^{K'}\Big(X_{s_m}(B_m)\Big)_{c_m}\prod_{m=K'+1}^K\!\!\Big(X_{\geq s_m}(B_m)\Big)_{c_m} \bigg],
		\ea \ee
		where the last step follows from the independence property of (marked) Poisson point processes and the choice of the sequences $(s_m, B_m)_{m\in[K]}$. The method of moments~\cite[Section $6.1$]{JanLucRuc00} then concludes the proof.
	\end{proof}
	
	It remains to prove Proposition~\ref{prop:momentconv}.
	
	\begin{proof}[Proof of Proposition~\ref{prop:momentconv}]
		The proof essentially follows a similar approach as the proof of~\cite[Proposition $5.4$]{Lod21}. However, as certain estimations and definitions differ, we include it here for completeness. 
		
		Recall $\mu=1-1/(2\log 2)$, $\sigma^2=1-1/(4\log 2)$, and that we have fixed $K\in\N$ and $(a_m)_{m\in[K]}\in [0,2)^K$. Moreover, we have a non-decreasing integer sequence $(s_m)_{m\in[K]}$ with $s_1+\log_2 n =\omega(1)$ and 
		\be 
		\lim_{n\to\infty} \frac{s_m+\log_2 n}{\log n}=a_m, 
		\ee
		for all $m\in[K]$, and a sequence $(B_m)_{m\in[K]}$ such that $B_m\in \CB(\R^2)$ for all $m\in[K]$ and $B_m\cap B_\ell=\emptyset$ when $s_m=s_\ell$ and $m\neq \ell$. We also define $K':=\min\{m: s_{m+1}=s_K\}$. Finally, we recall that $M$ and $N$ are two independent standard normal random variables. Then, take an arbitrary sequence $(c_m)_{m\in[K]}\in \N_0^K$ and set  $L:=\sum_{m=1}^K c_m$ and $L':=\sum_{m=1}^{K'}c_m$.
		
		We define $\bar d=(d_i)_{i\in[L]}\in \Z^L, (a'_i)_{i\in[L]},$ and $\bar A=(A_i)_{i\in[L]}\subset \CB(\R^2)^L$ as follows: For each $i\in[L]$, find the unique $m\in[K]$ such that $\sum_{\ell=1}^{m-1}c_\ell<i\leq \sum_{\ell=1}^m c_\ell$ and define $d_i:=\lfloor \log_2 n\rfloor +s_m$, $a'_i:=a_m$, $ A_i:=B_m$. We note that this construction implies that the first $c_1$ many $d_i,a_i'$ and $A_i$ equal $\lfloor \log_2 n\rfloor +s_1,a_1$ and $B_1$, respectively, that the next $c_2$ many $d_i,a'_i$ and $A_i$ equal $\lfloor \log_2 n\rfloor +s_2,a_2$ and $B_2$, respectively, etcetera. Furthermore, $\lim_{n\to\infty}d_i/\log n=a'_i$ for all $i\in[L]$. We then define the events
		\be \ba 
		\CH\CL_{\bar A,\bar d}&:=\Big\{\Big(\frac{h_n(i)-(\log n-d_i/2)}{\sqrt{\log n-d_i/4}}, \frac{\log \ell_n(i)-(\log n-d_i/2)}{\sqrt{d_i/4}}\Big)\in A_i, i\in[L]\Big\},\\
		\CD_{\bar d}(L',L)&:=\{\d_n(i')=d_i', i'\in[L'], \d_n(i)\geq d_i, L'<i\leq L\},\\
		\CE_{\bar d}(S)&:=\{\d_n(i)\geq d_i+\ind_{\{i\in S\}}, i\in [L]\}.
		\ea \ee 
		We know from~\cite[Lemma $5.1$]{AddEsl18} that by the inclusion-exclusion principle, 
		\be\label{eq:inex}
		\P{\CD_{\bar d}(L',L)}=\sum_{j=0}^{L'}\sum_{\substack{ S\subseteq [L']:\\ |S|=j}}(-1)^j\P{\CE_{\bar d}(S)},		
		\ee
		so that intersecting the event $\CH\CL_{\bar A,\bar d}$ in the probabilities on both sides yields
		\be \label{eq:probdhl}
		\P{\CD_{\bar d}(L',M)\cap \CH\CL_{\bar A,\bar d}}=\sum_{j=0}^{L'}\sum_{\substack{ S\subseteq [L']:\\ |S|=j}}(-1)^j\P{\CE_{\bar d}(S)\cap \CH\CL_{\bar A,\bar d}}.
		\ee 
		Let us then define 
		\be \ba
		\wt{\CH\CL}_{\bar A,\bar d}(S):=\Big\{\Big({}&\frac{h_n(i)-(\log n-(d_i+\ind_{\{i\in S\}})/2)}{\sqrt{\log n-(d_i+\ind_{\{i\in S\}})/4}},\\
		& \frac{\log \ell_n(i)-(\log n-(d_i+\ind_{\{i\in S\}})/2)}{\sqrt{(d_i+\ind_{\{i\in S\}})/4}}\Big)\in A_i, i\in[L]\Big\}.
		\ea \ee 
		We use Proposition~\ref{prop:condprobmult} (combined with Remark~\ref{rem:condprobmult}) with $a_i'=\lim_{n\to\infty}(d_i+\ind_{\{i\in S\}})/\log n$ for all $i\in[L]$ and Proposition~\ref{prop:deg} to then obtain
		\be \ba
		\mathbb P{}&\Big(\CE_{\bar d}(S)\cap \wt{\CH\CL}_{\bar A,\bar d}(S)\Big)\\
		&=\P{ \wt{\CH\CL}_{\bar A,\bar d}(S)\,\Big|\,\CE_{\bar d}(S)}\P{\CE_{\bar d}(S)}\\
		&=(1+o(1))2^{-\sum_{i=1}^L (d_i+\ind_{\{i\in S\}})}\prod_{i=1}^L \P{\bigg(M\sqrt{\frac{a_i'}{4-a_i'}}+N\sqrt{1-\frac{a_i'}{4-a_i'}}, M\bigg)\in A_i}.
		\ea \ee 
		Since
		\be \ba
		\Big({}&\frac{h_n(i)-(\log n-(d_i+\ind_{\{i\in S\}})/2)}{\sqrt{\log n-(d_i+\ind_{\{i\in S\}})/4}},
		\frac{\log \ell_n(i)-(\log n-(d_i+\ind_{\{i\in S\}})/2)}{\sqrt{(d_i+\ind_{\{i\in S\}})/4}}\Big)\\
		&=(1+o(1))\Big(\frac{h_n(i)-(\log n-d_i/2)}{\sqrt{\log n-d_i/4}},
		\frac{\log \ell_n(i)-(\log n-d_i/2)}{\sqrt{d_i/4}}\Big),
		\ea \ee 
		we also obtain from Slutsky's theorem~\cite[Lemma $2.8$]{Vaart00} that 
		\be\ba 
		\mathbb P{}&(\CE_{\bar d}(S)\cap \CH\CL_{\bar A,\bar d})\\
		&=(1+o(1))2^{-\sum_{i=1}^L (d_i+\ind_{\{i\in S\}})}\prod_{i=1}^L  \P{\bigg(M\sqrt{\frac{a_i'}{4-a_i'}}+N\sqrt{1-\frac{a_i'}{4-a_i'}}, M\bigg)\in A_i}.
		\ea \ee 
		The right-hand side of~\eqref{eq:probdhl} then equals 
		\be \label{eq:condlimit2}
		\prod_{i=1}^L\bigg[ \P{\bigg(M\sqrt{\frac{a_i'}{4-a_i'}}+N\sqrt{1-\frac{a_i'}{4-a_i'}}, M\bigg)\in A_i}\bigg]\sum_{j=0}^L\sum_{\substack{ S\subseteq [L']:\\ |S|=j}}\frac{(1+o(1))(-1)^j}{ 2^{\sum_{i=1}^L (d_i+\ind_{\{i\in S\}})}},
		\ee 
		where the product is independent of $S$ and $j$ and can therefore be taken out of the double sum. The double sum equals
		\be \label{eq:doublesum}
		(1+o(1))\sum_{j=0}^L\sum_{\substack{ S\subseteq [L']:\\ |S|=j}}\!\!\!(-1)^j 2^{-j-\sum_{i=1}^L d_i}=(1+o(1))2^{-L'-\sum_{i=1}^L d_i}.
		\ee 
		Now, recall the definition of the variables $X_s^{(n)}(B),X_{\geq s}^{(n)}(B)$ as in~\eqref{eq:xnking}. Combining~\eqref{eq:probdhl}, \eqref{eq:condlimit2} and \eqref{eq:doublesum} together with the exchangeability of the degree, depth, and label of vertices $1,\ldots, K$, we arrive at
		\be\ba \label{eq:bigprod}
		\mathbb E\Bigg[{}&\prod_{m=1}^{K'}\Big(X_{s_m}^{(n)}(B_m)\Big)_{c_m}\prod_{m=K'+1}^K \Big(X_{\geq s_m}^{(n)}(B_m)\Big)_{c_m}\Bigg]\\
		={}&(n)_L \P{ \CD_{\bar d}(L',L)\cap \CH\CL_{\bar A,\bar d} }\\
		={}&(1+o(1))2^{L\log_2 n-L'-\sum_{i=1}^L d_i}\prod_{i=1}^L\P{\bigg(M\sqrt{\frac{a_i'}{4-a_i'}}+N\sqrt{1-\frac{a_i'}{4-a_i'}}, M\bigg)\in A_i},
		\ea\ee 
		since $(n)_L:=n(n-1)\cdots (n-(L-1))=(1+o(1))n^L$. We now recall that there are exactly $c_m$ many $d_i,a'_i,$ and $A_i$ that equal $\lfloor \log_2n\rfloor +s_m,a_m,$ and $B_m$, respectively, for each $m\in[K]$ and that $s_{K'+1}=\ldots =s_K$, so that 
		\be \ba \label{eq:mult}
		\prod_{i=1}^L{}& \P{\bigg(M\sqrt{\frac{a_i'}{4-a_i'}}+N\sqrt{1-\frac{a_i'}{4-a_i'}}, M\bigg)\in A_i}\\
		&=\prod_{m=1}^K \P{\bigg(M\sqrt{\frac{a_m}{4-a_m}}+N\sqrt{1-\frac{a_m}{4-a_m}}, M\bigg)\in B_m}^{c_m},\\
		L\log_2 n-L'-\sum_{i=1}^L d_i &= -\sum_{m=1}^{K'}(s_m+1-\eps_n)c_m-\sum_{m=K'+1}^K (s_K-\eps_n)c_m.
		\ea \ee 
		Combined with~\eqref{eq:bigprod}, this finally yields
		\be \ba\label{eq:fin}
		\mathbb E\Bigg[\prod_{m=1}^{K'}{}&\bigg(X_{s_m}^{(n)}(B_m)\bigg)_{c_m}\prod_{m=K'+1}^K \bigg(X_{\geq s_m}^{(n)}(B_m)\bigg)_{c_m}\Bigg]\\
		={}&(1+o(1))\prod_{m=1}^{K'} \bigg(\P{\bigg(M\sqrt{\frac{a_m}{4-a_m}}+N\sqrt{1-\frac{a_m}{4-a_m}}, M\bigg)\in B_m}2^{-(s_m+1)+\eps_n}\bigg)^{c_m}\\
		&\times \prod_{m=K'+1}^K \bigg(\P{\bigg(M\sqrt{\frac{a_m}{4-a_m}}+N\sqrt{1-\frac{a_m}{4-a_m}}, M\bigg)\in B_m}2^{-s_K+\eps_n}\bigg)^{c_m}.
		\ea \ee 
		To prove the second result in Proposition~\ref{prop:momentconv}, we use that for $s_1, \ldots, s_K=o(\sqrt{\log n})$, 
		\be \ba
		\Big({}&\frac{h_n(i)-(\log n-d_i/2)}{\sqrt{\log n-d_i/4}}, \frac{\log \ell_n(i)-(\log n-d_i/2)}{\sqrt{d_i/4}}\Big)\\
		&=(1+o(1))\Big(\frac{h_n(i)-\mu\log n}{\sqrt{\sigma^2\log n}}, \frac{\log \ell_n(i)-\mu\log n}{\sqrt{(1-\sigma^2)\log n}}\Big),
		\ea \ee 
		and that in this case, $a_m=\lim_{n\to\infty}(s_m+\log_2n)/\log n=1/\log 2$ for all $m\in[K]$. As a result, noting that 
		\be 
		\frac{1/\log 2}{4-1/\log 2}=1-\mu/\sigma^2,
		\ee
		a similar approach as the above proof for the random variables $\wt X_s^{(n)}(B),\wt X_{\geq s}^{(n)}(B)$ yields the desired result.
	\end{proof}

	\section{Proof of Theorems ~\ref{thrm:condconvking} and~\ref{thrm:kingfixlabel}}\label{sec:kingproof}
	
	In this section we provide the final steps that build on Propositions~\ref{prop:limit} and~\ref{prop:limitfixell} to prove Theorems~\ref{thrm:condconvking} and~\ref{thrm:kingfixlabel}. In particular, we show how to include the graph distance between vertices $1,\ldots, k$ in the Kingman $n$-coalescent. As mentioned at the end of Section~\ref{sec:king}, combining Theorems~\ref{thrm:condconvking} and~\ref{thrm:kingfixlabel} with Corollary~\ref{cor:unifnodes} then immediately implies Theorems~\ref{thrm:condconvrrt} and~\ref{thrm:fixedlabel}, respectively.
	
	Intuitively, the graph distance between vertices can be related to their (truncated) depth. By the definition of $\tau_k$, the largest common ancestor of any two distinct vertices $i,j\in [k]$ in the random recursive tree has label at most $\tau_k$ and hence the sum of the depths and truncated depths of vertices $i$ and $j$ form an upper and lower bound for the graph distance between these vertices in the Kingman $n$-coalescent, respectively. Since the depth and the truncated depth are asymptotically equal under certain constraints on the truncation sequence $t_n$ (see Lemma~\ref{lemma:h2n} and Remark~\ref{rem:h2n}), and since $(\tau_k)_{n\in\N}$ forms a tight sequence of random variables by Lemma~\ref{lemma:tauk}, these bounds on the graph distance are sufficiently sharp. Using the largest common ancestor to provide a lower bound on the graph distance has been used by Munsonius and R\"uschendorf for $b$-are recursive trees~\cite{MunRus11} and by Ryvkina for random split trees~\cite{Ryv08}, previously.
	
	We formalise the above intuition in the remainder of the section, in which we prove Theorems~\ref{thrm:condconvking} and~\ref{thrm:kingfixlabel}.
	
	\begin{proof}[Proof of Theorem~\ref{thrm:condconvking}]
		We prove~\eqref{eq:withlab}. The proof of~\eqref{eq:withoutlab} uses an analogous approach with~\eqref{eq:multres} in Proposition~\ref{prop:condprobmult}, and hence the proof is omitted. 
		
		We set $t_n=\lceil (\log n)^2\rceil$, $\mathcal D_k:=\{\d_n(i)\geq d_i, i\in[k]\}$ and, for ease of writing, let $f_i:=\log n-d_i/2$ and recall that $d_i$ diverges with $n$ such that $a_i:=\lim_{n\to\infty} d_i/\log n$ exists for all $i\in[k]$. From Proposition~\ref{prop:condprobmult}, we obtain that the tuple
		\be 
		\Big(\frac{h_{n,1}(i)-f_i}{\sqrt{\log n-d_i/4}},\frac{\log \ell_n(i)-f_i}{\sqrt{d_i/4}}\Big)_{i\in[k]},
		\ee 
		conditionally on the event $\mathcal D_k$, converges in distribution to 
		\be 
		\Big(M_i\sqrt{\frac{a_i}{4-a_i}}+N_i\sqrt{1-\frac{a_i}{4-a_i}}, M_i\Big)_{i\in[k]},
		\ee 
		where we recall that $a_i:=\lim_{n\to\infty} d_i/\log n$. Moreover, by the choice of $t_n$ at the start of the proof, combining the above result with Lemma~\ref{lemma:h2n} and Remark~\ref{rem:h2n} yields the same result when substituting $h_n(i)$ for $h_{n,1}(i)$. What remains is to include the graph distance between the vertices $1,\ldots, k$ to prove Theorem~\ref{thrm:condconvking}. We use the trivial upper bound $\dist_n(i,j)\leq h_n(i)+h_n(j)$ $i,j\in[n]$ to obtain 
		\be \ba \label{eq:distub}
		\bigg({}&\Big(\frac{h_n(i)-f_i}{\sqrt{\log n-d_i/4}},\frac{\log \ell_n(i)-f_i}{\sqrt{d_i/4}}\Big)_{i\in[k]},\Big(\frac{\dist_n(i,j)-(f_i+f_j)}{\sqrt{2\log n-(d_i+d_j)/4}}\Big)_{1\leq i<j\leq k}\bigg)\\
		&\leq 	\bigg(\Big(\frac{h_n(i)-f_i}{\sqrt{\log n-d_i/4}},\frac{\log \ell_n(i)-f_i}{\sqrt{d_i/4}}\Big)_{i\in[k]},\Big(\frac{h_n(i)+h_n(j)-(f_i+f_j)}{\sqrt{2\log n-(d_i+d_j)/4}}\Big)_{1\leq i<j\leq k}\bigg),
		\ea\ee 
		where the inequality holds element-wise and almost surely. We now observe that 
		\be \ba \label{eq:sumsplit}
		\frac{h_n(i)+h_n(j)-(f_i+f_j)}{\sqrt{2\log n-(d_i+d_j)/4}}={}&\frac{h_n(i)-(f_i)}{\sqrt{\log n-d_i/4}}\sqrt{\frac{\log n-d_i/4}{2\log n-(d_i+d_j)/4}}\\
		&+\frac{h_n(j)-f_j}{\sqrt{\log n-d_j/4}}\sqrt{\frac{\log n-d_j/4}{2\log n-(d_i+d_j)/4}}.
		\ea \ee 
		Since $d_i/\log n\to a_i$, it follows that the two deterministic square root terms on the right-hand side converge to $\sqrt{(4-a_i)/(8-(a_i+a_j))}$ and $\sqrt{(4-a_j)/(8-(a_i+a_j))}$, respectively. Furthermore, by the joint convergence of the depth and label of vertices $1,\ldots, k$, conditionally on $\mathcal D_k$, it thus follows from the continuous mapping theorem~\cite{Bill99} and Slutsky's theorem~\cite[Lemma $2.8$]{Vaart00}, that 
		\be \ba
		\bigg({}&\Big(\frac{h_n(i)-f_i}{\sqrt{\log n-d_i/4}},\frac{\log \ell_n(i)-f_i}{\sqrt{d_i/4}}\Big)_{i\in[k]},\Big(\frac{h_n(i)+h_n(j)-(f_i+f_j)}{\sqrt{2\log n-(d_i+d_j)/4}}\Big)_{1\leq i<j\leq k}\bigg)\\
		&\toindis \bigg(\Big(M_i\sqrt{\frac{a_i}{4-a_i}}+N_i\sqrt{1-\frac{a_i}{4-a_i}}, M_i\Big)_{i\in[k]},\\
		&\hphantom{\toindis \bigg(\ } \Big(\frac{M_i\sqrt{a_i}+N_i\sqrt{4-2a_i}+M_j\sqrt{a_j}+N_j\sqrt{4-2a_j}}{\sqrt{8-(a_i+a_j)}}\Big)_{1\leq i<j\leq k}\bigg).
		\ea \ee 
		Combined with~\eqref{eq:distub}, and letting, for $(x_i,y_i)_{i\in[k]}\in (\R^2)^k$, and $(z_{i,j})_{1\leq i<j\leq k}\in \R^{k(k-1)/2}$ fixed,
		\be \ba \label{eq:terms}
		h_i&:=(\log n-d_i/2)+x_i\sqrt{\log n-d_i/4}, \quad \wt \ell_i:=(\log n-d_i/2)+y_i\sqrt{d_i/4},\\ 
		L_{i,j}&:=(2\log n-(d_i+d_j)/2)+z_{i,j}\sqrt{2\log n-(d_i+d_j)/4}, 
		\ea\ee 
		this yields,
		\be\ba  \label{eq:liminf2}
		\liminf_{n\to\infty}\mathbb P ({}&h_n(i)\leq h_i,\log \ell_n(i)\leq \wt\ell_i,i\in[k], \dist_n(i,j)\leq L_{i,j}, 1\leq i<j\leq k\,|\,\mathcal D_k)\\
		\geq \mathbb P \bigg({}&M_i\sqrt{\frac{a_i}{4-a_i}}+N_i\sqrt{1-\frac{a_i}{4-a_i}}\leq x_i, M_i\leq y_i, i\in[k],\\
		&\frac{M_i\sqrt{a_i}+N_i\sqrt{4-2a_i}+M_j\sqrt{a_j}+N_j\sqrt{4-2a_j}}{\sqrt{8-(a_i+a_j)}}\leq z_{i,j}, 1\leq i<j\leq k \bigg).
		\ea \ee 
		It remains to obtain a matching lower bound. We make use of the following observation: In the Kingman $n$-coalescent process, assume two vertices $i_1,i_2$ are in distinct trees at step $j$ of the coalescent. Then, the sum of their depths at step $j$ is bounded from above by the graph distance between $i_1$ and $i_2$ in the final tree of the coalescent. That is, $h_{F_j}(i_1)+h_{F_j}(i_2)\leq \dist_{F_1}(i_1,i_2)$ on the event that $i_1,i_2$ are in two distinct trees in the forest $F_j$. See Figure~\ref{fig:king} for an example, where the graph distance between vertices $1$ and $3$ in $F_1$ is larger than the sum of the depths of $1$ and $3$ in $F_2$.
		
		This observation allows us to use the truncated depths $h_{n,1}(i)$ to bound the graph distances between the vertices $1,\ldots, k$. Indeed, $h_{n,1}(i)=h_{F_{t_n}}(i)$ denotes the depth of vertex $i$ in the tree at the truncation time $t_n$. Recall that the event $\{\tau_k<t_n\}$ denotes that the vertices $1,\ldots, k$ are in distinct trees at step $t_n$, which holds with high probability by Lemma~\ref{lemma:tauk}. For $h_i,\wt\ell_i,L_{i,j}$ as in~\eqref{eq:terms}, we thus have
		\be\ba 
		\mathbb P (h_n{}&(i)\leq h_i,\log \ell_n(i)\leq \wt \ell_i,i\in[k],\dist_n(i,j)\leq L_{i,j}, 1\leq i<j\leq k\,|\, \mathcal D_k)\\
		\leq \mathbb P({}&h_{n,1}(i)\leq h_i,\log \ell_n(i)\leq \wt \ell_i,i\in[k],h_{n,1}(i)+h_{n,1}(j)\leq L_{i,j}, 1\leq i<j\leq k,\tau_k<t_n\,|\, \mathcal D_k)\\
		&+\mathbb P(\tau_k\geq t_n\,|\, \mathcal D_k)\\
		\leq \mathbb P({}&h_{n,1}(i)\leq h_i,\log \ell_n(i)\leq \wt \ell_i,i\in[k],h_{n,1}(i)+h_{n,1}(j)\leq L_{i,j}, 1\leq i<j\leq k\,|\,\mathcal D_k)\\
		&+\mathbb P(\tau_k\geq t_n\,|\,\mathcal D_k).
		\ea\ee 
		The last term tends to zero with $n$ by Lemma~\ref{lemma:tauk}. With the same approach as in~\eqref{eq:sumsplit} and~\eqref{eq:liminf2}, we thus obtain
		\be\ba 
		\limsup_{n\to\infty}\mathbb P ({}&h_n(i)\leq h_i,\log \ell_n(i)\leq \wt \ell_i,i\in[k],\dist_n(i,j)\leq L_{i,j}, 1\leq i<j\leq k\,|\, \CD_k)\\
		\leq \mathbb P\bigg({}&M_i\sqrt{\frac{a_i}{4-a_i}}+N_i\sqrt{1-\frac{a_i}{4-a_i}}\leq x_i, M_i\leq y_i,i\in[k],\\
		& \frac{M_i\sqrt{a_i}+N_i\sqrt{4-2a_i}+M_j\sqrt{a_j}+N_j\sqrt{4-2a_j}}{\sqrt{8-(a_i+a_j)}}\leq z_{i,j}, 1\leq i<j\leq k\bigg).
		\ea\ee 
		Combined with the matching lower bound which follows from~\eqref{eq:liminf2}, this concludes the proof.
	\end{proof}
	
	In a similar spirit, we prove Theorem~\ref{thrm:kingfixlabel}. Again, combined with Corollary~\ref{cor:unifnodes}, this implies Theorem~\ref{thrm:fixedlabel}.
	
	\begin{proof}[Proof of Theorem~\ref{thrm:kingfixlabel}]
		The proof follows a similar approach to the proof of Theorem~\ref{thrm:condconvking}. Recall the random variables $(d^*_n(i))_{i\in[k]}$ and $(Z_i)_{i\in[k]}$ from~\eqref{eq:dstar} and set $t_n=\min_{i\in[k]}\log \ell_i$. Proposition~\ref{prop:limitfixell} provides that the tuple
		\be \label{eq:jointfixell}
		\Big(d^*_n(i), \frac{h_{n,1}(i)-\log \ell_i}{\sqrt{\log \ell_i}}\Big)_{i\in[k]},
		\ee 
		conditionally on the event $\CL_k:=\{\ell_n(i)=\ell_i,i\in[k]\}$, converges in distribution to $(Z_i,N_i)_{i\in[k]}$, where the $N_i$ are i.i.d.\ standard normal random variables, also independent of the $Z_i$. By our choice of $t_n$, Lemma~\ref{lemma:h2n} and Remark~\ref{rem:h2n} yield that the result holds when $h_{n,1}(i)$ is substituted by $h_n(i)$ as well. As in~\eqref{eq:distub}, we can use the trivial upper bound $\dist_n(i,j)\leq h_n(i)+h_n(i), i,j\in[n]$. We can thus write, similar to~\eqref{eq:sumsplit}, 
		\be \ba \label{eq:distubfixell}
		\frac{\dist_n(i,j)-(\log \ell_i+\log \ell_j)}{\sqrt{\log \ell_i+\log\ell_j}}\leq{}& \frac{h_n(i)-\log \ell_i}{\sqrt{\ell_i}}\sqrt{\frac{\log \ell_i}{\log\ell_i+\log\ell_j}}\\
		&+\frac{h_n(j)-\log \ell_j}{\sqrt{\log \ell_j}}\sqrt{\frac{\log\ell_j}{\log\ell_i+\log\ell_j}}.
		\ea \ee 
		Define, for $(y_i)_{i\in[k]}\in \R^k, (z_{i,j})_{1\leq i<j\leq k}\in \R^{k(k-1)/2}$ fixed, 
		\be \label{eq:termsfixell}
		h_i:=\log\ell_i+y_i\sqrt{\log\ell_i}, \qquad L_{i,j}:=(\log\ell_i+\log\ell_j)+z_{i,j}\sqrt{\log\ell_i+\log\ell_j}, \quad 1\leq i<j\leq k.
		\ee 
		Recall the limits $c_{i,j},c_{j,i}$ of the two square-root terms on the right-hand side of~\eqref{eq:distubfixell} from~\eqref{eq:cij}. We thus obtain, for $(x_i)_{i\in[k]}\in \R^k$ fixed, by~\eqref{eq:distubfixell} and~\eqref{eq:jointfixell} (and the remark on the $h_n(i)$ below~\eqref{eq:jointfixell}) together with the continuous mapping theorem~\cite{Bill99},
		\be \ba \label{eq:liminf}
		\liminf_{n\to\infty}{}&\P{d^*_n(i)\leq x_i,h_n(i)\leq h_i,i\in[k], \dist_n(i,j)\leq L_{i,j}, 1\leq i<j\leq k\,|\,\CL_k}\\
		&\geq \P{Z_i\leq x_i, N_i\leq y_i,i\in[k], c_{i,j}N_i+c_{j,i}N_j\leq z_{i,j},1\leq i<j\leq k}.
		\ea \ee 
		We now use the same observation made after~\eqref{eq:liminf2}. That is, on the event $\{\tau_k<t_n\}$, it holds that $\dist_n(i,j)\geq h_{n,1}(i)+h_{n,1}(j)$ for any two distinct vertices $i,j\in[k]$. We hence have 
		\be \ba 
		\mathbb P({}&d^*_n(i)\leq x_i,h_n(i)\leq h_i,i\in[k], \dist_n(i,j)\leq L_{i,j}, 1\leq i<j\leq k\,|\,\CL_k)\\
		\leq{}&\mathbb P(d^*_n(i)\leq x_i,h_n(i)\leq h_i,i\in[k], h_{n,1}(i)+h_{n,1}(j)\leq L_{i,j}, 1\leq i<j\leq k,\tau_k<t_n\,|\,\CL_k)\\
		&+\P{\tau_k\geq t_n\,|\, \CL_k}\\
		\leq{}&\mathbb P(d^*_n(i)\leq x_i,h_n(i)\leq h_i,i\in[k], h_{n,1}(i)+h_{n,1}(j)\leq L_{i,j}, 1\leq i<j\leq k\,|\,\CL_k)\\
		&+\P{\tau_k\geq t_n\,|\, \CL_k}.
		\ea\ee 
		The last term on the right-hand side tends to zero by Lemma~\ref{lemma:labeldegprob}. Using the right-hand side of~\eqref{eq:distubfixell} to rewrite the event $\{h_{n,1}(i)+h_{n,1}(j)\leq L_{i,j}\}$, we thus obtain 
		\be\ba 
		\limsup_{n\to\infty}{}&\mathbb P(d^*_n(i)\leq x_i,h_n(i)\leq h_i,i\in[k], \dist_n(i,j)\leq L_{i,j}, 1\leq i<j\leq k\,|\,\CL_k)\\
		\leq{}& \P{Z_i\leq x_i, N_i\leq y_i,i\in[k], c_{i,j}N_i+c_{j,i}N_j\leq z_{i,j},1\leq i<j\leq k},
		\ea\ee 
		which matches the lower bound in~\eqref{eq:liminf} and concludes the proof.
	\end{proof} 	
	
	\noindent \textbf{Acknowledgements}\\
	Bas Lodewijks has been supported by grant GrHyDy ANR-20-CE40-0002, and would like to thank Laura Eslava for some useful discussions related to the Kingman $n$-coalescent and for providing the source code of the figures in this paper.
	
	He would also like to thank the anonymous referees for providing helpful suggestions which led to an improved presentation of the results and proofs.
	
	\bibliographystyle{abbrv}
	\bibliography{wrtbib}

\end{document}